%% file: stability_conv_rates.tex
\documentclass[a4paper, 10pt]{article}

\usepackage{packageList}
\usepackage{macros}

\usepackage{cleveref}

\usepackage[makeroom]{cancel}			%

\newcommand{\dist}{\mathrm{dist}}

\newcommand{\F}{\mathcal{F}}
\newcommand{\Ha}{\mathcal{H}}
\newcommand{\Ft}{\tilde{\mathcal{F}}}

\newcommand{\nsa}{\mathcal{H}_{\tilde{k}} (\tilde{\Omega})}
\newcommand{\nsb}{\mathcal H_k (\Omega)}
\newcommand{\Omegat}{\tilde{\Omega}}

\usepackage[round]{natbib}

\graphicspath{{./Figures/}}
\pgfplotsset{compat=1.11}
\newlength\fwidth

\newif\ifappendix
\appendixfalse

\title{Stability of convergence rates: \\ Kernel interpolation on non-Lipschitz domains}
\author[1]{Tizian Wenzel \thanks{tizian.wenzel@mathematik.uni-stuttgart.de, corresponding author}}
\author[2]{Gabriele Santin \thanks{gsantin@fbk.eu, \href{http://orcid.org/0000-0001-6959-1070}{orcid.org/0000-0001-6959-1070}}}
\author[1]{Bernard Haasdonk \thanks{haasdonk@mathematik.uni-stuttgart.de}}
\affil[1]{Institute for Applied Analysis and Numerical Simulation, University of Stuttgart, Germany}
\affil[2]{Digital Society Center, Bruno Kessler Foundation, Trento, Italy}

\begin{document}

\maketitle
  
\begin{abstract}
Error estimates for kernel interpolation in Reproducing Kernel Hilbert Spaces (RKHS) usually assume quite restrictive properties on the shape of the domain, 
especially in the case of infinitely smooth kernels like the popular Gaussian kernel. 

In this paper we prove that it is possible to obtain convergence results (in the number of interpolation 
points) for kernel interpolation for arbitrary domains $\Omega \subset \R^d$, thus allowing for non-Lipschitz domains including e.g.\ cusps and irregular boundaries. %
Especially we 
show that, when going to a smaller domain $\Omegat \subset \Omega \subset \R^d$, the convergence rate does not deteriorate --- i.e.\ the convergence rates are stable with respect to going to a subset.
We obtain this by leveraging an analysis of greedy kernel algorithms.

The impact of this result is explained on the examples of kernels of finite as well as infinite smoothness. A comparison to 
approximation in Sobolev spaces is drawn, where the shape of the domain $\Omega$ has an impact on the approximation properties.
Numerical experiments illustrate and confirm the analysis. %
\end{abstract}

\section{Introduction} \label{sec:introduction}

On a nonempty set $\Omega \subset \R^d$ a real-valued \textit{kernel} is defined as a symmetric function $k: \Omega \times \Omega \rightarrow \R$. For a given set of 
points $X_n := \{x_1, ..., x_n \} \subset \Omega$ the kernel matrix $A \in \R^{n \times n}$ is defined as $A_{ij} := k(x_i, x_j)$. If the kernel matrix is 
positive definite for any choice of pairwise distinct points $X_n \subset \Omega, n \in \N$, then the kernel is called \textit{strictly positive definite}. 

In the following we focus on this class of kernels. For those kernels, there is a unique \textit{Reproducing Kernel Hilbert Space} (RKHS) $\ns$ of functions, which satisfies
\begin{enumerate}
\item $k(\cdot, x) \in \ns \quad \forall x \in \Omega$,
\item $f(x) = \langle f , k(\cdot, x) \rangle_{\ns} \quad \forall x \in \Omega, f \in \ns$.
\end{enumerate}

We remark that for certain types of kernels and domains the RKHS $\ns$ is found to be norm equivalent to some Sobolev space with specific smoothness $\tau > 
d/2$, i.e.\ $\ns \asymp H^\tau(\Omega)$. Examples are some radial basis function kernels like the Mat{\'e}rn kernels (see Eq.\ \eqref{eq:matern_kernels} below) 
on Lipschitz domains $\Omega$.

For a function $f \in \ns$ and pairwise distinct points $X_n \subset \Omega$ there is a unique minimum-norm interpolating function, which is given by the 
projection of $f$ onto the subspace $V(X_n) := \{ k(\cdot, x_i), x_i \in X_n \}$ spanned by kernel functions. 
It holds
\begin{align}
\label{eq:interpolation_conditions}
\Pi_{\ns, V(X_n)}(f)(x_i) = f(x_i) \quad \forall x_i \in X_n,
\end{align}
where $\Pi_{\ns, V(X_n)}$ denotes the orthogonal projector from $\ns$ onto the closed subspace $V(X_n)$.

In order to quantify the worst-case (i.e.\ for any function $f \in \ns$) pointwise error between a function and its interpolant, the so called \textit{power function} $P_{k, \Omega, X_n}$ is introduced, 
which is defined as
\begin{align}
\label{eq:definition_power_func}
P_{k, \Omega, X_n}(x) :=& ~ \sup_{0 \neq f \in \nsb} \frac{|(f - \Pi_{\nsb, V(X_N)}(f))(x)|}{\Vert f \Vert_{\nsb}} \notag \\
=& \Vert k(\cdot, x) - \Pi_{\nsb, V(X_N)}(k(\cdot, x)) \Vert_{\nsb}. 
\end{align}
Note that the dependency of the power function on the kernel $k$ and the domain $\Omega$ was explicitly included in the subscript of the power function for later 
purposes. 

Typical results for bounding the maximal worst case interpolation error $\Vert P_{k, \Omega, X_n} \Vert_{L^\infty(\Omega)}$ in the number $n$ of interpolation 
points $X_n$ make use of the equivalence of the RKHS $\ns$ to Sobolev spaces $H^\tau(\Omega)$ as mentioned earlier. These equivalence results however 
require restrictive assumptions on the domain $\Omega$ such as Lipschitz boundaries or cone conditions. On the other hand, these conditions are not required for 
having an RKHS $\ns$.
In this paper we close this gap by showing that the convergence rates for interpolation in RKHS are stable when going to a smaller domain --- in the sense that the convergence order (in the number of interpolation points) does not deteriorate.
In this way we obtain convergence rates also for non-Lipschitz domains.

The paper is organized as follows. After recalling background information on kernel based interpolation in Section \ref{sec:background}, we review and extend 
an analysis of greedy algorithms in an abstract setting in Section \ref{sec:abstract_setting}. 
This analysis of greedy algorithms is leveraged as a tool and applied to general (non-greedy) kernel based interpolation in Section \ref{sec:kernel_interpolation}, where our main result is derived in Theorem \ref{th:main_result} and discussed and applied afterwards. 
Section \ref{sec:num_experiments} illustrates the theoretical findings with numerical experiments. 
The final Section \ref{sec:conclusion} summarizes and concludes the paper by giving an outlook to future research.

\section{Background} \label{sec:background}

We proceed by recalling background information 
on greedy kernel 
interpolation \citep{marchi2005optimal, santin2017convergence, wenzel2022analysis} and on the dependency of the RKHS $\ns$ on the domain $\Omega$ \citep{wendland2005scattered}, 
which is required to prove our main result. 

\subsection{Greedy kernel interpolation}

A suitable choice of interpolation points $X_n \subset \Omega$ is usually unclear a priori. For this task, a popular approach is to use greedy kernel methods 
to 
pick the interpolation points step by step. One starts with an empty set $X_0 := \emptyset$ and adds points incrementally via $X_{n+1} := X_n \cup \{ x_{n+1} 
\}$, where $x_{n+1}$ is selected according to some selection criterion. For later purposes we recall the so-called $P$-greedy selection criterion, which uses
\begin{align}
\label{eq:p_greedy_criterion}
x_{n+1} := \argmax_{x \in \Omega} P_{k, \Omega, X_n}(x).
\end{align}
Since it holds $P_{k, \Omega, X_n}(x_i) = 0$ if and only if $x_i \in X_n$, the $P$-greedy selection criterion yields an infinite sequence of pairwise distinct points 
(given that $\Omega$ consists of infinitely many points). The convergence of $\Vert P_{k, \Omega, X_n} \Vert_{L^\infty(\Omega)}$ for $n \rightarrow \infty$ was 
analyzed in \citep{marchi2005optimal,santin2017convergence,wenzel2021novel}.

We remark that there are more greedy methods available, especially the target-data dependent $f$-greedy or $f/P$-greedy \citep{mueller2009komplexitaet, SchWen2000}, but also stabilization of those algorithms as analyzed in \citep{wenzel2021novel}. An analysis of a unifying scale of target-data dependent algorithms including the aforementioned ones can be found in \citep{wenzel2022analysis}.

\subsection{Restrictions and extensions of domains}
\label{subsec:restriction_extension_domain}

This section gives some background information on the dependency of the RKHS $\nsb$ on the domain $\Omega$, and we refer to \citep[Section 10.7]{wendland2005scattered} 
for further details. Especially, we will state our notation more precisely. 

We consider two nested domains 
\begin{align}
\label{eq:nested_domains}
\tilde{\Omega} \subset \Omega \subset \R^d.
\end{align}
Let $k := k_\Omega: \Omega \times \Omega \rightarrow \R$ be a strictly positive definite kernel defined on $\Omega$. Then we define the \textit{restricted kernel} $\tilde{k}$ via
\begin{align}
\label{eq:def_restricted_kernel}
\tilde{k} := k_{\tilde{\Omega}}: \Omegat \times \Omegat \rightarrow \R, (x, y) \mapsto k(x, y).
\end{align}
With this notation we can distinguish between $k(\cdot, x) \in \nsb$ and $\tilde{k}(\cdot, x) \in \nsa$ 
for $x \in 
\tilde{\Omega}$ as elements of different spaces, while both $k$ and $\tilde{k}$ refer to the same mapping. 

Theorem 10.46 in \citep{wendland2005scattered} states that each function $\tilde{f} \in \nsa$ has a natural extension to a function $E\tilde{f} \in \nsb$ and it holds $\Vert E\tilde{f} \Vert_{\nsb} = \Vert \tilde{f} \Vert_{\nsa}$. 
The operator $E: \nsa \rightarrow \nsb$ will be called the extension operator.
Especially it holds $E\tilde{k}(\cdot, x) = k(\cdot, x)$ for all $x \in \Omegat$.

For $X_n \subset \Omegat$ we define
\begin{align*}
\tilde{V}(X_n) := \Sp \{ \tilde{k}(\cdot, x_i), x_i \in X_n \} \subset \nsa, \\
V(X_n) := \Sp \{ k(\cdot, x_i), x_i \in X_n \} \subset \nsb,
\end{align*}
where we recalled the same definition of $V(X_n)$ given above.

We have the following result, which states that these projections and extensions commute.

\begin{lemma}
\label{lem:utility_lemma}
For $\tilde{f} \in \nsa$ and $X_n \subset \Omegat$ it holds 
\begin{align*}
E\Pi_{\nsa, \tilde{V}(X_n)}(\tilde{f}) = \Pi_{\nsb, V(X_n)}(E \tilde{f}).
\end{align*}
\end{lemma}
\begin{proof}
Let 
\begin{align*}
\Pi_{\nsa, \tilde{V}(X_n)}(\tilde{f}) = \sum_{j=1}^n \alpha_j^{(n)} \tilde{k}(\cdot, x_j),
\end{align*}
with the kernel expansion coefficients $\{ \alpha_j^{(n)} \}_{j=1}^n \subset \R$ being determined by the interpolation conditions (see Eq.~\eqref{eq:interpolation_conditions}), 
i.e.\ $\tilde{f}(x_i) = \sum_{j=1}^n \alpha_j^{(n)} \tilde{k}(x_i, x_j) $ for $i=1, \dots, n$. 
Based on the linearity of the extension operator $E$ and the identity $E\tilde{k}(\cdot, x) = k(\cdot, x)$ it holds
\begin{align*}
E\Pi_{\nsa, \tilde{V}(X_n)}(\tilde{f}) &= E \sum_{j=1}^n \alpha_j^{(n)} \tilde{k}(\cdot, x_j) 
= \sum_{j=1}^n \alpha_j^{(n)} E\tilde{k}(\cdot, x_j) \\
&= \sum_{j=1}^n \alpha_j^{(n)} k(\cdot, x_j) 
= \Pi_{\nsb, V(X_n)}(E\tilde{f}),
\end{align*}
where in the last step we used $E\tilde{f}|_{\tilde{\Omega}} = \tilde{f}$ and $X_n \subset \tilde{\Omega}$, 
so that the interpolation coefficients $\{ \alpha_j^{(n)} \}_{j=1}^n$ are the same both in $\nsa$ and $\nsb$.
\end{proof}

The following lemma shows that the power functions $P_{k, \Omega, X_n}$ and $P_{\tilde{k}, \Omegat, X_n}$ coincide on $\Omegat$, if $X_n \subset \Omegat$:

\begin{lemma}
\label{lem:power_funcs_on_diff_domains}
For $X_n \subset \Omegat$ it holds
\begin{align*}
P_{k, \Omega, X_n}(x) = P_{\tilde{k}, \Omegat, X_n}(x) \qquad \forall x \in \Omegat.
\end{align*}
\end{lemma}

\begin{proof}
From the definition of the power function in Eq.~\eqref{eq:definition_power_func},  by using the fact that $E\tilde{k}(\cdot, x) = k(\cdot, x)$ for all $x \in 
\Omegat$, and by Lemma \ref{lem:utility_lemma} we can calculate
\begin{align*}
P_{k, \Omega, X_n}(x) =& \Vert k(\cdot, x) - \Pi_{\ns, V(X_n)}(k(\cdot, x)) \Vert_{\nsb} \\
\equiv& \Vert E\tilde{k}(\cdot, x) - \Pi_{\nsb, V(X_n)}(E\tilde{k}(\cdot, x)) \Vert_{\nsb} \\
=& \Vert E\tilde{k}(\cdot, x) - E\Pi_{\nsa, \tilde{V}(X_n)}(\tilde{k}(\cdot, x)) \Vert_{\nsb} \\
=& \Vert \tilde{k}(\cdot, x) - \Pi_{\nsa, \tilde{V}(X_n)}(\tilde{k}(\cdot, x)) \Vert_{\nsa} \\
\equiv& P_{\tilde{k}, \Omegat, X_n}(x),
\end{align*}
which holds for all $x \in \Omegat$.
\end{proof}

We remark that Lemma \ref{lem:utility_lemma} and Lemma \ref{lem:power_funcs_on_diff_domains} can be formulated for general closed subspaces $\tilde{V}_n \subset \nsa$,
$E\tilde{V}_n =: V_n \subset \nsb$ instead of the particular kernel-based subspaces $\tilde{V}(X_n), V(X_n)$ considered here. 

\subsection{Related work}

From a high level point of view, 
our main result (in \Cref{th:main_result} below) allows us to derive (worst-case) error estimates (in the number of interpolation points) for kernel-based interpolation for arbitrary (infinite) sets $\Omegat \subset \R^d$, e.g.\ non-Lipschitz domains. 
This principle holds for general continuous kernels $k$ with $\sup_{x \in \Omega} k(x, x) < \infty$.
Both the error estimates for non-Lipschitz domains as well as for general continuous bounded kernels was not available so far.
Thus here we briefly comment on related work on kernel-based interpolation, 
with a particular focus on the dependency on the underlying domain.

First we note that most worst case error estimates are formulated in terms of the fill distance, which is usually defined as
\begin{align*}
h_{X_n, \Omega} := \sup_{x \in \Omega} \min_{x_j \in X_n} \Vert x - x_j \Vert_2.
\end{align*}
For bounded domains and uniformly distributed points there exist constants $c, C > 0$ such that
\begin{align}
\label{eq:coupling_fill_distance}
c n^{-1/d} \leq h_{X_n, \Omega} \leq C n^{-1/d}.
\end{align}
Using this relationship between the number of points $n$ and the fill distance $h_{X_n, \Omega}$, 
any error estimate expressed in the fill distance can be turned into an error estimate based on the number of points and vice versa.

A first line of research works via the power function of Eq.~\eqref{eq:definition_power_func},
and \citep[Section 4]{schaback1995multivariate} provides a good overview of the used techniques:
\citep{madych1988multivariate, madych1992bounds} consider kernels $k(x, y) := \Phi(\Vert x - y \Vert_2)$ (radial basis function kernels) and derive error estimates by investigating the polynomial expansion of the radial basis function $\Phi$ around zero.
Those references assume $\Omega$ to have a ``sufficiently smooth boundary''.
The article \citep{madych1990multivariate} (for infinitely smooth kernels) as well as \citep{wu1993local} (for finitely smooth kernels) consider kernels $k(x, y) := \Phi(x - y)$ (translational invariant kernels) and derive error estimates with Fourier transform techniques, 
as this is the standard tool for translational invariant problems.
The paper \citep{madych1990multivariate} assumes $\Omega$ to satisfy a cone condition, 
while \citep{wu1993local} does not state to require a cone condition.
They work with a local fill distance, which is defined as
\begin{align*}
h_{\rho, X_n}(x) := \max_{y \in B_\rho(x)} \min_{x_j \in X_n} \Vert y - x_j \Vert_2,
\end{align*}
(using $B_\rho(x) = \{ y \in \R^d ~ | ~ \Vert x - y \Vert_2 < \rho \}$) and is assumed to be locally small enough, i.e.\ ``$h_{\rho, X_n} \leq h_0$''.
However, a close inspection of the interplay of $h_{\rho, X_n}, h_0$ and $\rho$ (see \citep[Eq.~(5.5)]{wu1993local}) reveals that plenty of interpolation points $X_n$ are required in the neighborhood $B_\rho(x)$, 
thus requiring many interpolation points in possible cusps of the domain $\Omega$. 
This then breaks the coupling of the local fill distance to the number of interpolation points of Eq.~\eqref{eq:coupling_fill_distance}.
Therefore, if one wants to obtain convergence rates in terms of interpolation points with the results from \citep{wu1993local}, a cone condition is required.

This is one on the contributions of our work over previous works \citep{madych1988multivariate, madych1992bounds, madych1990multivariate, wu1993local}: 
We obtain convergence rates in terms of the number of points for domains without such cone conditions and also for more general kernels than radial or translational invariant kernels.
This is achieved by a considerably new proof technique.

A second line of research deals with sampling inequalities \citep{narcowich2005sobolev, wendland2005approximate, narcowich2006sobolev}.
In contrast to the power function analysis before, this approach also allows for approximation statements on functions outside the RKHS of the used kernel.
By applying the sampling inequalites to the residual $f - s_n$ and leveraging norm equivalences, namely between the RKHS and corresponding Sobolev spaces,
it is possible to derive error estimates (in terms of the fill distance) for kernel-based interpolation.
In all of those references, the underlying assumption is that $\Omega$ is a compact domain with Lipschitz boundary, 
which is required e.g.\ to ensure a continuous extension operator to $\R^d$.

Some more details on convergence rates for kernel interpolation using the Gaussian kernel are discussed in \Cref{subsubsec:stability_wrt_domain}.

Both the power function analysis as well as the sampling inequality approach are quite constructive.
This is in contrast to our approach, which mostly makes use of abstract approximation theory and leverages greedy kernel analysis to apply this to possibly non-greedy kernel-based approximation.

\section{Analysis of greedy algorithms in an abstract setting} \label{sec:abstract_setting}

We start in \Cref{subsec:abstract_sett_rev} by recalling the abstract analysis of greedy algorithms in Hilbert spaces of \citep{binev2011convergence}, which was later refined in \citep{devore2013greedy}.
This review will serve as background for the refined analysis in \Cref{subsec:abstract_sett_gen}, 
where we both lift unneccessary restrictive assumptions and extend the results.

\subsection{Review of abstract setting} \label{subsec:abstract_sett_rev}

We start by reviewing the framework from \citep{binev2011convergence, devore2013greedy}:
For this let $\Ha$ be a Hilbert space with norm $\Vert \cdot \Vert = \Vert \cdot \Vert_{\Ha}$.
Furthermore let $\F \subset \mathcal{H}$ be a compact (and infinite) subset and assume, for notational convenience only,
that it holds $\Vert f \Vert_{\Ha} \leq 1$ for all $f \in \F$. 
Now a \textit{weak greedy algorithm with constant $\gamma \in (0, 1]$} is considered that selects a sequence $\{ f_0, f_1, ... \} \subset \F$ of elements.
It starts by picking any element $f_0 \in \F$ that satisfies
\begin{align*}
\Vert f_0 \Vert \geq \gamma \cdot \max_{f \in \F} \Vert f \Vert,
\end{align*}
and proceeds by defining $V_n := \Sp \{f_0, ..., f_{n-1} \}$ and selecting the next elements according to 
\begin{align}
\label{eq:greedy_criterion_abstract}
\dist(f_{n}, V_n)_{\Ha} \geq \gamma \cdot \sup_{f \in \F} \dist(f, V_n)_{\Ha},
\end{align}
The resulting sequence $\{ f_0, f_1, ... \} \subset \F$ will not be unique, but the subsequent analysis holds for any sequence that is selected according to this weak selection criterion.

We recall that two important quantities $d_n(\F)_\mathcal{H}$ and $\sigma_n(\F)_\mathcal{H}$ considered in the analysis, 
that are defined as
\begin{align} 
\begin{aligned} \label{eq:quantities}
d_n(\F)_\mathcal{H} &:= \inf_{\substack{Y_n \subset \mathcal{H}\\ \dim(Y_n) = n}} \sup_{f \in \F} \mathrm{dist}(f, Y_n)_\mathcal{H}, \\
\sigma_n(\F)_\mathcal{H} &:= \sup_{f \in \F} \mathrm{dist}(f, V_n)_\mathcal{H}, %
\end{aligned}
\end{align}
where $d_n(\F)_\mathcal{H}$ is the Kolmogorov $n$-width of $\F$ in $\mathcal{H}$.
For the proof of the main result of \citep{devore2013greedy}, a lemma is required which is recalled here:

\begin{lemma}\citep[Lemma 2.1]{devore2013greedy} \label{lem:abstract_lemma}
Let $G = (g_{i,j})_{i,j=1}^K$ be a $K \times K$ lower triangular matrix with rows $\textbf{g}_1, ..., \textbf{g}_K$, 
$W$ be any $m$-dimensional subspace of $\R^K$, and $P$ be the orthogonal projection of $\R^K$ onto $W$.
Then 
\begin{align*}
\prod_{i=1}^K g_{i,i}^2 \leq \left \{ \frac{1}{m} \sum_{i=1}^K \Vert P\textbf{g}_i \Vert_2^2  \right \}^m \left \{ \frac{1}{K-m} \sum_{i=1}^K \Vert \textbf{g}_i - P\textbf{g}_i \Vert_2^2 \right \}^{K-m}
\end{align*}
where $\Vert \cdot \Vert_2$ is the Euclidean norm of a vector in $\R^K$.
\end{lemma}

Then, the main statement \citep[Theorem 3.2]{devore2013greedy} puts the quantities $\sigma_n(\F)_{\Ha}$ and $d_n(\F)_{\Ha}$ from Eq.~\eqref{eq:quantities} into relation:

\begin{theorem}\citep[Theorem 3.2]{devore2013greedy} \label{th:abstract_estimate}
For the weak greedy algorithm with constant $\gamma$ in a Hilbert space $\Ha$ and for any compact set $\F$, we have the following inequalities between $\sigma_n := \sigma_n(\F)_{\Ha}$ and $d_n := d_n(\F)_{\Ha}$, for any $N \geq 0, K \geq 1$, and $1 \leq m < K$:
\begin{align*}
\prod_{i=1}^K \sigma_{N+i}^2 \leq \gamma^{-2K} \left( \frac{K}{m} \right)^m \left( \frac{K}{K-m} \right)^{K-m} \sigma_{N+1}^{2m} d_m^{2K-2m}.
\end{align*}
\end{theorem}

\subsection{Generalization of abstract setting} \label{subsec:abstract_sett_gen}

This subsection generalizes and extends the results recalled in \Cref{subsec:abstract_sett_rev},
especially \Cref{th:abstract_estimate}.
In particular:
\begin{enumerate}
\item We lift the unnecessarily restrictive assumption on $\F \subset \Ha$ being compact and replace it by $\F$ being precompact.
\item We extend the result by introducing an additional subset $\Ft \subset \F$ and analyzing its effect.
\end{enumerate}

Both these points are adressed in our main abstract statement in \Cref{th:abstract_estimate_2} below.

In order to understand why $\F$ being precompact is sufficient, we recall that for the Kolmogorov $n$-widths of a set $\F \subset \Ha$ it holds
\begin{align*}
d_n(\F)_{\Ha} = d_n(\overline{\F})_{\Ha},
\end{align*}
i.e.\ the Kolmogorov widths for a set and for its closure coincide \citep[Theorem 1.1 (page 10)]{pinkus2012n}.
This can also be seen easily from the definition of the Kolmogorov $n$-width in Eq.~\eqref{eq:quantities}, as it already incorporates the supremum $\sup_{f \in \F}$. 
Moreover, the weak selection criteria, i.e. $\gamma \in (0, 1)$, do not require that the supremum $\sup_{f \in \F}$ actually needs to be attained within $\F$.
Beyond these points, the compactness (in conjunction with its boundedness) of $\F$ is only used to have $d_n(\F)_{\Ha} \rightarrow 0$ (see \citep[Prop. 1.2 (page 10)]{pinkus2012n}), 
for which, however, the precompactness of $\F$ is sufficient. 

Regarding the second point, we additionally consider another set $\Ft \subset \F$, i.e.
\begin{align*}
\Ft \subset \F \subset \Ha.
\end{align*}
Now we consider running a greedy algorithm that picks subsequently elements $f_0, f_1, \dots$ from $\Ft$ according to 
\begin{align}
\label{eq:selection_criterion_subset}
f_n \in \Ft ~~ \text{such that} ~~ \dist(f_{n}, V_n)_{\Ha} \geq \gamma \cdot \sup_{f \in \Ft} \dist(f, V_n)_{\Ha}.
\end{align}
It is crucial to point out that the elements are selected from $\Ft \subset \F$ according to the largest error (measured as distance in $\Ha$) which is attained for the elements in $\Ft$. 

Like this we are ready to formulate our main abstract result in \Cref{th:abstract_estimate_2}, 
which is a slight modification of \citep[Theorem 3.2]{devore2013greedy}.
A full proof is provided.

\begin{theorem} \label{th:abstract_estimate_2}
Consider a precompact set $\F$ in a separable Hilbert space $\mathcal{H}$
and a subset $\Ft \subseteq \F$ that contains infinitely many elements.
Consider a greedy algorithm that selects pairwise different elements $\{ f_0, f_1, ... \}$ from $\Ft \subseteq \F$ according to Eq.~\eqref{eq:selection_criterion_subset}.
We have the following inequalities between $\sigma_n(\Ft)$ and $d_n(\F)$ for any $N \geq 0, K \geq 1, 1 \leq m < K$:
\begin{align}
\label{eq:abstract_theorem}
\prod_{i=1}^K &\sigma_{N+i}(\Ft)_{\Ha}^2 \leq \gamma^{-2K} \left( \frac{K}{m} \right)^m \left( \frac{K}{K-m} \right)^{K-m} \sigma_{N+1}(\Ft)_{\Ha}^{2m} d_m(\F)_{\Ha}^{2K-2m}.
\end{align}
\end{theorem}

\begin{proof}
We proceed similarly to the proof of \citep[Theorem 3.2]{devore2013greedy}:
As $\Ha$ is separable, we assume without loss of generality that $\Ha$ is $\ell^2(\N \cup \{0\})$.
For the infinite sequence $(f_n)_{n \geq 0} \subseteq \Ft \subset \Ha$ consisting of pairwise different elements, we consider the orthonormal system $(f_n^*)_{n \geq 0}$ obtained by Gram-Schmidt orthogonalization.
For the orthogonal projector $P_n: \Ha \rightarrow V_n \equiv \Sp \{f_0, ..., f_{n-1} \}$ it then holds
\begin{align*}
P_n f = \sum_{i=0}^{n-1} \langle f, f_i^* \rangle f_i^*.
\end{align*}
Especially $f_n$ can be expressed in this orthogonal basis, and we collect the coefficients in an (infinite dimensional) lower triangular matrix $A$:
\begin{align*}
A := (a_{i,j})_{i,j=0}^\infty, ~ a_{i,j} := \langle f_i, f_j^* \rangle_{\Ha}.
\end{align*}
Now we consider the $K \times K$ matrix $G = (g_{i,j})$ which is formed by the rows and columns of $A$ with indices from $\{ N+1, ..., N+K\}$. 
Each row $\textbf{g}_i$ is the restriction of $f_{N+i}$ to the coordinates $N+1, ..., N+K$.
Let $\Ha_m$ be the $m$-dimensional so-called Kolmogorov subspace of $\Ha$ for which $\dist(\F, \Ha_m) = d_m(\F)$.
Then, $\dist(f_{N+i}, \Ha_m) \leq d_m(\F)$, $i=1, ..., K$.
Let $\tilde{W}$ be the linear space which is the restriction of $\Ha_m$ to the coordinates $N+1, ..., N+K$.
In general, $\dim(\tilde{W}) \leq m$. 
Let $W$ be an $m$-dimensional space, $W \subset \Sp \{e_{N+1}, ..., e_{N+K} \}$, such that $\tilde{W} \subset W$ and $P$ and $\tilde{P}$ are the projections from $\R^K$ onto $W$ and $\tilde{W}$, respectively.

As the selection criterion chooses only elements from $\Ft$, it holds $f_i \in \Ft$ for all $i=0, 1, ...$, and thus
\begin{align*}
\Vert P \textbf{g}_i \Vert_{\ell^2}^2 &\leq \Vert \textbf{g}_i \Vert_{\ell^2}^2 = \Vert f_{N+i} - \Pi_{V_{N+1}}(f_{N+i}) \Vert_{\Ha}^2 \\
&\leq \sup_{f \in \Ft} \Vert f - \Pi_{V_{N+1}}(f) \Vert_{\Ha}^2 = \sigma_{N+1}(\Ft)_{\Ha}^2 \quad \forall m \geq n.
\end{align*}
Furthermore,
\begin{align*}
\Vert \textbf{g}_i - P\textbf{g}_i \Vert_{\ell^2} &\leq \Vert \textbf{g}_i - \tilde{P} \textbf{g}_i \Vert_{\ell^2} = \dist(\textbf{g}_i, \tilde{W}) \\
&\leq \dist(f_{N+i}, \Ha_m) \leq d_m(\F), \quad i=1, ..., K.
\end{align*}
Finally we obtain
\begin{align*}
g_{i,i} \equiv a_{N+i, N+i} &\hspace{.35cm}\equiv \Vert f_{N+i} - \Pi_{V_{N+i}} f_{N+i} \Vert_{\Ha} = \dist(f_{N+i}, V_{N+i})_\Ha \\
&\stackrel{\text{Eq.\ }\eqref{eq:selection_criterion_subset}}{\geq} \gamma \cdot \sigma_{N+i}(\Ft)_{\Ha}.
\end{align*}
Now \Cref{lem:abstract_lemma} can be leveraged:
\begin{align*}
\prod_{i=1}^K g_{i,i}^2 &\leq \left \{ \frac{1}{m} \sum_{i=1}^K \Vert P\textbf{g}_i \Vert_2^2  \right \}^m \left \{ \frac{1}{K-m} \sum_{i=1}^K \Vert \textbf{g}_i - P\textbf{g}_i \Vert_2^2 \right \}^{K-m} \\
\Rightarrow \quad \prod_{i=1}^K \gamma^2 \sigma_{N+i}(\Ft)^2 &\leq \left \{ \frac{1}{m} \sum_{i=1}^K \sigma_{N+1}(\Ft)_{\Ha}^2  \right \}^m \left \{ \frac{1}{K-m} \sum_{i=1}^K d_m(\F)^2 \right \}^{K-m} \\
\Leftrightarrow \qquad \prod_{i=1}^K \sigma_{N+i}(\Ft)^2 &\leq \gamma^{-2K} \left( \frac{K}{m} \right)^m \left( \frac{K}{K-m} \right)^{K-m} \sigma_{N+1}(\Ft)^{2m} d_m(\F)^{2K-2m}.
\end{align*}
\end{proof}

Similarly to \citep[Corollary 3.3]{devore2013greedy}, we can derive decay statements for $\sigma_n(\Ft)_{\Ha}$ based on decay statements for $d_n(\F)_{\Ha}$ due to 
Theorem \ref{th:abstract_estimate_2}. This is summarized in the following corollary.

\begin{cor} \label{cor:decay_abstract_setting_prod_2} 
For the greedy algorithm of Eq.~\eqref{eq:selection_criterion_subset} we have the following.
\begin{enumerate}[label=\roman*)]
\item If $d_n(\F)_{\Ha} \leq C_0 n^{-\alpha}, n=1, 2, \dots$, then it holds
\begin{align} \label{eq:estimate_sigma_alg}
\sigma_n(\Ft)_{\Ha} \leq C_1 n^{-\alpha},
\end{align}
for $n = 1, 2, \dots$ with $C_1 := 2^{5\alpha + 1} C_0 \gamma^{-2}$.
\item If $d_n(\F)_{\Ha} \leq C_0 e^{-c_0 n^\alpha}, n=1,2,\dots$, then it holds
\begin{align} \label{eq:estimate_sigma_exp}
\sigma_n(\Ft)_{\Ha} \leq \sqrt{2 \tilde{C}_0} \gamma^{-1} e^{-c_1 n^\alpha}
\end{align}
for $n=2, 3, \dots$ with $\tilde{C}_0 := \max \{ 1, C_0 \}$ and $c_1 = 2^{-(2+\alpha)}c_0 < c_0$.
\end{enumerate}
\end{cor}

\begin{proof}
For the first statement i), 
we can directly reuse \citep[Proof of Corollary 3.3]{devore2013greedy}. 
The only difference is that we have $\sigma_n(\Ft)_{\Ha}$ instead of $\sigma_n(\F)_{\Ha}$, 
which does not matter since the proof of \citep[Corollary 3.3]{devore2013greedy} is purely algebraical. 

For the second statement ii), we proceed similarly, however we note that \citep[Proof of Corollary 3.3]{devore2013greedy} is slightly inaccurate, 
because its Eq.~(3.9) uses 
\begin{align*}
\sigma_{2n+1} \leq \sqrt{2C_0} \gamma^{-1} e^{-c_0 2^{-1-\alpha} (2n)^\alpha},
\end{align*}
however 
\begin{align*}
\sigma_{2n+1} \leq \sqrt{2C_0} \gamma^{-1} e^{-c_0 2^{-1-\alpha} (2n+1)^\alpha},
\end{align*}
would be required.
Therefore we include the full proof:
We leverage \Cref{th:abstract_estimate_2} for $N = 0, K = n$, 
thus Eq.~\eqref{eq:abstract_theorem} turns into
\begin{align*}
\prod_{i=1}^n &\sigma_{i}(\Ft)_{\Ha}^2 \leq \gamma^{-2n} \left( \frac{n}{m} \right)^m \left( \frac{n}{n-m} \right)^{n-m} \sigma_{1}(\Ft)_{\Ha}^{2m} d_m(\F)_{\Ha}^{2n-2m}.
\end{align*}
For $0 < x := m/n < 1$ it holds $x^{-x} (1-x)^{x-1} \leq x$, thus we have
\begin{align*}
\left[ \left( \frac{n}{m} \right)^m \left( \frac{n}{n-m} \right)^{n-} \right]^{1/n} = x^{-x}(1-x)^{x-1} \leq 2.
\end{align*}
Using $\sigma_1(\Ft)_{\Ha} \leq 1$ due to $\Vert f \Vert_{\Ha} \leq 1$ for all $f \in \F$ we then obtain
\begin{align*}
\left( \prod_{i=1}^n \sigma_{i}(\Ft)_{\Ha} \right)^{1/n} \leq \sqrt{2} \gamma^{-1} d_m(\F)_{\Ha}^{(n-m)/n}.
\end{align*}
Finally using the monotonicity of $\sigma_i(\Ft)_{\Ha}$ and picking $m = \lceil n/2 \rceil$ gives 
\begin{align}
\label{eq:step_within_proof}
\sigma_n(\Ft)_{\Ha} &\leq \sqrt{2} \gamma^{-1} \cdot d_m(\F)^{(n-m)/n} = \sqrt{2} \cdot d_{\lceil n/2 \rceil}(\F)^{(n-\lceil n/2 \rceil)/n} \\
&\leq \sqrt{2} \gamma^{-1} \cdot \tilde{C}_0^{1/2} e^{-c_0 \lceil n/2 \rceil^\alpha \cdot (n-\lceil n/2 \rceil)/n} \notag \\
&\leq \sqrt{2} \gamma^{-1} \cdot \tilde{C}_0^{1/2} e^{-c_0 2^{-\alpha} n^\alpha \cdot (n-\lceil n/2 \rceil)/n} \notag \\
&\stackrel{n \geq 2}{\leq} \sqrt{2} \gamma^{-1} \cdot \tilde{C}_0^{1/2} e^{-c_0 2^{-2-\alpha} n^\alpha} \notag \\
&\equiv \sqrt{2 \tilde{C}_0}\gamma^{-1} e^{-c_1 n^\alpha}. \notag 
\end{align}
\end{proof}

\section{Kernel interpolation on arbitrary domains} \label{sec:kernel_interpolation}

We start in \Cref{subsec:convenient_connection} by recalling a convenient connection of the abstract analysis of Section \ref{sec:abstract_setting} to greedy kernel 
interpolation which was introduced in \citep{santin2017convergence} and subsequently leveraged as well in \citep{wenzel2021novel, wenzel2022analysis, wenzel2022adaptive}. 
By doing so, we are able to derive our main result in \Cref{th:main_result} in \Cref{subsec:main_result}. 
In \Cref{subsec:application_main_result}, several applications of that main result are 
given.

\subsection{Convenient connection} \label{subsec:convenient_connection}

We make use of the convenient connection between the abstract setting and the kernel setting as introduced in \citep{santin2017convergence}, and extend it to the consideration 
of two sets $\tilde{\Omega} \subset \Omega \subset \R^d$ and choose
\begin{align}
\label{eq:connection}
\begin{aligned}
\Ha &:= \nsb, \\
\Ft &:= \{ k(\cdot, x) , x \in \tilde{\Omega} \} \subset \Ha, \\
\F &:= \{ k(\cdot, x) , x \in \Omega \} \subset \Ha,
\end{aligned}
\end{align}
for some continuous kernel $k: \Omega \times \Omega \rightarrow \R$ such that $k(x, x) \leq 1$ (in order to satisfy the convenience assumption $\Vert f 
\Vert_{\Ha} \leq 1 ~ \forall f \in \F$).
As long as $\sup_{x \in \Omega} k(x, x) < \infty$, this can always be enforced by normalizing the kernel.
Thus we have 
\begin{align*}
V_n &= \Sp \{ f_0, ..., f_{n-1} \} \\
&= \Sp \{ k(\cdot, x_1), ..., k(\cdot, x_n) \} \\
&= \Sp \{ k(\cdot, x_i ), x_i \in X_n \} = V(X_n),
\end{align*}
for $X_n \subset \tilde{\Omega}$. 
With these choices we can relate the quantities from Eq.~\eqref{eq:quantities} from the abstract setting to the kernel setting as done in \citep{santin2017convergence}:
\begin{align}
\sigma_n(\F)_{\Ha} & \equiv \sup_{f \in \mathcal{F}} \mathrm{dist}(f, V_n)_\mathcal{H} = \sup_{f \in \mathcal{F}} \Vert f - 
\Pi_{V_n}(f) \Vert_{\mathcal{H}} \nonumber \\
&~= \sup_{x \in \Omega} \Vert k(\cdot, x) - \Pi_{\nsb, V(X_n)}(k(\cdot, x)) \Vert_{\nsb} 
= \Vert P_{k, \Omega, X_n} \Vert_{L^\infty(\Omega)},  \label{eq:connection_1} \\
d_n(\F)_{\Ha} &\equiv \inf_{Y_n \subset \mathcal{H}} \sup_{f \in \mathcal{F}} \mathrm{dist}(f, Y_n)_\mathcal{H}  = \inf_{Y_n 
\subset \mathcal{H}} \sup_{f \in \mathcal{F}} \Vert f - \Pi_{Y_n}(f) \Vert_\mathcal{H} \notag \\
&\leq \inf_{Y_n \subset \mathcal{F}} \sup_{f \in \mathcal{F}} \Vert f - \Pi_{Y_n}(f) \Vert_{\ns} = \inf_{X_n \subset \Omega} \Vert P_{k, \Omega, X_n} 
\Vert_{L^\infty(\Omega)} \label{eq:connection_2}.
\end{align}
For brevity, in Eq.~\eqref{eq:connection_2} and in the following we suppress the dimension and size constraints $\dim(Y_n) = n$ and $|X_n| = n$.
The precompactness of $\F$ is equivalent to $d_n(\F)_{\Ha} \stackrel{n \rightarrow \infty}{\longrightarrow} 0$ (see e.g.~\citep[Prop.\ 1.2 (page 10)]{pinkus2012n} together with the comments on \citep[Th.\ 1.1 i) (page 10)]{pinkus2012n}) and thus, due to 
Eq.~\eqref{eq:connection_2}, $\inf_{X_n \subset \Omega} \Vert P_{k, \Omega, X_n} \Vert_{L^\infty(\Omega)}$ $\stackrel{n \rightarrow \infty}{\longrightarrow} 0$ 
implies precompactness of $\F$.

In order to make use of the abstract result of Theorem \ref{th:abstract_estimate_2} and the resulting decay statements of Corollary 
\ref{cor:decay_abstract_setting_prod_2}, we need to understand the quantity $\sigma_n(\Ft)_{\Ha}$ which was defined in Eq.~\eqref{eq:quantities}. We have the 
following lemma:
\begin{lemma}
\label{lem:sigma_equals_powerfunc}
It holds
\begin{align*}
\sigma_n(\Ft)_\Ha = \Vert P_{\tilde{k}, \Omegat, X_n} \Vert_{L^\infty(\Omegat)}.
\end{align*}
\end{lemma}

\begin{proof}
By using the definition of $\sigma_n(\Ft)_{\Ha}$ from Eq.~\eqref{eq:quantities}, our choice of $\Ft$ and $\Ha$ in Eq.~\eqref{eq:connection}, and Lemma 
\ref{lem:power_funcs_on_diff_domains}, we have
\begin{align*}
\sigma_n(\Ft)_{\Ha} &= \sup_{f \in \Ft} \mathrm{dist}(f, V_n)_\mathcal{\Ha} = \sup_{f \in \Ft} \Vert f - \Pi_{V_n}(f) \Vert_\mathcal{\Ha} \\
&= \sup_{x \in \Omegat} \Vert k(\cdot, x) - \Pi_{\nsb, V(X_n)}(k(\cdot, x)) \Vert_{\nsb} \\
&= \sup_{x \in \Omegat} P_{k, \Omega, X_n}(x) \\
&= \sup_{x \in \Omegat} P_{\tilde{k}, \Omegat, X_n}(x) = \Vert P_{\tilde{k}, \Omegat, X_n} \Vert_{L^\infty(\Omegat)},
\end{align*}
which concludes the proof.
\end{proof}

Furthermore the weak greedy selection criterion from Eq.~\eqref{eq:selection_criterion_subset} results in
\begin{align}
\label{eq:selection_criterion_subset_2}
x_{n+1} \in \Omegat ~~ \text{such that} ~~ P_{\tilde{k}, \Omegat, X_n}(x_{n+1}) \geq \gamma \cdot \Vert P_{\tilde{k}, \Omegat, X_n} \Vert_{L^\infty(\Omegat)},
\end{align}
i.e.\ a weak $P$-greedy algorithm on $\Omegat$.

\subsection{Main result} \label{subsec:main_result}

Corollary \ref{cor:decay_abstract_setting_prod_2} in conjunction with Lemma \ref{lem:sigma_equals_powerfunc} now immediately yields the following theorem.

\begin{theorem}
\label{th:main_result}
Let $\Omegat \subset \R^d$ be an arbitrary infinite set.
If there exists a superset $\Omega \supset \Omegat$ and a sequence of (non-necessarily nested) sets of points $(X_n)_{n \in \N} \subset \Omega$ such that for 
some $\alpha > 0$ it holds
\begin{enumerate}
\item (Algebraic decay case)
\begin{align}
\label{eq:assumption_alg_decay}
\Vert f - \Pi_{\nsb, V(X_n)} f \Vert_{L^\infty(\Omega)} \leq C_0 n^{-\alpha }\cdot \Vert f \Vert_{\nsb} \quad \forall f \in \nsb,
\end{align}
then the weak $P$-greedy algorithm with parameter $\gamma \in (0, 1]$ using $\tilde{k}$ applied to $\Omegat$ (see Eq.~\eqref{eq:selection_criterion_subset_2}) gives a nested sequence of sets of points $(\tilde{X}_n)_{n \in \N} \subset \Omegat$ such 
that
\begin{align*}
\Vert f - \Pi_{\nsa, \tilde{V}(\tilde{X}_n)} f \Vert_{L^\infty(\Omegat)} \leq C_1 n^{-\alpha }\cdot \Vert f \Vert_{\nsa} \quad \forall f \in \nsa
\end{align*}
for $n=1, 2, ...$ with $C_1 := 2^{5\alpha + 1} C_0 \gamma^{-2}$.
\item (Exponential decay case)
\begin{align}
\label{eq:assumption_exp_decay}
\Vert f - \Pi_{\nsb, V(X_n)} f \Vert_{L^\infty(\Omega)} \leq C_0 e^{-c_0 n^{\alpha}} \cdot \Vert f \Vert_{\nsb} \quad \forall f \in \nsb,
\end{align}
then the weak $P$-greedy algorithm with parameter $\gamma \in (0, 1]$ using $\tilde{k}$ applied to $\Omegat$ (see Eq.~\eqref{eq:selection_criterion_subset_2}) gives a nested sequence of sets of points $(\tilde{X}_n)_{n \in \N} \subset \Omegat$ such that 
\begin{align*}
\Vert f - \Pi_{\nsa, \tilde{V}(\tilde{X}_n)} f \Vert_{L^\infty(\Omegat)} \leq \sqrt{2 \tilde{C}_0} \gamma^{-1} e^{-c_1 n^\alpha}  \cdot \Vert f \Vert_{\nsa} \quad \forall f 
\in \nsa
\end{align*}
for $n=2, 3, ...$ with $\tilde{C}_0 = \max \{ 1, C_0 \}$ and $c_1 = 2^{-(2+\alpha)} c_0 < c_0$.
\end{enumerate}
\end{theorem}

\begin{proof}
We make use of the choices from Eq.~\eqref{eq:connection} to connect the abstract setting with the kernel setting:
The prerequisites from Eq.~\eqref{eq:assumption_alg_decay} and \eqref{eq:assumption_exp_decay} are equivalent to bounds on $\inf_{X_n \subset \Omega} \Vert 
P_{k, \Omega, X_n} \Vert_{L^\infty(\Omega)}$ (see Eq.~\eqref{eq:definition_power_func}). Thus, due to Eq.~\eqref{eq:connection_2} we obtain bounds on 
$d_n(\F)_{\Ha}$ as
\begin{align*}
d_n(\F)_{\Ha} \leq \inf_{X_n \subset \Omega} \Vert P_{k, \Omega, X_n} \Vert_{L^\infty(\Omega)} \leq 
\left\{
\begin{array}{ll}
C_0 n^{-\alpha} & \text{algebraic decay case} \\
C_0 e^{-c_0 n^\alpha} & \text{exponential decay case},
\end{array}
\right.
\end{align*}
for $n=1,2, \dots$\ . This ensures also the precompactness of the set $\F = \{ k(\cdot, x), x \in \Omega \}$ in $\Ha = \nsb$.

From Lemma \ref{lem:sigma_equals_powerfunc} we recall $\sigma_n(\Ft)_\Ha = \Vert P_{\tilde{k}, \Omegat, X_n} \Vert_{L^\infty(\Omegat)}$, and thus an application 
of Corollary \ref{cor:decay_abstract_setting_prod_2} by using the bounds on $d_n(\F)_{\Ha}$ gives 
\begin{align*}
\Vert P_{\tilde{k}, \Omegat, X_n} \Vert_{L^\infty(\Omegat)} = \sigma_n(\Ft)_{\Ha} \leq 
\left\{
\begin{array}{ll}
C_1 n^{-\alpha} & \text{algebraic decay case} \\
\sqrt{2 \tilde{C}_0} \gamma^{-1} e^{-c_1 n^\alpha} & \text{exponential decay case} \\
\end{array}
\right.
\end{align*}
for $n=1, 2, ...$ respective $n=2, 3, ...$ with $C_1, \tilde{C}_0, c_1$ as specified in the statement of the theorem. 
Applying the $\Vert \cdot \Vert_{L^\infty(\Omegat)}$-norm to 
Eq.~\eqref{eq:definition_power_func} finally gives
\begin{align*}
&~ \sup_{0 \neq f \in \nsa} \frac{\Vert f - \Pi_{\nsa, \tilde{V}(X_N)}(f) \Vert_{L^\infty(\Omegat)}}{\Vert f \Vert_{\nsa}} \\
=& ~ \Vert P_{\tilde{k}, \Omegat, X_n} \Vert_{L^\infty(\Omegat)} \\
\leq& \left\{
\begin{array}{ll}
C_1 n^{-\alpha} & \text{algebraic decay case} \\
\sqrt{2 \tilde{C}_0} \gamma^{-1} e^{-c_1 n^\alpha} & \text{exponential decay case} \\
\end{array}
\right. .
\end{align*}
Rearranging gives the final result
\begin{align*}
\Vert f - \Pi_{\nsa, \tilde{V}(\tilde{X}_n)} f \Vert_{L^\infty(\Omegat)} \leq \Vert f \Vert_{\nsa} \cdot \left\{
\begin{array}{ll}
C_1 n^{-\alpha} & \text{algebraic decay case} \\
\sqrt{2\tilde{C}_0} \gamma^{-1} e^{-c_1 n^\alpha} & \text{exponential decay case} \\
\end{array}
\right.
\end{align*}
for all $f \in \nsa$.
\end{proof}

\subsection{Applications of the main result} \label{subsec:application_main_result}

There are several implications of \Cref{th:main_result}, and we discuss them separately in the next subsections.  
Under the notion \textit{non-Lipschitz} we address in the following quite arbitrary domains $\Omegat$, emphasizing that they can e.g.\ have irregular boundaries. We only require convergence rates on a larger domain $\Omega$, which are frequently available under some conditions on $\Omega$, e.g.\ having a Lipschitz boundary. 

The general overall idea is always to leverage \Cref{th:main_result} to carry over convergence rates (in the number of interpolation points) from a larger domain $\Omega$ to a smaller domain $\Omegat \subset \Omega$.

\subsubsection{Error estimates for optimally chosen points}

Theorem \ref{th:main_result} states $\Vert \cdot \Vert_{L^\infty(\Omegat)}$ convergence rates for interpolation on $\Omegat$ via $P$-greedily chosen interpolation 
points $X_n$, as soon as one knows convergence rates for interpolation on a larger domain $\Omega \supset \Omegat$. 
However, these greedily chosen points do not need to be optimal. For optimal points we have the trivial estimate
\begin{align*}
\inf_{\hat{X}_n \subset \Omegat} \Vert P_{k, \Omegat, \hat{X}} \Vert_{L^\infty(\Omegat)} \leq \Vert P_{k, \Omegat, X_n} \Vert_{L^\infty(\Omegat)} \qquad \forall X_n \subset \Omegat,
\end{align*}
where the right hand side can be further upper bounded with help of Theorem \ref{th:main_result}. 
This implies that the estimates of \Cref{th:main_result} are not limited to greedily selected points: 
It is also possible to obtain error estimates for optimally chosen interpolation points on $\Omegat$.

\subsubsection{Stability of the convergence rates with respect to the domain} 
\label{subsubsec:stability_wrt_domain}

Theorem \ref{th:main_result} implies the stability of the convergence rate with respect to the domain: Given a kernel $k$ and two nested 
domains 
$\Omegat \subset \Omega$, the convergence rate on $\Omegat$  (in terms of polynomial convergence $n^{-\alpha}$ or exponential convergence $\exp(-cn^\alpha)$) is 
at least as fast as the convergence on the larger domain $\Omega$. The case of an even faster convergence rate is exemplified in a numerical example in Subsection 
\ref{subsec:monotonicity_conv_rates}. 

An important application of this fact are kernels of infinite smoothness such as the Gaussian kernel: Sampling inequalities for those kernels 
\citep{oversampling_boundary, LEE201440, rieger2010sampling} usually assume quite restrictive assumptions on the domain\footnote{These sampling inequalities are 
usually formulated in terms of the fill distance $h_n \equiv h_{X_n, \Omega}$. For asymptotically uniformly distributed points it holds $h_n \asymp n^{-1/d}$. Thus we directly state 
these results in terms of the number of interpolation points.}: 
Here, $C$ is a generic constant.
\begin{enumerate}
\item Theorem 3.5 in \citep{rieger2010sampling} gives a rate of decay $\exp(-C\log(n)n^\frac{1}{2d})$ for bounded Lipschitz domains satisfying an interior cone 
condition, while \citep[Theorem 4.5]{rieger2010sampling} gives a decay of $\exp(-C\log(n)n^\frac{1}{d})$ for compact cubes $\Omega$.
\item Theorem 3.5 in \citep{oversampling_boundary} gives an improved rate of $\exp(-C\log(Cn)n^\frac{1}{d})$ for bounded domains that are star-shaped  with 
respect to a ball, if oversampling near the boundary is used.
\item Similar estimates for a broader class of kernels are given in \citep{LEE201440}. However, they are restricted to sets which are given as the union of 
cubes of the same side length.
\end{enumerate}
Using our approach we can directly conclude the best available convergence rate for any bounded domain $\Omegat$. To this end we simply consider a compact cube 
$\Omega$ such that $\Omega \supset \Omegat$. Then the application of Theorem \ref{th:main_result} in conjunction with \citep[Theorem 3.5]{oversampling_boundary} 
directly gives the same convergence rate for $\Omegat$. We summarize this statement in the following corollary. In the proof we use an extension of 
Corollary~\ref{cor:decay_abstract_setting_prod_2} that allows us to deal with an additional $\log(n)$ factor in the exponent. This extension is stated and 
proven in Corollary~\ref{cor:decay_abstract_setting_prod_3} in Appendix~\ref{sec:appendix}.

\begin{cor}
Consider the Gaussian kernel $k: \Omegat \times \Omegat \rightarrow \R$ and a bounded domain $\Omegat \subset \R^d$. 
Then there exists a nested sequence of point sets $X_n \subset \Omegat$ such that it holds
\begin{align*}
\Vert f - \Pi_{\nsa, V(X_n)}(f) \Vert_{L^\infty(\Omegat)} \leq C' \cdot e^{-C \log(n) n^{1/d}} \cdot \Vert f \Vert_{\nsa}
\end{align*}
for all $f \in \ns$ and $n \in \N$.
\end{cor}

\begin{proof}
The proof is very similar to the proof of Theorem \ref{th:main_result} for the exponential decay case, we only need to address the additional $\log(n)$ term in the exponent: \\
Since $\Omegat$ is bounded, we can consider a compact cube $\Omega$ such that $\Omegat \subset \Omega$. Due to \citep[Theorem 4.5]{rieger2010sampling} we have the existence of (non-necessarily nested) set of points $(X_n)_{n \in \N} \subset \Omega$ such that for $n \in \N, f \in \ns$ it holds
\begin{align*}
\Vert f - \Pi_{\ns, V(X_n)}(f) \Vert_{L^\infty(\Omega)} \leq C_0 \exp(-c_0 \log(n)n^\frac{1}{d}).
\end{align*}
Therefore it follows that $d_n(\F)_{\Ha} \leq C_0 \exp(-c_0 \log(n)n^\frac{1}{d})$. Leveraging Corollary \ref{cor:decay_abstract_setting_prod_3} from Appendix~\ref{sec:appendix} we obtain 
\begin{equation*}
\sigma_n(\Ft)_{\Ha} \leq \sqrt{2 \tilde{C}_0} e^{-\tilde{c}_1 \log(n) n^\alpha}
\end{equation*}
for $n = 4, 5, \dots $ with $\tilde{c}_1 = 2^{-(3+\alpha)} c_0 < c_0$.
Now we conclude as in the proof of Theorem \ref{th:main_result}: 
\begin{align*}
~ \sup_{0 \neq f \in \nsa} \frac{\Vert f - \Pi_{\nsa, \tilde{V}(X_n)}(f) \Vert_{L^\infty(\Omegat)}}{\Vert f \Vert_{\nsa}} 
=& ~ \Vert P_{\tilde{k}, \Omegat, X_n} \Vert_{L^\infty(\Omegat)} \\
=& \sigma_n(\Ft)_{\Ha} \\
\leq& \sqrt{2C_0} e^{- \tilde{c}_1 \log(n) n^\alpha}
\end{align*}
Rearranging the last equation (and possibly an adjustment of the constants to include $n \in \{1, 2, 3\}$) gives the final result.
\end{proof}

We remark that the use of Theorem \ref{th:main_result} may possibly lead to non-optimal convergence rates, see e.g.\ Subsection \ref{subsec:monotonicity_conv_rates}. 
Another possible application of Theorem \ref{th:main_result} is given in the case when $\Omegat$ is a 
manifold embedded in a larger ambient domain $\Omega \subset \R^d$. The same machinery works also in this case, but due the to smaller intrinsic dimension 
of the manifold it is sometimes possible to obtain a faster convergence by working directly in $\Omegat$. 

\subsubsection{Convergence rates for non-Lipschitz domains}

\Cref{th:main_result} can be used to deduce convergence rates for interpolation on non-Lipschitz domains and thus allows a comparison to 
Sobolev spaces: 
We consider a kernel $k$ such that $\ns \asymp H^\tau(\Omega)$ for $\tau > d/2$ for a well-shaped domain $\Omega$ (e.g.\ satisfying a Lipschitz boundary), see also \Cref{sec:introduction}.
In this case, decay statements as in Eq.~\eqref{eq:assumption_alg_decay} are available, e.g.\ via sampling inequalities \citep{narcowich2005sobolev, wendland2005approximate, narcowich2006sobolev}.
Then, \Cref{th:main_result} states that the same convergence rates (with adjusted constants) also hold in $\nsa$, i.e.\ in the RKHS over the smaller domains $\Omegat \subset \Omega$. 
On top, there is an algorithm to obtain such interpolation points, namely the (weak) $P$-greedy algorithm applied to $\Omegat$. 
This is exemplified in a numerical example in Subsection 
\ref{subsec:conv_rates_cusps}. 

It is crucial to point out that the decay property of the power function, which is independent of the shape of the domain $\Omega$, 
is also a striking difference compared to Sobolev spaces $H^\tau(\Omegat)$ on $\Omegat$ if defined in their standard way via the existence of integrable weak derivates. 
Using that definition of Sobolev spaces, the RKHS $\nsa$ of the considered kernel can be smaller than the corresponding Sobolev space.
The reason for this is that the RKHS $\nsa$ always allows for an extension from $\Omega \subset \R^d$ to $\R^d$ (see \Cref{subsec:restriction_extension_domain}), 
while for Sobolev spaces specific assumptions on the boundary are required in order to have an extension operator \citep{adams2003sobolev, agranovich2015sobolev}.
However, if one defines Sobolev spaces via the restriction of the global Sobolev spaces to $\Omega \subset \R^d$,
see e.g.\ \citep[Theorem 5.1.1]{agranovich2015sobolev} or \citep[Section 2.4]{novak2006function}, 
the spaces will coincide with $\nsa$.

\begin{cor}
Consider a kernel $k: \Omega \times \Omega \rightarrow \R$ and a domain $\Omega \subset \R^d$ such that $\ns \asymp H^\tau(\Omega)$ for some $\tau > d/2$. 
Furthermore let $\Omegat \subset \Omega$ be a measurable, non-Lipschitz domain such that $H^\tau(\Omegat) \cancel{\hookrightarrow} H^\tau(\Omega)$, 
i.e.\ there is no continuous embedding of $H^\tau(\Omegat)$ into $H^\tau(\Omega)$. Then the following inclusions of the 
function spaces hold:
\begin{center}
\begin{tabular}{ccc}
  $H^\tau(\Omegat)$ & $\cancel{\hookrightarrow}$ & $H^\tau(\Omega)$ \\
  \rotatebox[origin=c]{90}{$\subsetneq$} & & \rotatebox[origin=c]{90}{$\asymp$} \\
  $\nsa$ & $\hookrightarrow$ & $\ns$ \\
\end{tabular}
\end{center}
i.e.\ the RKHS $\nsa$ over $\Omegat$ is only a subset of the corresponding Sobolev space $H^\tau(\Omegat)$ over $\Omegat$.
\end{cor}
\begin{proof}
The equivalence $\ns \asymp H^\tau(\Omega)$ was stated as an assumption, and holds often in practice, see e.g.\ \citep[Corollary 10.48]{wendland2005scattered}. The 
embedding $\nsa \hookrightarrow \ns$ follows from \citep[Theorem 10.46]{wendland2005scattered}. For the subset-relation $\nsa \subsetneq H^\tau(\Omegat)$ we have the 
following reasoning. Let $E$ be the extension operator $E: \nsa \rightarrow \ns$, so that
\begin{align*} %
f \in \nsa &\Rightarrow Ef \in \ns \\
&\Rightarrow Ef \in H^\tau(\Omega) \\
&\Rightarrow f = Ef|_{\Omegat} \in H^\tau(\Omegat).
\end{align*}
For the last step we simply used that in the computation of the Sobolev norm, we are just integrating over a smaller domain $\Omegat \subset \Omega$. \\
It holds $H^\tau(\Omegat) \neq \nsa$ because otherwise this would contradict the assumption $H^\tau(\Omegat) \cancel{\hookrightarrow} H^\tau(\Omega)$.
\end{proof}

\section{Numerical experiments} \label{sec:num_experiments}

In the following numerical experiments we use the basic, the linear and the quadratic Matérn kernels, which are defined as 
\begin{align}
\label{eq:matern_kernels}
\begin{aligned}
k_\mathrm{basic}(x, y) &= e^{-\Vert x - y \Vert_2}, \\
k_\mathrm{lin.}(x, y) &= e^{-\Vert x - y \Vert_2} \cdot (1 - \Vert x - y \Vert_2), \\
k_\mathrm{quadr.}(x, y) &= e^{-\Vert x - y \Vert_2} \cdot \frac{1}{3} \cdot ( 3 + 3 \Vert x - y \Vert_2 + \Vert x-y \Vert_2^2).
\end{aligned}
\end{align}
For Lipschitz domains $\Omega$, the RKHS $\ns$ of these kernels is norm equivalent to certain Sobolev spaces. 
Further details are provided in the following.

\subsection{Improvement of the convergence rates on subdomains}
\label{subsec:monotonicity_conv_rates}

The first numerical experiment will illustrate that it is also possible to obtain a faster convergence order, i.e.\ the proven stability of the convergence order also allows to have an improved convergence order.\footnote{We also considered ``monotonicity'' as possible notion instead of ``stability'', but this would be misleading as monotonicity in 
convergence behaviour might be misunderstood as the (trivial) decay of convergence rates over increasing dimension.}
For this we consider the Lipschitz domains
\begin{align*}
\Omega &= B_1(0), \\
\Omegat &= B_1(0) \setminus B_{1/2}(0) \subset \Omega,
\end{align*}
with $B_r(0) := \{ x \in \R^2 ~ | ~ \Vert x \Vert < r \}$. We use the kernel
\begin{align*}
k(x, y) := k_\mathrm{quadr.}(x, y) + \chi(x) \chi(y) \cdot k_\mathrm{lin.}(x, y)
\end{align*}
where $k_\mathrm{quadr.}$ is the quadratic Mat{\'e}rn kernel and $k_\mathrm{lin.}$ is the linear Mat{\'e}rn kernel as defined in Eq.\ \eqref{eq:matern_kernels} and $\chi$ is the indicator function of $B_{1/2}(0)$. For Lipschitz 
domains $\Omega'$ such as $\Omega$ or $\tilde{\Omega}$, the quadratic Mat{\'e}rn kernel satisfies $\mathcal{H}_{k_\mathrm{quadr.}}(\Omega') \asymp H^{\tau_1}(\Omega')$ with 
$\tau_1 = (d+5)/2 = 7/2$, while the linear Mat{\'e}rn kernel satisfies $\mathcal{H}_{k_\mathrm{lin.}}(\Omega') \asymp H^{\tau_2}(\Omega')$ with $\tau_2 = (d+3)/2 = 5/2$. 
Therefore, according to \citep{santin2017convergence} the expected convergence rate of $P$-greedy using $k_\mathrm{quadr.}$ or $k_\mathrm{lin.}$ on either $\Omegat$ or $\Omega$ is
\begin{align}
\label{eq:conv_rates_matern}
\frac{1}{2} - \frac{\tau}{d} = \left\{
\begin{array}{ll}
-0.75 & \text{linear Mat{\'e}rn kernel } k_\mathrm{lin.} \\ %
-1.25 & \, \text{quadratic Mat{\'e}rn kernel } k_\mathrm{quadr.} \\ %
\end{array}
\right. .
\end{align}

The numerically observed convergence using $k$ is visualized in Figure \ref{fig:P_greedy_monotonicity}. One can observe that the decay of $P$-greedy using $k$ 
on the smaller domain $\Omegat \subset \Omega$ is faster compared to the decay on the larger domain $\Omega$. The decay on $\Omegat$ follows an asymptotic of 
$n^{-1.25}$ because it holds $k|_{\Omegat \times \Omegat} = k_\mathrm{quadr.}$ in accordance with Eq.~\eqref{eq:conv_rates_matern}. The decay on $\Omega$ is 
slower and seems to follow $n^{-0.75}$ which is motivated by the contribution of the linear Matérn kernel $k_\mathrm{lin.}$.

We point out that the application of Theorem \ref{th:main_result} provides only the same convergence rate for the smaller domain $\Omegat$ in this example, 
while we can observe an even faster convergence. So our notion of ``stability'' does not exclude such better rates, but rather should express that it prevents 
``worse'' rates.

\begin{figure}[h]
\centering
\setlength\fwidth{.35\textwidth}
\input{Figures2/monotonicity_ctrs2.tex}
\input{Figures2/monotonicity_ctrs1.tex}
\setlength\fwidth{.65\textwidth}
\input{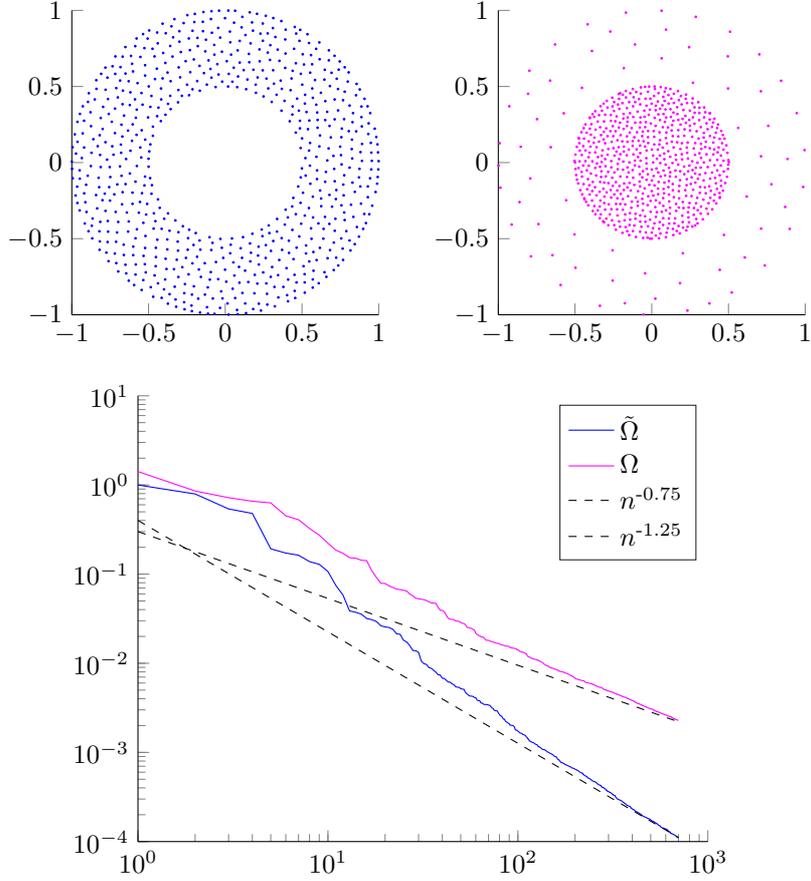}
\caption{Top: Visualization of the selected points for $P$-greedy applied to $\Omegat$ (left) and $\Omega$ (right). Bottom: Visualization of the decay $\Vert P_{k, \Omega, X_n} \Vert_{L^\infty(\Omega)}$ and $\Vert P_{\tilde{k}, \tilde{\Omega}, X_n} \Vert_{L^\infty(\Omega)}$ ($y$-axis)
in the number $n$ of selected interpolation points ($x$-axis). The convergence rate on the smaller domain $\Omegat$ is faster compared to the larger domain $\Omega$.}
\label{fig:P_greedy_monotonicity}
\end{figure}

\subsection{Convergence rates for non-Lipschitz domains}
\label{subsec:conv_rates_cusps}

In this subsection we consider the application of the $P$-greedy algorithm to the two domains
\begin{align}
\label{eq:definition_domains}
\begin{aligned}
\Omega &:= [0, 1]^2, \\
\tilde{\Omega} &:= \{ x \in \Omega ~ | ~ \Vert x \Vert_2 > 1 \} \subset \Omega,
\end{aligned}
\end{align}
so that the domain $\tilde{\Omega}$ has two cusps at $(0, 1)^\top, (1, 0)^\top \in \R^2$, i.e.\ it does not have a Lipschitz boundary. 
For this example we use in Subsection \ref{subsubsec:matern}
the basic Mat{\'e}rn kernel $k_\mathrm{basic}$ and the linear Mat{\'e}rn kernel $k_\mathrm{lin.}$, and in Subsection \ref{subsubsec:matern} the 
Gaussian kernel. 

\subsubsection{Matérn kernels} \label{subsubsec:matern}

Although $\tilde{\Omega}$ does not have a Lipschitz boundary, the decay of the power function follows the expected decay for Lipschitz domains like 
$\Omega$. Especially -- in accordance with Theorem \ref{th:main_result} -- we can observe a convergence rate as for the Matérn kernels
\begin{align*}
\Vert P_{\tilde{k}, \tilde{\Omega}, X_n} \Vert_{L^\infty(\tilde{\Omega})}
&\asymp n^{-\frac{1}{2} - \frac{\tau}{d}} \\
&= \left\{
\begin{array}{ll}
n^{-1/4} & \, \textrm{basic Mat{\'e}rn kernel} \\
n^{-3/4} & \, \textrm{linear Mat{\'e}rn kernel} \\
\end{array}
\right., 
\end{align*}
as visualized in Figure \ref{fig:P_greedy_2D}. Furthermore, the points chosen by $P$-greedy seem to be asymptotically uniformly distributed in $\Omegat$ and in particular no clustering next to the cusps is 
observed. Moreover it can be observed that, despite having the same rate, the blue $\Vert P_{\tilde{k}, \Omegat, X_n} \Vert_{L^\infty(\Omegat)}$ 
curve is lower than the magenta $\Vert P_{k, \Omega, X_n} \Vert_{L^\infty(\Omega)}$ curve, i.e., there is a smaller multiplicative constant in front of the 
rates.
This makes actually sense, since $500$ points were selected in $\Omegat$, which results in a higher density compared to 
$500$ points within $\Omega$.

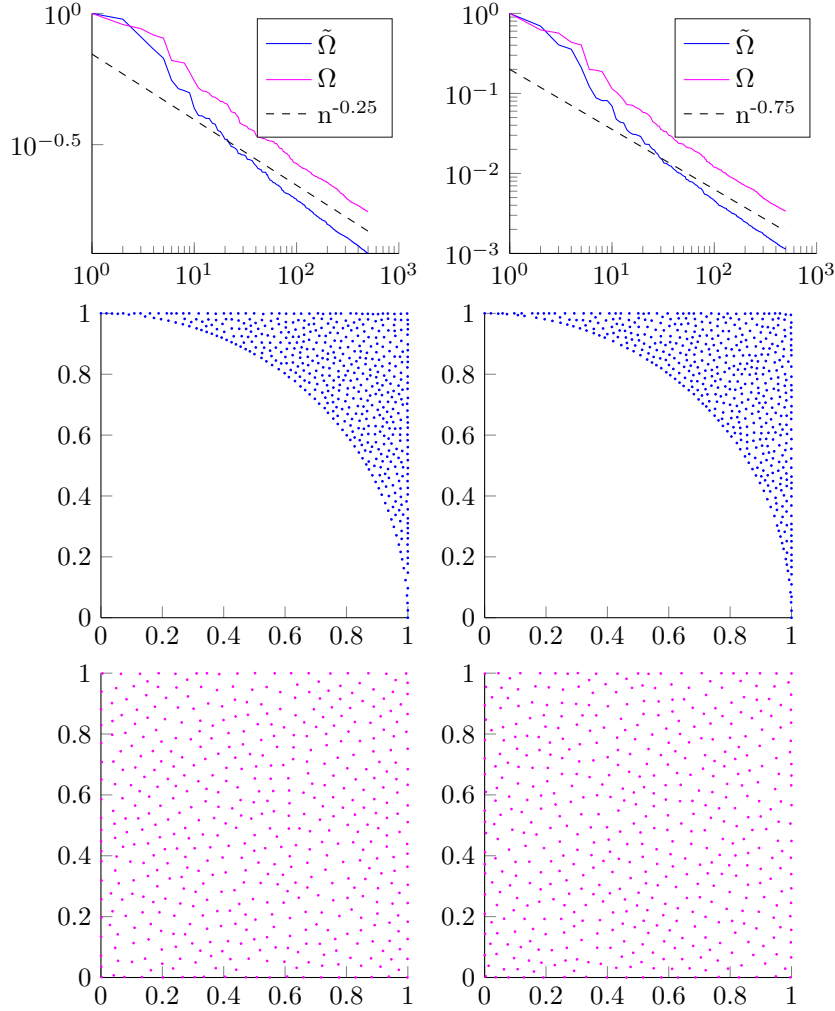
\begin{figure}[h]
\centering
\setlength\fwidth{.35\textwidth}
\input{Figures2/mat0_P_decay.tex}
\input{Figures2/mat1_P_decay.tex}
\input{Figures2/mat0_ctrs1.tex}
\input{Figures2/mat1_ctrs1.tex}
\input{Figures2/mat0_ctrs2.tex}
\input{Figures2/mat1_ctrs2.tex}
\caption{Visualization of the $P$-greedy algorithm applied to $\Omega$ and $\tilde{\Omega} \subset \Omega$ as defined in Eq.~\eqref{eq:definition_domains}: 
Left: Basic Mat{\'e}rn kernel, right: Linear Mat{\'e}rn kernel. On top the decay of $\Vert P_{k, \Omega, X_n} \Vert_{L^\infty(\Omega)}$ respectively $\Vert 
P_{\tilde{k}, \tilde{\Omega}, X_n} \Vert_{L^\infty(\tilde{\Omega})}$ is displayed. The $P$-greedy selected points are visualized in the middle row (for $\Omegat$) and bottom row (for $\Omega$).}
\label{fig:P_greedy_2D}
\end{figure}

\subsubsection{Gaussian kernel}  \label{subsubsec:gaussian}

Also for the Gaussian kernel the decay of the power function for $P$-greedy on the non-Lipschitz domain $\Omegat$ follows at least the same decay as for 
Lipschitz domains like $\Omega$ as visualized in Figure \ref{fig:P_greedy_2D_gaussian}.

This especially illustrates the results of Subsection \ref{subsubsec:stability_wrt_domain} about the Gaussian kernel, where it was elaborated that there are no convergence rates known so far for interpolation with the Gaussian kernel on non-Lipschitz domains.
The visualization of the chosen interpolation points in Figure \ref{fig:P_greedy_2D_gaussian} shows that the points cluster next to the boundary, 
which can be expected for analytic kernels in accordance with results in \citep{oversampling_boundary}. 
This clustering is especially visible in the sharp corners of the domain $\Omegat$.

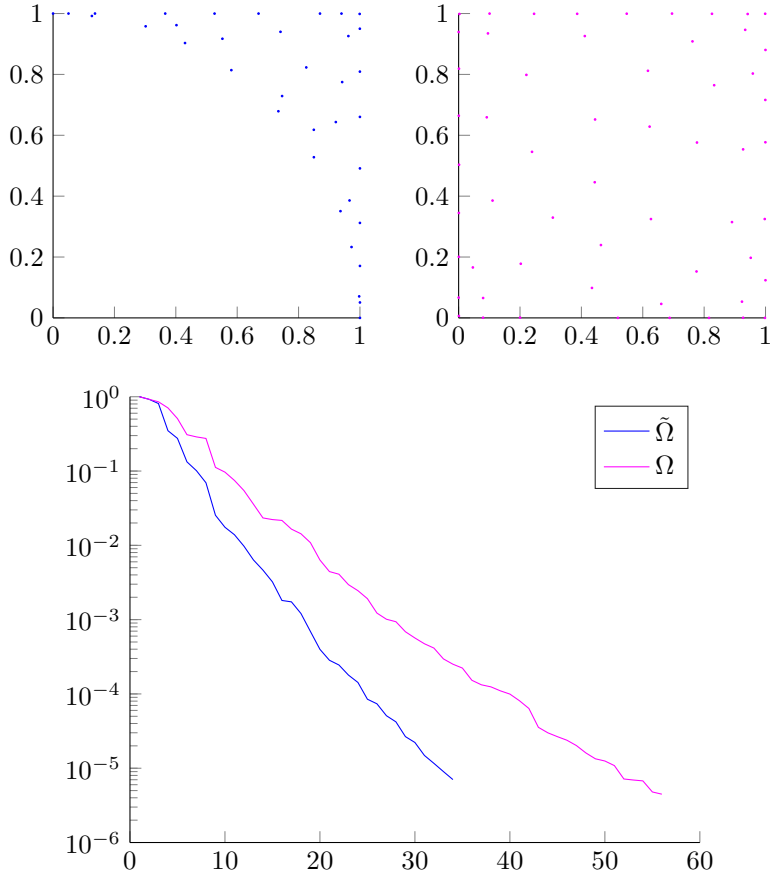
\begin{figure}[h]
\centering
\setlength\fwidth{.35\textwidth}
\input{Figures2/monotonicity_ctrs2_gaussian.tex}
\input{Figures2/monotonicity_ctrs1_gaussian.tex}
\setlength\fwidth{.65\textwidth}
\input{Figures2/monotonicity_conv_rates_gaussian.tex}
\caption{Top: Visualization of the selected points for $P$-greedy applied to $\Omegat$ (left) and $\Omega$ (right). Bottom:
Visualization of the decay $\Vert P_{\tilde{k}, \Omegat, X_n} \Vert_{L^\infty(\Omegat)}$ and $\Vert P_{k, \Omega, X_n} \Vert_{L^\infty(\Omega)}$ ($y$-axis) in the number $n$ of selected interpolation points ($x$-axis) for the Gaussian kernel. The convergence rate on the smaller domain $\Omegat$ is clearly at least as fast as on the larger domain.}
\label{fig:P_greedy_2D_gaussian}
\end{figure}

\section{Conclusion} \label{sec:conclusion}

We extended an abstract analysis of greedy algorithms in Hilbert spaces and proved that we obtain at least the same convergence rates, if we approximate only a 
smaller set of elements. 
By using a convenient link between this abstract setting, greedy kernel approximation and restriction and extension properties of Reproducing Kernel Hilbert 
Spaces, we transfered these results to potentially non-greedy kernel interpolation: In this case, the results imply a stability of the convergence rates with respect to 
the size of the domain. Especially we were able to obtain approximation rates for kernel interpolation on general even non-Lipschitz domains, including cases with cusps 
and irregular boundaries or even manifolds. A comparison with Sobolev spaces was drawn. Comments on the implications of these results were provided.

Future work will address the distribution of the $P$-greedy selected points on non-Lipschitz domains. For Lipschitz domains the point distribution for kernels of
finite smoothness was already analyzed in \citep{wenzel2021novel}, however its analysis made use of Sobolev space arguments, which are not necessarily available for 
non-Lipschitz domains. 

\vspace{1cm}
\textbf{Acknowledgements:} The authors acknowledge the funding of the project by the Deutsche Forschungsgemeinschaft (DFG, German Research Foundation) under
Germany's Excellence Strategy - EXC 2075 - 390740016 and funding by the BMBF under contract 05M20VSA.

\bibliography{references_stability}			%

\bibliographystyle{plainnat}

\appendix

\section{Technical lemma}\label{sec:appendix}

In the following we have an extension of \Cref{cor:decay_abstract_setting_prod_2} that additionaly includes a logarithmic term.

\begin{cor} \label{cor:decay_abstract_setting_prod_3} 
For the greedy algorithm of Eq.~\eqref{eq:selection_criterion_subset} we have the following
\begin{enumerate}[label=\roman*)]
\item[iii)] If $d_n(\F)_{\Ha} \leq C_0 e^{-c_0 \log(n) n^\alpha}, n=1,2,\dots$, then it holds
\begin{align*} %
\sigma_n(\Ft)_{\Ha} &\leq \sqrt{2 \tilde{C}_0} \gamma^{-1} e^{-\tilde{c}_1 \log(n) n^\alpha}
\end{align*}
for $n=4, 5, \dots$ with $\tilde{C}_0 := \max \{ 1, C_0 \}$ and $\tilde{c}_1 = 2^{-(3+\alpha)}c_0 < c_0$.
\end{enumerate}
\end{cor}

\begin{proof}
The proof is a modification of the proof of \Cref{cor:decay_abstract_setting_prod_2}.
We continue at Eq.~\eqref{eq:step_within_proof}:
\begin{align*}
\sigma_n(\Ft)_{\Ha} &\leq \sqrt{2} \gamma^{-1} \cdot d_m(\F)^{(n-m)/n} = \sqrt{2} \cdot d_{\lceil n/2 \rceil}(\F)^{(n-\lceil n/2 \rceil)/n} \\
&\leq \sqrt{2} \gamma^{-1} \cdot \tilde{C}_0^{1/2} e^{-c_0 \lceil n/2 \rceil^\alpha \cdot \log (\lceil n/2 \rceil) \cdot (n-\lceil n/2 \rceil)/n} \\
&\leq \sqrt{2} \gamma^{-1} \cdot \tilde{C}_0^{1/2} e^{-c_0 2^{-\alpha} n^\alpha \cdot \log (n/2) \cdot (n-\lceil n/2 \rceil)/n} \\
&\stackrel{n \geq 2}{\leq} \sqrt{2} \gamma^{-1} \cdot \tilde{C}_0^{1/2} e^{-c_0 2^{-2-\alpha} n^\alpha \log (n/2)}.
\end{align*}
Now using $\log(n/2) = \log(n) - \log(2) \geq \frac{1}{2} \log(n)$ for $n \geq 4$ we obtain
\begin{align*}
\sigma_n(\Ft)_{\Ha} &\stackrel{n \geq 4}{\leq} \sqrt{2 \tilde{C}_0} \gamma^{-1} \cdot e^{-c_0 2^{-3-\alpha} \log (n) n^\alpha}.
\end{align*}
\end{proof}

\end{document}

%% file: Figures2/monotonicity_ctrs2.tex
%
%
\begin{tikzpicture}

\begin{axis}[%
width=0.951\fwidth,
height=0.951\fwidth,
at={(0\fwidth,0\fwidth)},
scale only axis,
xmin=-1,
xmax=1,
ymin=-1,
ymax=1,
axis background/.style={fill=white},
axis x line*=bottom,
axis y line*=left
]
\addplot [color=blue, draw=none, mark size=0.3pt, mark=*, mark options={solid, blue}, forget plot]
  table[row sep=crcr]{%
0.666876131644869	-0.628552168403217\\
-0.716116179814444	0.697603196386512\\
-0.697486787632111	-0.71535204432924\\
0.675294951257282	0.737544119321126\\
0.0215804824855932	-0.998728407239278\\
-0.040453837655581	0.99818058967749\\
-0.499573483407422	-0.0269392519730483\\
0.999688315559911	0.00960168666229455\\
-0.999777540471861	0.00617681892768052\\
0.459594522235639	0.19725596775956\\
0.0437708696014969	-0.498602825503792\\
-0.119942441398248	0.512196622998674\\
0.907597518833523	0.419161692282059\\
-0.91564211585411	-0.400143590016986\\
0.408776130024993	-0.911227561860171\\
0.490471655595145	-0.264653534815728\\
-0.408257552983507	0.912746678651025\\
-0.442786164675143	-0.507110516087009\\
0.29120994091936	0.674580891124289\\
-0.622039326362539	0.377831093499956\\
0.916910129347178	-0.397393586778396\\
-0.373830410263378	-0.927059806280116\\
-0.928366979746132	0.369919288873293\\
0.365313275009628	0.929792676919353\\
0.754351906057788	0.0490138223465262\\
0.247032852712334	-0.721549591297347\\
-0.756533714976518	-0.161203428626529\\
-0.289552261134737	0.726136680699221\\
-0.158645063095198	-0.751148306648323\\
0.623473869671857	0.458586093687792\\
-0.355943188265553	0.351456175142761\\
0.75718350840512	-0.280983668831987\\
0.24949653689558	0.434669887756367\\
-0.780741343275221	0.133882425906218\\
-0.276101208299453	-0.418116295111191\\
0.0326032854231766	0.790469172144631\\
0.414443669049743	-0.513700186568578\\
-0.705319640091128	-0.45949977742858\\
0.609214593288913	-0.792744680230301\\
-0.522498653449241	0.639802928031957\\
0.974228819967863	0.222656658648961\\
-0.460469033037412	-0.741350234866366\\
0.159292688060753	0.987216084861456\\
-0.438333188679836	-0.241701963858898\\
0.816203960139732	-0.577294674655208\\
0.499265788554645	-0.0380667551617653\\
-0.97548758494011	-0.216134298207648\\
0.283096264321699	-0.412934563119652\\
-0.164314693356846	-0.985683660373647\\
0.774745932791628	0.285426214775546\\
-0.840541242502669	0.541287360963182\\
0.504664994694117	0.711725811909046\\
0.979793085448434	-0.193393203763007\\
0.0484622921836542	-0.829168978745292\\
-0.471472248668251	0.175910767527226\\
0.23258246054968	-0.972505072914898\\
-0.246858802933048	0.968811617562776\\
0.804482789237181	0.592642521401231\\
0.0834653665691323	0.56599451147253\\
-0.819282969764857	-0.573251486590247\\
0.425696961318542	0.475686483184059\\
0.451331807660336	-0.737558783602232\\
-0.355961054470408	0.534823933700121\\
-0.980681073910539	0.185621833924473\\
-0.544966234322427	-0.837250972812072\\
-0.577941505110389	0.815531402362237\\
-0.796072427062678	0.365774635148843\\
0.617432816169561	0.178410889269727\\
-0.604422850166354	-0.271324098916608\\
-0.162205848111068	-0.572291938555298\\
0.640746840184663	-0.435753436620119\\
0.535758825871111	0.843432256811478\\
0.651158938338185	-0.126375264823352\\
0.229201387201706	0.847462860516005\\
0.0528912130823169	-0.655721568412886\\
-0.646397364959736	0.0232714420009243\\
0.871580067324709	-0.0973030480378387\\
-0.145214612641445	0.863479814439501\\
-0.108843898733761	0.671487824945841\\
-0.876041304480816	-0.0229010675291768\\
-0.685202971418424	0.539111294516569\\
-0.59982630585705	-0.608183053182791\\
0.217078303909557	-0.552738300268149\\
-0.834809415711609	-0.301046507577006\\
-0.315913521007648	-0.638721420654742\\
0.529368132905331	0.332189210942915\\
-0.266865171210675	-0.858011862991133\\
0.891772388556541	0.12222196872595\\
0.789133027479301	-0.450495826307159\\
-0.633923588168155	0.209999079076161\\
0.271859367509489	-0.862329615570252\\
0.778428511468511	0.465853629536363\\
0.366252572837259	0.340857608551918\\
-0.439831367717336	0.783591941169692\\
-0.481218891215053	0.437769924839021\\
0.64859247570715	0.609775248892211\\
-0.533630695507898	-0.395810873575392\\
0.0185934279127415	0.499908229834103\\
-0.0810466107538788	-0.893154715710859\\
0.39098366501417	0.809135467420647\\
0.707105484944884	-0.70632665856455\\
0.393143740617515	-0.309488905227427\\
-0.585317865232219	-0.125383488174455\\
0.0575046775823462	0.919250370406028\\
-0.235837248409104	0.442616490154952\\
0.593958185487917	0.0379010251953582\\
-0.102581463199523	-0.490317557061481\\
0.526043770416113	-0.597214011616821\\
-0.874810396090844	0.239384851347654\\
0.884824361977673	-0.270161889918309\\
0.494691825645357	-0.40181016362914\\
-0.298902407295722	0.863187423303303\\
-0.502783584438572	0.296924320014473\\
0.154169651725322	0.713937265981252\\
0.130076654512327	-0.925846997704537\\
-0.397593801350842	-0.36926124503236\\
0.492215598947058	0.583242964457114\\
0.887651521872786	0.298744841822404\\
-0.906873484554142	-0.145839412969639\\
0.351898150102483	-0.640132326737565\\
-0.2279625092527	0.594067535410717\\
0.268411661810043	0.551964191111367\\
-0.919758953170925	0.0990123097427873\\
0.625469331498169	-0.286605500440527\\
-0.719982597738165	-0.597689934288431\\
-0.614261104332764	0.701057755070375\\
-0.820407412028246	-0.451590225375512\\
0.517963019530248	-0.854786401840568\\
0.475303168021951	-0.156958267843921\\
-0.0603532023207842	-0.653534836420928\\
-0.585021286806884	-0.735782953521626\\
0.136255190252189	0.482144624863977\\
0.664168253130027	0.333338901669974\\
-0.398970638641573	-0.83308686277473\\
0.769294717795518	-0.144750540414554\\
-0.309941062146099	-0.519221318993019\\
-0.553932091504939	0.0990836751230206\\
-0.714229526688385	-0.324185497426069\\
0.176660965671439	-0.467825488207415\\
-0.993929140277056	-0.103163148791737\\
0.491006732002879	0.0957706718406008\\
0.756828692402662	0.166713042938566\\
-0.390150605502799	0.654878755748579\\
-0.73952812395748	0.2679696974175\\
0.264114189182088	0.964443723413645\\
0.390656918017938	0.622143158258586\\
-0.571711690545286	0.522090821863133\\
0.150054457545697	-0.784933215886658\\
0.992085011786403	0.119805723187319\\
-0.753032604202818	-0.0295514839589068\\
-0.142069771800464	0.989144266483165\\
0.00992357478964934	0.666322133698719\\
-0.2650575062264	-0.963839202589249\\
0.575128647079088	-0.710768102729441\\
-0.451926944495095	-0.620718367945337\\
-0.30537186824079	-0.751743367668227\\
0.871592968273771	-0.486822801245617\\
-0.477801446452576	-0.14844683854778\\
-0.781326048611912	0.624115643020268\\
0.856053578847136	0.516871398168541\\
0.378608196256886	-0.816939800771179\\
0.99361247950096	-0.0947206503071862\\
-0.0459362980417013	-0.765861421378974\\
-0.585026245180736	-0.497536981462323\\
-0.178944316029102	0.765017218755857\\
0.891191160155951	0.0109496879175832\\
-0.49548276637546	0.868221072373056\\
0.135482772004861	-0.989834860241164\\
-0.780282611376482	0.478445979311051\\
0.71018254875668	-0.0417381704919384\\
-0.0460039335644051	0.884892185492318\\
0.592410116098042	0.747340619411474\\
0.1631595518111	-0.649159435183578\\
0.562123074241903	-0.500945555341422\\
-0.889146944152616	0.45574104013393\\
-0.419757000042602	0.272430399676395\\
0.95780200071145	-0.283491496315471\\
0.707582560227646	-0.525901639433552\\
0.520444083761996	0.445549716779346\\
0.387750017269592	-0.407758141609491\\
-0.949091921390413	-0.314252634872842\\
0.452783296237209	0.891094817004442\\
-0.457944261896997	0.5485515645436\\
-0.192266107307785	-0.463123379419953\\
-0.0718783507412295	0.773128845338862\\
-0.509072967773273	-0.271713712341503\\
-0.469442021874976	-0.881784330820553\\
0.554552405760496	-0.139112554970136\\
0.944873725502128	0.325021173366209\\
0.100432934277809	-0.562740666342188\\
0.131385372796202	0.838073659198132\\
0.183276192627625	0.60997282384971\\
-0.213690757287576	-0.662035419642736\\
-0.0437081127040235	0.578801339694379\\
-0.700263838666176	0.428414221228408\\
0.746118708983234	0.664672237188417\\
-0.680435235362634	0.117971657728898\\
-0.963490551431377	0.26739836834005\\
0.312849230979982	-0.515370924207894\\
0.0641539058960046	0.997682373803379\\
-0.866879405661141	-0.498105509705466\\
0.726520346988766	0.550835875151887\\
-0.358537658530153	0.438400025416828\\
-0.0221876369522724	-0.559951502760965\\
-0.668873261874317	-0.194539742348194\\
-0.0700083291143996	-0.994138167645941\\
0.293497556678302	0.771310703039498\\
0.323151684480581	-0.944531992153399\\
0.538892939118659	0.21947157235901\\
-0.629024568352968	-0.776116791786787\\
0.439101517072366	0.369046495700775\\
0.717923577321679	-0.369497670952275\\
-0.676054692120491	-0.0869908654609184\\
-0.62912101387337	-0.386726609009014\\
-0.526879903404211	0.752205258414887\\
0.334276540253306	0.445517579302974\\
0.835189969813211	-0.356889891021425\\
-0.495122514382371	0.0736480035787552\\
-0.17593409284662	-0.894194598850589\\
0.81034926346639	0.374700771983945\\
-0.68311645063654	0.627218055941917\\
0.392637290831726	0.716648622808671\\
0.810205484549146	-0.0333476014996403\\
-0.760939138946305	-0.6478204386933\\
0.343721126403487	-0.739393326015238\\
-0.99475486455393	0.100871205530821\\
0.767939182218792	-0.639821605433314\\
-0.815874749642689	0.0501600542945373\\
0.0149660611573068	-0.925843102096941\\
0.584856083681316	0.542225346072917\\
0.679527902262476	0.240643940373316\\
0.67356532345575	0.0928482742403196\\
0.444714733730068	-0.634436189431142\\
-0.248726266418112	0.511845771631657\\
0.679092967334404	-0.212455086629424\\
0.42000882297721	0.27140590063754\\
0.31056897803117	0.88363781802891\\
-0.330986606722802	0.942758710058966\\
-0.842002917677848	-0.21365573776607\\
0.483407976172769	0.79982050125386\\
-0.659771985134514	0.308927931247299\\
0.174553301381432	0.921118350340526\\
0.853564043343987	0.205254473909051\\
0.712810519013719	0.414848076468\\
-0.522142024510962	-0.561190137775442\\
-0.435439674651244	0.360991985103074\\
-0.825596802632627	-0.100315054428724\\
-0.372934655397793	-0.33358653906327\\
0.0596500136130222	-0.742261465780401\\
0.572176938733347	-0.360409144255671\\
-0.201431857374425	0.923434525656214\\
0.59829487928507	-0.0512350999433604\\
-0.303475331860851	0.625575737332493\\
-0.650811880100873	0.759093842520115\\
-0.382789311479578	-0.699357007085698\\
0.517777136531431	-0.787369995692952\\
-0.878161710237518	0.330644021551762\\
0.844523771200812	-0.192545957181455\\
-0.57619079882408	-0.022447892870137\\
0.566118039588121	0.645914467066546\\
0.616685445360941	-0.560290674213902\\
-0.911832602362614	-0.260437640603279\\
-0.524222716911801	-0.667579218204457\\
0.263584111091239	-0.626877970202689\\
-0.544633853538028	0.212404201511828\\
-0.765803291323147	-0.383946766941101\\
-0.204633885146163	0.67797996504864\\
0.361052064998543	0.538643099336526\\
0.92929446261408	-0.156488059459907\\
-0.38786573243983	0.849776152092224\\
-0.244637953601894	-0.57266256423932\\
0.570836427171175	-0.224296669274527\\
0.094378555365118	0.657273973516928\\
-0.601149994215411	0.59891898107976\\
0.613055053350324	0.789680736902031\\
-0.542710604156277	-0.200988238064352\\
-0.204277855904064	-0.812973749409715\\
0.198221816864723	0.511907217283385\\
-0.35867924344735	0.76627390126602\\
0.824452825950735	0.0885807763644242\\
-0.363033094274533	-0.447727584990567\\
0.552755615464986	0.128697136622151\\
-0.0933062838693046	0.943253450976504\\
-0.139010271835055	0.587846179444741\\
-0.94897330141762	-0.0507972487171409\\
0.189734884508701	-0.858922046621791\\
-0.751280769976998	-0.249703539735922\\
0.443396637172312	-0.23197667196257\\
-0.662157379581617	-0.539770278164957\\
-0.377600484449864	-0.570636353308026\\
-0.539030408785041	0.38053412660409\\
-0.85965787590749	0.153669998850256\\
0.951601279312906	-0.0447138334408055\\
-0.113856620522135	-0.820114946822111\\
0.601778609365527	0.279992784898121\\
-0.766889028512817	-0.521976256365158\\
0.452315263993161	-0.843304549652884\\
-0.165686392694918	0.472223601028819\\
-0.237198795448546	0.815573813411862\\
0.765038628829817	-0.565524504829652\\
-0.618473385733984	0.461536534536878\\
-0.767960848601263	0.558294645890975\\
-0.728558669861926	0.0516974470306928\\
-0.58283223714653	0.294253759594849\\
0.593035913261034	0.380333241759058\\
-0.460879701575977	0.697793551353488\\
-0.0625755905774248	0.496207830436577\\
0.267614142004223	-0.794003024631894\\
0.222703348946327	0.743636545330554\\
-0.485813886595009	-0.81889278090365\\
0.952115636359513	0.0653445605324083\\
-0.133430015781048	-0.661281560881487\\
-0.650517661132037	-0.671242395720199\\
-0.466863601110257	-0.433809201250511\\
-0.799227831983153	0.220716880591596\\
0.482199752987901	-0.485475144564525\\
0.337817295186437	-0.370332302396185\\
-0.719159502940483	0.193705476631427\\
-0.464969319907559	-0.333410191020002\\
-0.86472433302281	-0.36780868214325\\
0.675090018018454	0.671956297469688\\
0.0807859951758081	0.737235866089337\\
0.932402919916717	0.187818931808566\\
-0.850331609052064	0.415019341916379\\
0.51004159285267	-0.680626563966739\\
0.852401041269207	0.442362944727044\\
-0.026604213548316	-0.845983296709557\\
-0.323650821839004	-0.899453357056584\\
0.522181511388034	0.0294645720568871\\
0.681916864482654	0.502555954903524\\
0.242338373703874	-0.483425274543814\\
0.345837376167407	-0.886190051063204\\
0.85803326221775	-0.423542532973214\\
-0.242066677852045	-0.496435472755846\\
0.460184009400901	0.652832958845418\\
-0.307671545318619	0.395863974524378\\
0.302611595143507	0.398427998113728\\
-0.934817999868047	0.0297293679340875\\
0.459912926183087	-0.338881802370853\\
-0.0315712222326225	-0.501594994154254\\
-0.00162631516252243	0.948883173249286\\
0.595365927589591	-0.640667547653128\\
0.66509890880552	0.0141063030974995\\
0.129246864868187	-0.710092204423485\\
0.646706781511475	-0.7205573769744\\
-0.522858114198907	-0.0920759547805206\\
0.370280387946631	-0.573703370181372\\
-0.418350881967625	0.48689033219185\\
0.486348770359807	0.274263950709517\\
0.218806809659744	-0.927330125804492\\
-0.929169374643044	0.191921315875942\\
-0.091053613511866	-0.574453314338588\\
0.9059545814239	-0.337575523133553\\
-0.021795830812354	0.728365333290654\\
-0.718589232268843	0.354912050009162\\
-0.238829770099756	-0.734220354346052\\
0.649958362407904	-0.356905808910216\\
0.771566436231604	-0.215252945165526\\
-0.56334031633515	-0.329550807582972\\
0.50182819586979	0.512043450001593\\
0.393473860049089	0.873470969467942\\
0.11581852347158	-0.489943342055675\\
0.887678832212115	0.368388235519676\\
0.311976046933059	0.606777693082415\\
-0.445940332560177	0.616897062841843\\
0.713678531823352	-0.448286588731292\\
0.224251082907173	0.665582507580108\\
-0.118704620115935	-0.950194816937611\\
-0.520004571015181	-0.482268267324565\\
-0.608045967802115	0.142700629038116\\
-0.329410432490326	-0.815717210109808\\
0.735410640737836	0.342802674060087\\
-0.324874067515119	-0.380080797607867\\
0.336936504735023	-0.453527506439184\\
0.46473949073877	-0.559275462348265\\
-0.505875501351619	0.814835082920081\\
0.11188333028801	-0.856705442146924\\
-0.391554795939903	-0.768103248327554\\
-0.00125077703938392	-0.703951464083836\\
-0.546131922266534	0.457025132159071\\
-0.219808328075093	-0.939652717951721\\
-0.811680916550276	0.292963635435104\\
-0.547481606252685	0.0336911858917637\\
0.43165769775593	0.552672275552013\\
0.0301837980093447	0.860672665296883\\
0.798487001271574	0.532193455989664\\
0.076802685510327	-0.968406044845734\\
0.0229260603806496	0.582810434397837\\
-0.672832223223148	-0.274155912920075\\
0.559005546200344	-0.431955793339068\\
0.80991090547511	-0.511888372390354\\
0.688071004795129	0.165512453159875\\
0.929346675920263	-0.227387079260709\\
-0.305238378864324	0.476752476954555\\
-0.271069388743595	0.916635681672088\\
0.165068725450483	0.783339726075171\\
-0.633186692483147	-0.453398974335055\\
-0.522394252535652	0.571955092858256\\
0.696000162887988	-0.289825452841997\\
0.739216455798538	0.610770306176101\\
-0.0944185031591016	-0.719651595464682\\
-0.749198559167386	-0.0951524722326371\\
-0.53399796214968	-0.77206897901877\\
0.111303090980332	0.957405331341852\\
0.545219551629156	-0.298735896313427\\
0.386143315837571	0.405662203947516\\
0.244614530399098	0.913844038280416\\
0.398942735105826	-0.693555054200463\\
-0.96508761016736	-0.158236966714599\\
-0.471452418393539	0.236019011108178\\
0.523911422178442	-0.0854949454418406\\
-0.697772955428482	-0.389927711611824\\
0.713350827001519	-0.122649609822181\\
0.934605861345893	0.257687834595249\\
-0.293864305114997	0.560214461656034\\
-0.0516254281518223	0.64673066363102\\
-0.298504926279635	0.794493197777797\\
0.759856566490769	0.223461308997786\\
-0.586416464856266	0.764917124189676\\
0.160983155626537	-0.5313654325358\\
-0.686151842577322	-0.0248349165391908\\
0.280256278324413	0.487293583358248\\
-0.0338525676954278	-0.954877396326683\\
0.448484143004071	0.747665411821728\\
-0.167123797185529	-0.51180002369272\\
0.506407370544091	-0.19848372903979\\
0.659123029299002	-0.750946662474784\\
-0.839676425582941	0.484305292443427\\
-0.376035465730747	0.591643403372401\\
-0.0331664182289833	0.821322290071254\\
0.639508743822912	-0.499773387668449\\
0.0323358046486324	-0.595857676092232\\
0.545892249951037	0.785735553502038\\
0.775595874485479	-0.0835147458959784\\
-0.604704741119526	-0.197466331515512\\
-0.182648353036553	0.535345699189639\\
0.192359781419295	-0.737694661866666\\
-0.891330165408158	-0.450753593692438\\
0.819738655976262	-0.277250467297048\\
-0.872001649643824	0.0641817072015283\\
0.613383394446388	0.106121965477803\\
0.297696877139546	-0.580363428223811\\
0.14419748333351	0.545848301174403\\
0.456029494412664	-0.889743319587789\\
-0.521719183040381	0.150669509384003\\
0.0778054824265417	0.495714738637068\\
-0.27965115378961	-0.687161774961085\\
0.825946282947882	0.15069453667464\\
-0.896426268291779	-0.0867423306280446\\
0.326476434125509	0.825226061336697\\
-0.947229548979643	0.319674633940991\\
-0.443808667038696	0.854663592026441\\
-0.125688986604576	0.807297618747007\\
-0.4986411258428	0.500564443389255\\
-0.710212098243286	-0.660554913832153\\
-0.38884729709047	-0.632027269075857\\
-0.453578833742753	-0.680817720730201\\
-0.319793638220512	-0.947260918591257\\
0.468033682396739	0.420157311065737\\
-0.764762366675491	0.417287024541537\\
0.618208142475367	-0.183339552651285\\
0.614927847918865	0.679039393582558\\
-0.130087298928864	0.730831513106834\\
0.305515138307147	-0.686325247793502\\
-0.343778981676601	0.696746759599379\\
0.500371321507485	0.159091038802533\\
0.109297820156574	-0.626107322364806\\
0.778198936069521	-0.3776856210774\\
-0.660949589383039	-0.607000464788294\\
-0.668743626621777	0.696427988560382\\
-0.641252616176382	-0.32836408189718\\
-0.617988733105146	0.0779029643525302\\
0.437281873582062	-0.442337803194014\\
-0.81181286651362	-0.0294548186017267\\
0.748937660733171	0.10450838908464\\
-0.904221844066921	-0.324496474970547\\
-0.480908126964669	-0.204943406208145\\
0.652192034470029	0.403731858054789\\
0.320539854210525	-0.817175877717752\\
-0.710438973950291	0.48942552623993\\
0.722915829922369	-0.62506300316618\\
-0.350537240648447	0.8954352984062\\
-0.621584260575888	-0.0714294027053524\\
0.829302154162487	0.265488780430214\\
0.646090055040367	0.54505018913636\\
-0.377808434956185	-0.50509547616998\\
0.3128789700398	0.94952516187275\\
-0.734861613872043	0.632724828413651\\
0.244255692348548	0.606141098321919\\
0.995761463446326	0.0671987294690579\\
0.0942412555090686	0.799501398911055\\
0.341642595682947	0.742240549317453\\
-0.780476741433598	-0.317578556414926\\
-0.2132011954606	0.866129215713521\\
0.16735688162698	-0.592675707625393\\
0.656045888403417	-0.0614214671235789\\
-0.417295860458324	0.415093355109466\\
-0.920525110221708	0.274826946403577\\
0.827696166428883	0.0268968468731707\\
0.0645270039549442	-0.890086441623716\\
-0.404378375207737	0.720344361696472\\
-0.427134823947273	-0.297557665522687\\
0.273993034114106	-0.933144488581628\\
0.199828955316642	0.458347008013607\\
-0.583203816321028	-0.669941291671514\\
-0.57584526075485	-0.438113690077327\\
-0.305114166951538	-0.457369678407611\\
-0.821838909152054	-0.161549344995067\\
-0.55299398828657	0.693925468424427\\
-0.593760745931496	-0.804556919418141\\
-0.763908882746901	-0.454189397970558\\
-0.257938236930532	0.666214754636121\\
-0.67679142123119	0.248934724451194\\
0.352943473466495	0.668625316872082\\
0.518819620442383	0.387768848879569\\
-0.629695560312414	0.534470459603409\\
0.570072504319797	-0.821533686754574\\
-0.649756040343637	-0.14146330993632\\
-0.262374607323459	-0.798295998253142\\
0.106626956486473	0.888390146425902\\
0.551843022493348	-0.0201108763491615\\
-0.230462646161069	0.738913914210697\\
0.00655496290722501	-0.782828858729509\\
-0.919310151584147	-0.202036176700586\\
0.566970738333863	0.486890837534631\\
0.682815187466807	-0.575131382317177\\
0.829210964000372	-0.131670630911149\\
-0.0131868678478286	-0.626237109125886\\
0.40963014782278	-0.768971493224158\\
-0.41448337705873	-0.889002699856393\\
0.938757779439223	-0.33920414560419\\
-0.452477055390346	-0.560206022697591\\
0.712863689234823	0.284195826194174\\
-0.0895101724464298	0.994805593305596\\
-0.997893863493083	-0.0416354010997191\\
-0.0199865571980145	0.531850117647622\\
0.462238054006982	0.847839406349147\\
0.94399888925833	-0.103005585670919\\
-0.817484371355732	-0.518062745164446\\
0.237683625280732	-0.440045954288022\\
-0.307379849143536	-0.57870159516911\\
-0.774741215496616	-0.591632836720957\\
0.208961328743075	-0.795310083542665\\
0.83281342148354	0.324932035187858\\
0.206880974126173	0.977565351920555\\
0.875035965789125	-0.0413203010132115\\
0.218713007932391	-0.669701421203331\\
-0.532388138342711	-0.613194205338714\\
0.407683865247668	-0.36027754989081\\
0.526784749993403	-0.544832538274927\\
0.57046220379126	-0.767887803661775\\
-0.962830938173408	0.136199023828488\\
-0.487938995632559	0.36593154515825\\
-0.167281409065186	0.635438082424877\\
0.0816275475655226	-0.996478142255693\\
0.120196725678878	0.607521694302918\\
-0.08701061062187	0.853096237048745\\
0.553552232467044	0.590983241076436\\
0.989941358997734	-0.138648043280479\\
-0.83136561213954	0.107549921429209\\
0.437094992778658	0.31557755948424\\
-0.256096243140949	-0.626378208978917\\
-0.599679156137537	-0.551982244263188\\
-0.526589318540503	-0.149924487251421\\
0.881597087214235	0.0689590854777571\\
-0.607237213356837	0.649363253675069\\
0.0198293653908992	0.998864964959644\\
0.844063363010753	-0.535921866384203\\
0.876512678959005	0.480601952112328\\
0.0923325567882776	-0.789279802259598\\
0.712221105726491	0.701794903688855\\
-0.267987740059155	-0.917530059039753\\
-0.736493533728532	0.107136088662098\\
0.54602185458805	0.0788813367661694\\
0.758276383428961	-0.0118795637695499\\
0.513074048687077	0.644480050154085\\
0.156520083200115	0.65928408284112\\
-0.800344743183933	-0.247013069165763\\
-0.0709025650701991	0.709074160819003\\
0.723885716072393	-0.18065975160796\\
-0.198177493846286	0.9770176731237\\
-0.389670916849027	0.314109385933642\\
0.755953473077077	-0.498255462728558\\
0.436032481955798	-0.288198343304416\\
0.517037552417409	-0.733324207731033\\
-0.972783297098855	0.0591550775861416\\
-0.645496603919931	-0.72799271049674\\
-0.603397116746847	0.251250471368755\\
0.24503353823997	0.800377047744882\\
-0.703640546409453	-0.143459389105508\\
0.491813722672763	-0.0972924208669088\\
0.547769961192499	0.276382436735861\\
0.601464994907376	0.229007078232463\\
-0.083074413445331	-0.527497903098677\\
0.984570995200544	0.167794914523293\\
-0.443923033679952	0.304969818161999\\
0.368566661128719	0.484294087924378\\
0.518338686676448	-0.348203115737913\\
-0.414543549082959	-0.435168667969244\\
0.598083857286586	-0.108350770749449\\
0.948733054304182	0.12945870109138\\
0.0357198081892276	-0.540050033207867\\
0.775296756682907	-0.328118394568795\\
-0.18349469902017	-0.61910225949563\\
0.735021780442713	0.483257731596938\\
-0.518178907756998	-0.721466468815479\\
-0.14611707189329	0.925173465614442\\
-0.963511347566305	-0.265354194193774\\
0.181962085255722	-0.962361328901384\\
0.559006782580223	0.706197844082255\\
0.367319293375593	-0.496429484831837\\
-0.487908716395459	0.119530669078219\\
-0.808862636102386	0.583110054558409\\
-0.759285467025563	0.319394862517609\\
-0.489600469016773	-0.102052818404567\\
0.927669979111039	0.369833919108245\\
0.767700099518781	0.409756643667291\\
-0.171724494917147	-0.702564346242636\\
0.038068902433581	0.709441674975975\\
-0.814637663667531	-0.389323154033805\\
-0.0897729849630944	0.557856253134235\\
-0.711505508379969	-0.520589034383929\\
-0.678667777895526	0.0642162891136344\\
-0.132385171204561	-0.867653054069627\\
0.180519496529999	0.866770338467057\\
-0.455231661878394	-0.384871786875227\\
-0.948048683550587	-0.105895485477348\\
0.415261611476503	-0.593524018152924\\
0.0552825784039663	0.623249544589862\\
-0.717220062195571	0.582734217224559\\
0.57077097774729	0.425274375462872\\
-0.588404775114203	0.415476193943835\\
0.402849341189522	-0.871121923072415\\
0.705171376643784	0.0438382219595419\\
-0.112700183393041	-0.992840745405577\\
-0.333930078320858	-0.700289324769441\\
0.997472854664834	-0.0446952389010202\\
0.899166023448959	-0.435822183700642\\
-0.438334657358348	-0.793321289665543\\
-0.723145860942378	-0.208868414325775\\
0.21068319709446	0.559515878651983\\
-0.372914696570963	0.3857082759041\\
-0.535518775364942	0.844156172492681\\
0.692636604540038	0.589885110497208\\
0.805250962022767	0.206512407991101\\
-0.408414597471745	0.553178508503414\\
-0.894776334377259	0.398787215257771\\
0.276860202093903	0.722801300803312\\
0.464308724240641	-0.787452002592135\\
0.317285884892117	0.518327383065247\\
0.624850863606757	-0.2359990330513\\
0.880706673760731	-0.1570685484881\\
-0.339723263265814	0.823132099826854\\
-0.577029511852633	0.347480186638739\\
0.4883778704001	0.872174315836591\\
-0.869862614403973	-0.415479310541003\\
-0.0178470550134997	-0.892572241894044\\
0.213436055834745	-0.61104337242688\\
-0.0965044343343837	0.615050843484339\\
-0.670762034536678	0.171272510016341\\
-0.515663856583008	-0.341408329407495\\
0.574319429803037	-0.587561383575891\\
-0.474122054454432	0.747883630526527\\
-0.213212235934355	0.481518127508845\\
0.479699823312941	0.341152504285686\\
-0.555976654185109	-0.259017525438525\\
-0.234839799942669	-0.441458134787518\\
0.736618199818287	-0.676246239212807\\
-0.181107742421381	0.819356831402714\\
-0.216914980320567	-0.867565690911232\\
-0.650152732550822	0.585054361002588\\
-0.500067411698458	0.0251283971396314\\
0.607117976439612	-0.399717159320958\\
0.276578226655465	0.848795301736466\\
0.720712382060682	-0.239750855897158\\
0.296450879369557	-0.75099240725729\\
0.955465203830333	0.011381951342434\\
-0.667081635543258	0.472940717551786\\
0.438990152433081	0.603024314223336\\
0.599091328324296	0.329271847331099\\
-0.770105775828584	0.0192815721886614\\
-0.578141359755771	0.182186794876946\\
0.678560500480789	-0.402972631798743\\
-0.100592489636095	-0.769219482426558\\
-0.217290886599546	-0.975972684905447\\
0.890328824316757	0.237773107148889\\
0.486655539562131	0.224851430154887\\
-0.669511310304501	0.38367621049205\\
0.264428088836274	-0.542252122670321\\
0.0464364713210055	0.53117914898094\\
0.778359177610901	0.626374543438611\\
-0.818494466740739	0.179892172639379\\
-0.362861739015332	-0.8645762626419\\
-0.206588014034522	-0.54138179735521\\
0.54730142353371	-0.649608211186163\\
0.65163680674891	0.279302585826271\\
-0.282470799206395	0.438253145716109\\
0.287415007582748	-0.46325561637868\\
0.0896410758885757	-0.685518158782329\\
-0.785528265524015	-0.204355217485265\\
};
\end{axis}

\begin{axis}[%
width=1.227\fwidth,
height=0.92\fwidth,
at={(-0.16\fwidth,-0.101\fwidth)},
scale only axis,
xmin=0,
xmax=1,
ymin=0,
ymax=1,
axis line style={draw=none},
ticks=none,
axis x line*=bottom,
axis y line*=left
]
\end{axis}
\end{tikzpicture}%

%% file: Figures2/monotonicity_ctrs1.tex
%
%
\definecolor{mycolor1}{rgb}{1.00000,0.00000,1.00000}%
\begin{tikzpicture}

\begin{axis}[%
width=0.951\fwidth,
height=0.951\fwidth,
at={(0\fwidth,0\fwidth)},
scale only axis,
xmin=-1,
xmax=1,
ymin=-1,
ymax=1,
axis background/.style={fill=white},
axis x line*=bottom,
axis y line*=left
]
\addplot [color=mycolor1, draw=none, mark size=0.3pt, mark=*, mark options={solid, mycolor1}, forget plot]
  table[row sep=crcr]{%
-0.477063721561519	-0.0356784483063968\\
0.995694230559532	0.0820489556321062\\
0.499708434312526	0.0163540335189059\\
-0.593760745931496	-0.804556919418141\\
-0.417480132594845	0.908541013531429\\
0.0468743365984707	-0.497444893992345\\
-0.0168043331560042	0.499462582845728\\
0.521379906483603	-0.85331698999161\\
-0.991615615238376	0.126990738965544\\
0.510307731496302	0.859927265923844\\
0.00946143316531067	0.00102920413774887\\
-0.368259810177729	0.33744940482927\\
0.366460227988413	-0.336665228004399\\
-0.0110363580365846	-0.500237392074794\\
-0.317085291889058	-0.385686619015848\\
0.341668571381305	0.362411882521911\\
0.0185934279127415	0.499908229834103\\
-0.0541367424112735	-0.998230532922071\\
-0.131250073866441	0.22649729032455\\
-0.182026913123815	-0.180203970620125\\
0.259767135714942	-0.096758190109407\\
-0.503531235586562	0.011209324028719\\
0.899166023448959	-0.435822183700642\\
0.172919351587547	0.213183724252585\\
-0.797126473481059	0.603675440853859\\
0.502704945765365	0.00275494368783225\\
0.0752607949026172	-0.284319666791635\\
-0.473947043353501	0.159073006729269\\
0.170804136488288	0.469578761178675\\
-0.146890973134374	-0.476467215889401\\
-0.288401527712531	0.0512711274720805\\
-0.91307985002485	-0.407358194989433\\
0.464078787050052	-0.184556271350984\\
-0.439782451462922	-0.234449995152673\\
-0.215048019925061	0.451030283492778\\
0.843011718330751	0.537492523273989\\
0.453305548136597	0.20752000060722\\
0.225490113712657	-0.445789198058262\\
0.0641539058960046	0.997682373803379\\
0.323320320970753	0.0901952529553498\\
0.0255323423688985	0.344058280426485\\
-0.105839713407873	-0.33664692898916\\
-0.336370702473208	-0.127294091874327\\
0.246047097303351	-0.266121583980083\\
0.470583695055394	-0.475429841339737\\
-0.296704018357631	0.224859373078364\\
0.0826578353005876	-0.131448716389312\\
-0.45491785603233	0.490940667415001\\
0.156290153212143	0.0498404395411625\\
-0.14561590388923	0.000978960398491324\\
0.0160490378352813	0.169178341056169\\
-0.503471928472812	-0.472005101413994\\
0.497167661705596	0.487015672870751\\
-0.122052759072243	0.369709831510644\\
0.393530375440958	-0.0655544599600508\\
-0.280321278930102	-0.274121988240204\\
0.206240325573879	0.345303481051977\\
-0.0464841816794359	-0.169157620756788\\
0.320043159759777	0.233086800404239\\
0.113355191554892	-0.400606443747227\\
-0.405158164918664	0.064841121424968\\
0.352233970022955	-0.201475691452656\\
-0.246566436142129	0.33882982642948\\
-0.124345250655	0.109959917833451\\
-0.0248543468981781	-0.418444174247274\\
-0.479179054940418	-0.136787625930732\\
-0.431954286682899	0.250961821203661\\
0.435864347023304	0.0993455264490224\\
0.18420834974772	-0.792054031692334\\
-0.212861983950969	-0.386729276048971\\
-0.123690575124466	0.483486951822729\\
-0.497102361992298	0.0520680379808189\\
0.266920661209433	0.421355445819836\\
-0.243716089282988	-0.064994598572766\\
-0.255025495213824	-0.775802091430925\\
0.49037325068786	-0.0936613006790143\\
0.780180069693323	-0.142526113190679\\
-0.791309087016858	-0.104322573510534\\
0.176657656844163	-0.187857134833282\\
0.0753068182988634	0.43143386053504\\
-0.384025468738015	-0.318120311687333\\
-0.109291415365342	0.776215113563639\\
0.154154219628971	-0.0522768405843306\\
-0.223600883283671	0.146958184790861\\
0.408435602936637	0.288401085796866\\
-0.0783396544340698	-0.0763853469296474\\
0.0856278593191004	0.266802716957514\\
0.264055277435605	0.00841468036815907\\
-0.0563905618966163	-0.495177755315114\\
0.29872269170361	-0.399371607249901\\
-0.36817626538517	-0.0351118381169069\\
-0.747560483969519	0.28537974511709\\
-0.0265473736946111	-0.279751835562121\\
-0.0384465680114505	0.265573198911087\\
0.198709601513455	-0.349193162539369\\
-0.373651035004264	0.163557489166327\\
0.233911450005782	0.131869948222272\\
-0.293023978794582	0.404948783523752\\
0.421121686026525	-0.267782144424558\\
0.759167421068845	0.24132871413685\\
0.0985913398242086	0.12610366447635\\
-0.172949269484731	-0.274913451443279\\
0.140958314581047	-0.476190598760076\\
-0.0408873554485285	0.419688673189503\\
0.243629320566179	0.758116197010146\\
-0.0339669255924604	0.0792203263858997\\
-0.233897447846743	-0.44133454667979\\
0.486471046942885	0.113382830731657\\
-0.360250410934239	-0.215552782608958\\
0.0585604320588813	0.496400656176795\\
0.736618199818287	-0.676246239212807\\
0.379634160328698	0.0154830135398258\\
-0.267772925359041	-0.180469643226695\\
-0.206784921065786	0.262381371142805\\
0.268224866177351	0.297264373853441\\
0.041736332834539	-0.207563303041844\\
0.0119159195414762	-0.0880675057752804\\
0.12211706142737	0.351111115753278\\
-0.337898216263547	-0.941041686347693\\
0.379141240846091	0.174452218934819\\
-0.161062312719347	-0.0949564699300838\\
0.0821997546742104	0.0223270811304812\\
-0.415938495109237	-0.108760449834268\\
0.158729450751354	-0.271462372373457\\
-0.204657665200106	0.0548689461467498\\
0.0268287383889738	-0.356973915164642\\
0.269981284526017	-0.178714119359324\\
-0.104346789655141	-0.2299566992444\\
0.321503209421525	-0.289158850464547\\
-0.986227103417637	-0.165319440455435\\
-0.0626735094817126	0.173752350747757\\
0.23258246054968	-0.972505072914898\\
-0.351953512294189	0.26941499819793\\
-0.113353602601431	-0.413502588567135\\
0.338029837260005	-0.116954565497936\\
-0.303726241657303	0.127014175074884\\
0.2498162592763	0.211508129961721\\
0.418267873947563	-0.136333039590834\\
-0.120245832391476	0.299646452748933\\
-0.630188317050367	0.776199523304063\\
-0.0731704376616689	0.0124667282520199\\
-0.310501421661745	-0.392168457091895\\
-0.191281829042764	0.390793327307739\\
0.379763274642234	0.325262002557032\\
0.173158681323768	0.280477210205155\\
0.986660677747545	-0.160192192030894\\
0.318920406586436	-0.0329414319969981\\
0.0765271857911538	-0.0522409595493263\\
0.449765602132992	-0.0192164288364762\\
0.176517015073951	0.41036085361918\\
0.18397810283504	-0.115705774766057\\
-0.0485673243694906	0.341006424657353\\
0.0486976033198443	-0.434916264087959\\
0.0391821011684481	0.0855699715583751\\
-0.40335492124984	0.295608328037018\\
-0.273387912620266	-0.340097831766723\\
0.27317972347374	-0.343259307230746\\
-0.30122280614417	-0.0240788257115194\\
0.0951116103317438	0.197323978187592\\
0.283096264321699	-0.412934563119652\\
0.305307970793387	0.157664891552605\\
0.112552482476259	-0.20966570612243\\
-0.113612128255803	-0.146648715827525\\
0.341354044089788	0.296919814908173\\
0.165695885797115	0.133134521040605\\
-0.432358101156699	0.00925398306360026\\
-0.303866503937202	0.314063038868192\\
-0.420487830264824	-0.175104955631627\\
-0.787002867347369	-0.616071036353426\\
0.701320991895395	0.711056819686233\\
0.0247816779026806	0.241098179240754\\
0.230768983278043	0.0642126088351025\\
0.210758099006164	-0.0261492837841952\\
-0.446044529126557	0.107771987976878\\
-0.221912013302511	-0.00862875818631337\\
0.118082010431805	-0.335513955681751\\
-0.16172357651138	0.172259871716153\\
0.278803073755042	0.362628369797221\\
-0.168479054131914	0.985397733069966\\
-0.043855556124796	-0.353143217890399\\
-0.338901411025212	-0.28873972270146\\
-0.928366979746132	0.369919288873293\\
0.174662711515634	-0.416984901682272\\
-0.118642432471858	0.429171293888311\\
-0.233807816566047	-0.128417776249768\\
0.121351780699798	0.483370790454524\\
-0.350496563815073	0.0360728797612346\\
-0.173060031903195	-0.340732753308248\\
0.392994364286519	0.234436264215961\\
-0.224178519352807	-0.23481651349604\\
0.0151561661605304	0.437386484790041\\
-0.423903609119731	0.196190546409713\\
-0.492465900495605	-0.0819379455177893\\
-0.183860698633378	0.319679175257004\\
0.409753940743883	-0.210809646890398\\
-0.235872173845897	0.206143083045583\\
0.450940263525252	0.0461282236560925\\
-0.0707026737990164	0.494064574190487\\
0.0180308312886184	-0.149902683171758\\
0.443396637172312	-0.23197667196257\\
-2.04535220336144e-05	-0.499584977918881\\
0.378147561613281	0.093954131506039\\
-0.35972212006561	0.10124982000651\\
-0.119703355390052	0.0541156127947844\\
-0.0230793409633236	-0.221482612623103\\
-0.170337602977517	-0.42398108639862\\
-0.499995975806039	0.00113184812828937\\
0.223263707619257	0.447034923033337\\
-0.206322711507642	0.455843055666749\\
0.496750919929474	-0.0416010436825358\\
-0.299242085356404	-0.0905070642216836\\
-0.557457938512326	-0.247512006747309\\
0.242908066854975	-0.392401619385697\\
0.944873725502128	0.325021173366209\\
-0.413520733937011	-0.280837531156706\\
-0.225846443714965	-0.307180215012337\\
-0.404264572513639	0.292663953599567\\
0.296575745934817	-0.231607268626185\\
-0.0331704414512197	-0.040945839360085\\
0.442483153495544	0.156060517486498\\
0.321674437819382	0.0307323385950802\\
0.370088263368701	-0.267159092757707\\
-0.0907256728213466	-0.284348238828733\\
-0.26604293668089	0.273946345961838\\
-0.282940060839602	-0.411863808268894\\
0.136639311036518	-0.141194156634021\\
-0.181713669573569	0.105005694337967\\
-0.333950444454346	0.370744802138166\\
0.135155537524662	-0.00260285122212101\\
0.0695963413858769	0.376527436090646\\
-0.366121752583458	0.216274816364136\\
0.286904294137226	0.957790108312898\\
-0.249403360065024	0.0947952084466446\\
0.224160616040179	-0.218530247482605\\
0.712968611832449	-0.434005336470976\\
-0.022957906825239	0.128715366423965\\
-0.488321636086247	0.103237069335685\\
-0.308815259760875	-0.223531937787379\\
0.222521326238399	0.260865516105668\\
0.0189714832268224	0.293425266414783\\
-0.168867109795892	0.47042526741662\\
-0.389720518083356	0.733491224782645\\
-0.0762470145759315	-0.454597430204885\\
-0.124655903922977	-0.0492594464134799\\
0.0230735170570333	-0.292450776624613\\
-0.465772378284704	-0.181279889602932\\
-0.423359402897637	-0.0507253848694798\\
0.107459100990592	0.0750612523848289\\
-0.0777450279262752	0.230081753600663\\
0.442712033889421	-0.083010618052477\\
0.0967458769570102	-0.46050571683295\\
-0.0341591325214963	-0.113158753937669\\
0.12752663048578	0.420478064998985\\
0.227032146139617	-0.148539461386821\\
-0.245110767173883	0.397671002451692\\
-0.187353896822507	-0.046944859869428\\
0.0762999077468189	0.319272716415254\\
0.332975610926108	-0.372648991618002\\
-0.287728279080072	0.173436590909207\\
-0.738121061849982	-0.404134055162981\\
-0.0242054452804423	0.207397683170705\\
0.377394397662925	0.323977696670701\\
0.443350310542372	-0.70489572309559\\
-0.362354977387083	-0.0847231453556103\\
0.273602967819067	0.0932195604092112\\
0.493547517503441	0.0641076718667677\\
-0.076114529839397	0.111866553613713\\
0.206675106715156	-0.298786526881021\\
-0.000752344193784671	0.386704464333519\\
0.118436324242213	-0.0906799808006356\\
-0.154419941353885	-0.220924629023479\\
0.265805602708047	-0.0493926957574322\\
0.229979387007991	0.391456318084407\\
0.0748239769851127	-0.36569819634563\\
0.478936084382007	-0.134798786433943\\
0.132885190365224	0.246141091108929\\
-0.103001795258072	-0.488944770454084\\
-0.332480943749853	-0.340112284578604\\
-0.369564656480402	-0.161957200568997\\
-0.113550636418641	0.166097143077022\\
0.211409903068279	0.176475036512378\\
0.315703821620929	-0.16058495106461\\
0.198249619896959	0.0211079684879825\\
0.268541387441239	0.524342907570704\\
0.000425633687396232	-0.460069931628017\\
0.136728469768696	0.173674823557697\\
0.374875557593639	-0.156011252228142\\
0.30332633753596	0.396694000020096\\
-0.191066079747533	-0.461776834789949\\
0.182451562437252	-0.465509694126728\\
-0.0750466650481065	0.382540627921925\\
0.00872220756370723	0.0463594789383128\\
0.431944761468918	0.250521359111199\\
-0.45817435751676	0.199460177085525\\
-0.315199691556444	-0.17153753800888\\
0.0538485161722293	0.147121833244375\\
-0.183943667531989	0.221033128418986\\
0.211909323641124	-0.0763171014036352\\
0.124655408703504	0.299694298342947\\
0.394627034500984	-0.307000987490497\\
-0.456118320478556	0.0499906258732479\\
-0.0750632032923977	-0.388324130181153\\
0.186734586394615	0.0892355659169888\\
-0.388920004746773	-0.249904910387143\\
0.47425020769073	0.157307209213768\\
0.352785426275557	0.136864998936309\\
-0.26402152059762	0.00984294561668753\\
-0.0715266001163384	0.453332961987364\\
-0.662356506158491	0.538711658451775\\
-0.073600177206913	0.296409539283789\\
-0.410562693302018	0.13941356429572\\
0.444246830848409	-0.226699073350305\\
-0.141041678136236	-0.374837123929658\\
0.152173794993862	-0.366802078178715\\
-0.288593018347847	0.361671361488943\\
0.03052794806227	-0.0387300578599332\\
0.31251523354353	-0.0782569767857102\\
-0.158369940636618	0.268157166256509\\
0.0266151361528766	-0.247639480703675\\
0.113702783664328	-0.255035730726142\\
-0.0434454587326956	-0.727400803376547\\
0.278626428569538	0.255174930807947\\
-0.180082596719513	-0.13745795887889\\
0.495959877404096	0.69812499020244\\
-0.0897282663362389	-0.188618168456212\\
-0.168224723442755	0.427908135378194\\
0.0920052498053334	-0.490111921956859\\
0.178129045150577	-0.235091245358112\\
0.364911586261548	-0.0294468825111067\\
-0.1332010011122	-0.296959588308466\\
0.0766629898801612	-0.17321630371634\\
-0.330386021731883	0.180584920296488\\
0.165467896487036	0.363593775928619\\
-0.454114859004167	-0.0919616901896927\\
-0.282326794034814	-0.135031884043002\\
-0.255395874081417	-0.576312566856168\\
0.319155571823567	-0.335273338337151\\
0.0175469913224731	-0.39943102565324\\
-0.0353803169963254	0.0290722512604578\\
0.401763921316868	0.0537126712806544\\
0.339489355368488	0.194696776749079\\
-0.163244602737608	0.0478469146359997\\
-0.386743917110552	0.00838336470468448\\
0.05048432593375	0.198651893692491\\
0.310675531021396	0.325645863860055\\
0.057747695749863	-0.0912418943164799\\
0.275429338057635	-0.298823216275709\\
-0.259204238543156	-0.379697116726792\\
-0.0616959103886348	-0.238274093665392\\
-0.0797031340003602	-0.124703337443937\\
-0.49717493244933	-0.691018214671887\\
0.15924868546839	-0.317526671598566\\
0.711211308466494	0.46858879322952\\
-0.320951348227154	0.0805002870649989\\
0.20261259323769	-0.595768581701087\\
-0.224570690729887	-0.18301771977787\\
-0.0815330296736103	-0.0267504478340084\\
0.26352258810856	0.169578959176948\\
0.284149460102542	-0.131453354199826\\
0.380827042016526	-0.105718971622809\\
-0.257726564990168	0.426588813861874\\
-0.0252043132434305	0.465066386092151\\
0.282436180978919	0.0506565252651439\\
-0.0749687332192466	0.0623465795248643\\
-0.205105844375681	-0.0915514918588898\\
0.219470195389011	0.306794118398223\\
0.264858052085446	-0.420626817847151\\
-0.00739577203445019	-0.323458821236313\\
0.0741545523042704	-0.233321964374928\\
-0.393387103730343	0.252330976664069\\
0.409136208594701	-0.0174355689389298\\
-0.0195732349410445	0.313494702341917\\
0.563264213237926	0.227495656885598\\
0.0204753636970061	0.48359674312109\\
0.401453764245948	0.131529896651541\\
-0.162039188536705	0.357859345396595\\
0.0648642857208628	-0.323728832746914\\
-0.264769487936475	-0.229545566492415\\
-0.658147708203818	0.101058067820113\\
-0.326612326197356	0.00363653269549857\\
0.142283547902991	0.0996925678571452\\
-0.0584164190014711	-0.314093584982585\\
0.356985762648943	0.257871497797147\\
-0.121356687896586	-0.0993247423656594\\
-0.092349201686587	0.334871563632858\\
-0.137036803062585	-0.183427399723978\\
-0.229149907022385	0.303656036600558\\
-0.223287100824032	-0.347552705599235\\
0.0551203806931946	-0.00332967015896446\\
0.167481960419597	0.319263487872783\\
0.0165802062521625	0.118956379982204\\
-0.311291210152813	0.27025455304092\\
-0.32802215806846	-0.0572433126589227\\
-0.343470910442084	0.311904257876779\\
0.357186544197106	0.0540126124324958\\
0.0841629550117269	0.466566113942383\\
0.112833062403517	-0.0346294281334638\\
7.1892880828095e-05	-0.187534221681328\\
-0.185537098854842	0.00627527626308866\\
0.292358677674884	0.205835760517071\\
0.3365819344603	-0.243604823697145\\
0.138376351427644	-0.436115761847436\\
-0.243284732964215	0.0458953611978301\\
0.417594504940882	0.196442348498072\\
0.236829442223733	-0.326366420704699\\
-0.107516563803735	0.261876648814038\\
0.0533871141446904	0.0462387292602147\\
-0.404922002131667	-0.209888150238325\\
-0.261206445126068	0.136235342103988\\
0.0394014434385028	0.401427382248104\\
0.425127814705409	-0.172716672661262\\
-0.211300478599119	0.3524383231653\\
-0.444304253119115	-0.141638747715064\\
0.349752796738851	-0.0693435490221728\\
0.248683181823622	0.341061389201729\\
-0.122685808046018	-0.450132937892118\\
-0.303928517377149	-0.309169967692098\\
0.119157928373564	-0.296869388880408\\
0.618858675398881	-0.222752851602308\\
0.0738393286113481	-0.409948530541997\\
0.712132091260242	0.0413307694042129\\
-0.202924898638451	0.180379991198752\\
-0.367169376227922	-0.339296583446217\\
0.172433830944624	-0.0128569552479045\\
0.0718621143307003	0.232502709575646\\
-0.853837813288616	0.103913585232333\\
-0.158650703864631	0.136659050514836\\
0.201709504741679	-0.39099688222829\\
-0.334891243902418	-0.249695691741504\\
0.110052384268039	0.387991480112016\\
-0.344174508037327	0.137546359300741\\
-0.232372567208809	-0.271680450222235\\
0.309235538669564	0.278881238459118\\
-0.0115484904136738	-0.375663639002356\\
0.137890598487481	-0.181656569066535\\
-0.38435848180789	-0.127378960117649\\
-0.129374896434695	-0.256892998896081\\
-0.254091223931961	0.239073882792496\\
-0.470245570527445	0.00831064042883889\\
0.427360150894996	0.0176546773498647\\
-0.182043833980986	0.620429548428481\\
-0.263782681913331	-0.0313156229555751\\
0.283294198325969	0.128361736292096\\
-0.0387039474390662	-0.459169568544859\\
-0.399617266294133	0.104849032770569\\
0.213776458895861	0.218554433010359\\
0.0711361122585183	0.0983245253883143\\
-0.110186134329143	0.0129482370671834\\
0.157336076521996	-0.0932019692272812\\
-0.025680606365855	-0.0766784439607835\\
0.202652797400642	-0.259188903937903\\
0.14397268460979	0.452450959240592\\
0.064807489697722	0.685044104726036\\
0.0534325831161222	0.279366130133111\\
0.313303296761723	-0.198956349456631\\
-0.147451760597716	0.397419213091517\\
-0.257746225910982	-0.0999634249380756\\
0.25920620985462	-0.220097920402815\\
0.179660880483113	-0.153425141150941\\
-0.0962113991803855	0.199382634556555\\
-0.441012078366473	0.157627700893241\\
0.469136675946535	-0.0586772783672949\\
0.298318285449425	-0.00105951953632344\\
0.120354116837905	0.0380538366474834\\
-0.188696666737597	-0.23966083093894\\
0.356817152777194	-0.302530082061235\\
-0.0247784831063425	0.16676010082776\\
0.0447265037390157	-0.126838045730928\\
-0.336466387228328	0.232290741751844\\
-0.177382499843035	-0.379576252638629\\
0.218645500480447	-0.181354311053254\\
0.451640904244219	-0.119179430905726\\
-0.0344295009636879	0.373611939488971\\
0.0920746409920679	0.161247221212571\\
-0.149902177822772	0.0855182793231333\\
0.172891887109727	0.171749735796394\\
0.381898286745957	-0.23172102014178\\
-0.141706619912021	0.324319676617268\\
-0.283205770387623	0.0896102382204986\\
0.0441924641247358	0.456256089563968\\
-0.396823442806134	-0.0722718098452626\\
-0.207162087969256	-0.421531818703753\\
0.467857446150713	0.011954047017195\\
0.375966860307644	-0.925378269408562\\
-0.188784756213382	-0.310005120933798\\
-0.00639019443153632	0.256875797061181\\
0.143780506777585	-0.227513400548765\\
0.339690901321352	0.0010939054413055\\
-0.0829974773176008	0.417149821775493\\
-0.374330396068164	-0.285729342563323\\
-0.0346159418605532	-0.00543822622637102\\
-0.0675116422447681	-0.421634577804314\\
-0.0157926492348697	-0.14637139737796\\
0.236052730550145	0.0269807402102638\\
0.0566686323210281	-0.466237722819684\\
-0.144935989396486	-0.133517432654335\\
-0.497055911386383	-0.0420317893420021\\
-0.213089315540846	0.418817980179279\\
0.129371976674434	0.210976229433491\\
-0.250781713174657	0.173613237189089\\
0.377653834036666	0.282959712865218\\
0.290201058947303	-0.370699929305034\\
0.228493441488824	-0.109091109895926\\
0.196246417631038	0.127349933163432\\
0.289049834604876	-0.2669711918443\\
0.190330837931998	0.437170911239489\\
-0.368375502483374	0.0668489737302032\\
-0.215691018981222	0.101229661940155\\
0.467435244986501	0.0854961003564256\\
0.0123372243216835	0.207083147579539\\
-0.263233962490399	-0.306030326296343\\
-0.410196591565474	-0.0203520192981965\\
-0.0862736676809193	0.143552595725979\\
0.18759934097353	0.249742512502035\\
0.00070493051811793	0.0852557981951649\\
-0.283885355907658	-0.0611456032378093\\
0.2285846279074	0.0989948826615217\\
0.364083766056299	0.219453393382458\\
-0.158059554079061	-0.032157856259112\\
-0.296211598272335	-0.368213087821047\\
-0.0580123162923825	-0.203510874461646\\
0.244986747105602	-0.0197229840541371\\
-0.267287276556427	0.312799960898197\\
0.426772255869655	-0.051355124806556\\
-0.244474716735006	-0.410927876152569\\
0.0930944716691389	0.873678058392378\\
-0.57314856666428	0.32373196611577\\
0.606395884043346	-0.623433142906228\\
-4.81245862371527e-05	-0.0455972730349994\\
-0.0124670023775066	0.347299875083754\\
-0.139778722535162	0.455354193832771\\
0.0902853956605989	-0.0136580108921127\\
0.203997830215777	-0.427117140408777\\
0.108398183666458	-0.161448264425425\\
-0.324612630030378	0.340852472228703\\
-0.142739532263841	-0.332304610933256\\
0.239262433048517	-0.360076973634582\\
-0.106492403641187	-0.369153104512321\\
0.458422841371553	0.127501633013144\\
-0.46914881787951	0.0784851893390652\\
0.320742728828338	0.122943937567096\\
-0.374998223209351	0.300481171115781\\
0.195611662958419	0.0550442827388957\\
-0.00761882282032511	-0.251682902189507\\
0.384897829172748	-0.185656943414099\\
0.194527877160738	0.382047356785443\\
-0.453058676104839	-0.208722967596582\\
0.884910938709702	0.0238102129258069\\
0.130885389739746	0.136748527117745\\
-0.30334325393698	0.0230201284587068\\
-0.338952049976537	-0.365849493066037\\
0.185551047752727	-0.0570594322605273\\
0.0233037333712431	-0.894056878559757\\
0.0410388796884917	-0.171241163666829\\
-0.192107308435275	0.139218729196883\\
-0.101284086915588	0.463544447009576\\
-0.340022545766346	-0.191841033573878\\
-0.403561332657232	0.222593253571725\\
0.143626837561875	0.391998073783643\\
0.869517639300247	-0.291557929644241\\
-0.0985141847922419	0.0869150123011393\\
0.412867204772227	-0.101099615502307\\
-0.143783759721467	0.197056033118528\\
0.111730258567509	-0.36768895588875\\
-0.438333188679836	-0.241701963858898\\
0.264299689210909	0.390457799041521\\
-0.646765315838699	-0.608528801779475\\
0.310900651555002	0.360432893863679\\
-0.201585098113345	-0.206304601398666\\
0.348771365046627	0.330551471102112\\
0.498698967138197	-0.0141198383157879\\
-0.0570396751578348	-0.27009686529955\\
-0.213205594248059	0.0230903606814385\\
0.0896366718180408	0.34939016100268\\
-0.269351909062416	0.206361016036155\\
-0.18532594138263	0.289363615450742\\
-0.0428339547981524	0.232909065115476\\
-0.41172688148083	0.0351704117741891\\
0.45204847981522	-0.153483776589571\\
-0.861772139572162	-0.259734299458103\\
0.237350402722573	-0.0575381370551902\\
0.043478592344381	0.316723142439499\\
-0.210424893360185	-0.153755219123246\\
0.22170980882381	0.419225522123204\\
-0.397029418470925	0.182997863108567\\
-0.0751717835459067	-0.34713680136699\\
0.0005163016720382	-0.117042050466443\\
-0.0535555569434127	0.139856958887195\\
0.0329657897409026	-0.325679697896897\\
-0.0762276936121948	-0.156645099915028\\
-0.446908494779278	-0.0198057838521546\\
-0.221983005707176	-0.0436366060664612\\
-0.45723017561882	-0.0612516874650428\\
0.0521508112963069	-0.39139903835273\\
0.30084129357656	-0.106476350275106\\
-0.471606097889687	0.127855046905949\\
0.0812177413615731	0.0593927300082693\\
-0.151471026931894	-0.0647223454035082\\
0.147732589850059	-0.399302896086139\\
-0.406684738686024	-0.14878376204236\\
0.346468382874419	-0.162673358996716\\
-0.0413282212573578	-0.389010245663158\\
0.158874806128145	0.0180900754210873\\
0.108497050536313	0.443947626998412\\
-0.309711888335464	-0.265786643355248\\
-0.254007622060873	-0.153212477843726\\
0.0870498549881698	-0.0842413904029025\\
-0.0741356736447891	0.264444085656601\\
0.0516527654924042	-0.265440992622133\\
0.00826211839269964	-0.431236030595454\\
-0.432152260374854	0.0796591659301977\\
-0.443401866712287	0.225870908096329\\
-0.214411785496111	0.230460204672954\\
0.0259480301612014	-0.479535142493686\\
-0.330234011126047	-0.0946311881841599\\
-0.269945137809088	0.861094667241442\\
-0.175536400080239	-0.924190782964649\\
-0.486949663053524	-0.112255679416661\\
0.105776588212278	0.231059249527069\\
0.248703691080757	0.241795489826553\\
0.254565086156445	-0.140734537995073\\
-0.110139891885715	-0.0186222940441434\\
-0.00910848617621474	0.419476197542697\\
-0.144919632209614	-0.410563591190964\\
-0.356961669248	-0.00553994115258361\\
-0.255657908921033	0.369365397552287\\
0.310898859777856	0.0610999037268511\\
0.407090939321552	0.086555664310551\\
-0.377954731895032	0.130413762336591\\
-0.159698876130847	-0.450937277467446\\
-0.387434756683097	-0.187028324331195\\
0.40730094825291	0.164553787744371\\
-0.20495504295647	-0.272707869458388\\
0.279984128691097	0.327119879386895\\
-0.0278070813354592	-0.493539798282522\\
-0.107662947943854	0.39739430894142\\
0.338456127399534	0.164569740859974\\
0.0911497810139326	0.488645355346012\\
0.0102541046122333	-0.216369527846553\\
0.111104421769824	-0.12157611985899\\
-0.0433678295753213	0.497673205572654\\
0.405077078236604	-0.250715658923327\\
-0.294105175892183	-0.195659899320021\\
0.0541245334416647	0.353760105431321\\
0.104191208719565	-0.43169721292008\\
-0.237785846863243	0.271507628036785\\
-0.0588030523891101	0.203632816469333\\
0.240360327576793	-0.294955977323232\\
0.0800524080750726	-0.201761470493328\\
0.140724503418819	0.27491697900908\\
0.0335544709441844	0.0252805031877612\\
0.042404262705154	-0.0649547383300362\\
0.244650962446459	0.280120547033506\\
-0.244156571567771	-0.209273655306461\\
-0.325871938197659	0.0455604279470958\\
-0.0574508856509903	-0.0564542699004151\\
0.0928174396025137	-0.313565466412202\\
0.46322231072315	0.181322861226928\\
0.287293040721072	-0.0340132185173259\\
-0.3517296080525	-0.314051218348883\\
-0.312702484077158	0.202085427470996\\
-0.0440000738065007	0.105567425848123\\
0.283107762553898	-0.0721138515771111\\
-0.998426042677975	-0.0193161129626742\\
-0.161793095633443	0.240168250129934\\
-0.0920897274903039	-0.257162336287047\\
-0.22016770427793	0.381947592560525\\
-0.0468888019830542	0.300823971585874\\
0.348522207187785	0.102028093770702\\
-0.124209806910858	0.141358630683189\\
-0.0616791174837075	-0.0970795337682118\\
-0.182473452764138	0.0729025726359083\\
0.400110209428031	0.263410382932271\\
-0.427454849087595	-0.258021130423325\\
0.0495394958012076	0.119307073987306\\
-0.31407962467802	0.387152219706968\\
0.094030412117474	0.295829829714384\\
-0.337149862971787	-0.0275154142375615\\
0.429570902955432	0.0674086691219533\\
-0.313837613827757	0.15236288511947\\
0.141347827074087	0.328473683744374\\
0.120639852941811	-0.0634639542015187\\
0.299728018581051	-0.313061209355382\\
0.162755977584832	-0.44968321262006\\
-0.309612452874473	-0.131420572092126\\
0.0834891395324582	0.404807402783322\\
-0.156815736415724	0.296389460284659\\
-0.158314395256299	-0.163322295561073\\
0.0261887828683236	0.373327308068546\\
-0.188681427482457	0.445805234361292\\
-0.0894049082697286	0.0362483896308829\\
0.176492294786627	-0.295602702437431\\
-0.805844023318636	0.452049761997455\\
0.426218119134415	0.22675096673581\\
0.44189314238565	-0.200250549813961\\
-0.385502016386558	0.318341701413205\\
0.200768942189483	0.286747201142213\\
0.233527182717592	-0.421538729534954\\
-0.23016890683865	0.072213945181933\\
-0.476744706875635	0.160908239551468\\
};
\end{axis}
\end{tikzpicture}%

%% file: Figures2/mat0_P_decay.tex
%
%
\definecolor{mycolor1}{rgb}{1.00000,0.00000,1.00000}%
\begin{tikzpicture}

\begin{axis}[%
width=0.951\fwidth,
height=0.75\fwidth,
at={(0\fwidth,0\fwidth)},
scale only axis,
xmode=log,
xmin=1,
xmax=1000,
xminorticks=true,
ymode=log,
ymin=0.12172745941799,
ymax=1,
yminorticks=true,
axis background/.style={fill=white},
axis x line*=bottom,
axis y line*=left,
legend style={legend cell align=left, align=left, draw=white!15!black}
]
\addplot [color=blue]
  table[row sep=crcr]{%
1	1\\
2	0.950796343989871\\
3	0.814883209292226\\
4	0.730245315866961\\
5	0.674729492919343\\
6	0.556367732319098\\
7	0.516338216154256\\
8	0.507737751615695\\
9	0.498421931925036\\
10	0.435293013361085\\
11	0.408370142044938\\
12	0.406496628650297\\
13	0.398300796661867\\
14	0.397876185306692\\
15	0.38188916854815\\
16	0.378459101532411\\
17	0.357995150967399\\
18	0.354879522597859\\
19	0.343483998270682\\
20	0.331370139160541\\
21	0.331025078756519\\
22	0.319888761984585\\
23	0.311262673524524\\
24	0.309562003738994\\
25	0.306619990258533\\
26	0.302789059466092\\
27	0.292300786351718\\
28	0.292191964823111\\
29	0.289213243833149\\
30	0.288924715068213\\
31	0.288457078733667\\
32	0.284984711605359\\
33	0.278069592108396\\
34	0.277667153837643\\
35	0.276872375339208\\
36	0.275845956230997\\
37	0.274277858785425\\
38	0.269099997038668\\
39	0.265478296888394\\
40	0.263356640099394\\
41	0.259918645297971\\
42	0.256129682701421\\
43	0.255979621586787\\
44	0.254020696378078\\
45	0.253754386877877\\
46	0.248958698692978\\
47	0.248360164105351\\
48	0.248218013957623\\
49	0.248101639040955\\
50	0.247991348892144\\
51	0.242873219434764\\
52	0.2427409606152\\
53	0.236925133261456\\
54	0.234723889018913\\
55	0.231757417698291\\
56	0.231546136104476\\
57	0.230318843115649\\
58	0.22905165405906\\
59	0.226824919960336\\
60	0.226770641860985\\
61	0.225746234688907\\
62	0.225423855636618\\
63	0.22354348666181\\
64	0.222410599887325\\
65	0.221908573993921\\
66	0.221656021757746\\
67	0.219986946070327\\
68	0.219955098708508\\
69	0.218533550977478\\
70	0.21733140675431\\
71	0.215642659969826\\
72	0.212059501003185\\
73	0.212045554034165\\
74	0.211391429652399\\
75	0.20894074540194\\
76	0.208506655094503\\
77	0.207769969742762\\
78	0.206764504238103\\
79	0.206724560843517\\
80	0.205930245933054\\
81	0.205734125194748\\
82	0.205358568236537\\
83	0.205311895572852\\
84	0.204639866081818\\
85	0.204177091403785\\
86	0.203897786397317\\
87	0.20216560368608\\
88	0.202040702369004\\
89	0.201255755716189\\
90	0.200480433152474\\
91	0.200390511681066\\
92	0.200130608885603\\
93	0.19960515657834\\
94	0.198464782113048\\
95	0.198120002570499\\
96	0.197010703812468\\
97	0.196969812505207\\
98	0.196716588965587\\
99	0.19637864251189\\
100	0.195692141748584\\
101	0.193642386247232\\
102	0.193181546558932\\
103	0.192934666990157\\
104	0.192618338570917\\
105	0.191903191284713\\
106	0.191752468633312\\
107	0.191719428573983\\
108	0.191663972035762\\
109	0.190389402873768\\
110	0.190170501323851\\
111	0.189703315786899\\
112	0.189346938125278\\
113	0.18899676679825\\
114	0.188196000828133\\
115	0.188165756446878\\
116	0.186813847934534\\
117	0.186676284107581\\
118	0.186503007174643\\
119	0.186448236898598\\
120	0.185028414521749\\
121	0.184836468340341\\
122	0.18394888346655\\
123	0.182606837659607\\
124	0.181751978556411\\
125	0.181324837873376\\
126	0.181307778801078\\
127	0.181095399379685\\
128	0.18102554185094\\
129	0.180627622280066\\
130	0.180123055167908\\
131	0.179669755448942\\
132	0.179543266338201\\
133	0.179522025203544\\
134	0.17936959095032\\
135	0.177522691157993\\
136	0.177493543085988\\
137	0.177468629812779\\
138	0.176572247166966\\
139	0.176452586331577\\
140	0.175926337296257\\
141	0.175723505743144\\
142	0.175510895323783\\
143	0.175212485315473\\
144	0.174144890693905\\
145	0.174031805800257\\
146	0.173906943028746\\
147	0.173015112871875\\
148	0.172172353529969\\
149	0.171904794483693\\
150	0.171726972624445\\
151	0.171520551979691\\
152	0.170898569908672\\
153	0.170714801753272\\
154	0.170373084224586\\
155	0.170034140232031\\
156	0.169945624224181\\
157	0.169127215893641\\
158	0.168749430475311\\
159	0.168728262928447\\
160	0.168685933560726\\
161	0.168351083401\\
162	0.16799155328927\\
163	0.167761503963284\\
164	0.167136988578211\\
165	0.166696610444926\\
166	0.166283853482476\\
167	0.166255480406653\\
168	0.165757648673246\\
169	0.165320661424581\\
170	0.16522473499152\\
171	0.165108529345274\\
172	0.165024107882171\\
173	0.164888172429523\\
174	0.164547367593478\\
175	0.164535639581706\\
176	0.163721884017858\\
177	0.163542345380186\\
178	0.163163875619705\\
179	0.162968052453175\\
180	0.16288455314278\\
181	0.162873207393639\\
182	0.162494469399582\\
183	0.162383633260218\\
184	0.162378223280671\\
185	0.16234278457809\\
186	0.162326299899762\\
187	0.162121268271946\\
188	0.161905705250907\\
189	0.16187589980511\\
190	0.161388117143843\\
191	0.160947770589398\\
192	0.16055155213102\\
193	0.160544041371296\\
194	0.160116950657687\\
195	0.159904138862724\\
196	0.159615737403488\\
197	0.158959350825673\\
198	0.157470124105799\\
199	0.157386518641854\\
200	0.157345624805586\\
201	0.157233479155776\\
202	0.156794161903707\\
203	0.156657085910506\\
204	0.156483285740781\\
205	0.156446389459472\\
206	0.156399992322275\\
207	0.156174150439538\\
208	0.156120497246105\\
209	0.15567917369439\\
210	0.155574797299244\\
211	0.155529313001561\\
212	0.155261759375534\\
213	0.155186837602549\\
214	0.154966560420872\\
215	0.154563579418383\\
216	0.154168385590894\\
217	0.153943496136684\\
218	0.153205587701298\\
219	0.15319288462303\\
220	0.152962406295343\\
221	0.152825867683393\\
222	0.152576606749916\\
223	0.152379992617105\\
224	0.152315545057714\\
225	0.152028964871231\\
226	0.151709627104335\\
227	0.150950700812084\\
228	0.15092313576066\\
229	0.150919909366609\\
230	0.150835165883494\\
231	0.150566656745353\\
232	0.150328204313326\\
233	0.150150358707193\\
234	0.150046371098256\\
235	0.149981174524281\\
236	0.14953540142215\\
237	0.149473152899273\\
238	0.149329528238602\\
239	0.148776987945126\\
240	0.148466403096353\\
241	0.148392424545476\\
242	0.148389902923962\\
243	0.14829758057238\\
244	0.14775824335865\\
245	0.147701577437393\\
246	0.147518735504362\\
247	0.147512427368302\\
248	0.147496263854245\\
249	0.147114872643359\\
250	0.147094402986256\\
251	0.146756150638013\\
252	0.146655166972522\\
253	0.14640129386552\\
254	0.146311133925591\\
255	0.145971739281135\\
256	0.145931814932181\\
257	0.145889576445664\\
258	0.145883309913564\\
259	0.145496630242574\\
260	0.145482640256848\\
261	0.145319954838684\\
262	0.145301299040046\\
263	0.145205484417659\\
264	0.144993345188921\\
265	0.144864830509088\\
266	0.144662416030816\\
267	0.144660425173314\\
268	0.144652096896366\\
269	0.144467050270442\\
270	0.144205088431708\\
271	0.14418326376692\\
272	0.143866467315898\\
273	0.143819941387829\\
274	0.14377021329959\\
275	0.143673802917599\\
276	0.143513760014869\\
277	0.143504989911934\\
278	0.143443680016834\\
279	0.143382693838636\\
280	0.143337307070591\\
281	0.142932050146853\\
282	0.142838881012674\\
283	0.142774369633269\\
284	0.142732173488299\\
285	0.142595580439634\\
286	0.142570045708496\\
287	0.142538227146953\\
288	0.142495761750421\\
289	0.142467202254868\\
290	0.142223530653428\\
291	0.142006735662632\\
292	0.141980122092045\\
293	0.141838664875938\\
294	0.141797204207502\\
295	0.141632132695965\\
296	0.141501859661469\\
297	0.141499744742701\\
298	0.141174845497234\\
299	0.141080849577125\\
300	0.141022315149981\\
301	0.140848127387528\\
302	0.140803186513164\\
303	0.140801197035198\\
304	0.140794435897934\\
305	0.140652660475979\\
306	0.140594137373474\\
307	0.140347712303918\\
308	0.140334305894227\\
309	0.140152528061175\\
310	0.139949848822859\\
311	0.139912118391572\\
312	0.139859269768459\\
313	0.139768440319331\\
314	0.139618870241203\\
315	0.139608647449025\\
316	0.139598375123614\\
317	0.139382052337228\\
318	0.139378799245\\
319	0.139068067428807\\
320	0.138898675867068\\
321	0.138693868154075\\
322	0.138551874118207\\
323	0.138397054309545\\
324	0.138211437381393\\
325	0.138185307337424\\
326	0.138146765615596\\
327	0.137912941509665\\
328	0.137706277330728\\
329	0.137620615873408\\
330	0.137516351356236\\
331	0.137408346674013\\
332	0.13736048666155\\
333	0.1371989080805\\
334	0.137032597366834\\
335	0.136979345909639\\
336	0.136960619644824\\
337	0.136862580442521\\
338	0.136820705295017\\
339	0.136757159869437\\
340	0.136604812407914\\
341	0.13624347597219\\
342	0.13612553145213\\
343	0.13593663266878\\
344	0.135864684042974\\
345	0.135762752169906\\
346	0.135565334057096\\
347	0.135542728098085\\
348	0.135364466842048\\
349	0.135052687026031\\
350	0.135018310987943\\
351	0.135009627242445\\
352	0.134985259695226\\
353	0.134956186704961\\
354	0.134701377074438\\
355	0.134676422061055\\
356	0.134606024565874\\
357	0.134600347553486\\
358	0.134446541399134\\
359	0.134282983476537\\
360	0.134203253031232\\
361	0.133597409779046\\
362	0.133590623579309\\
363	0.133514214367957\\
364	0.133424927088601\\
365	0.133356299067035\\
366	0.133297306695065\\
367	0.133275154210539\\
368	0.133269592381621\\
369	0.133236070204177\\
370	0.132947487585946\\
371	0.132945039855768\\
372	0.132903793576474\\
373	0.132899037794153\\
374	0.132890719804032\\
375	0.132592136357225\\
376	0.132569226905107\\
377	0.132440808544206\\
378	0.132261158406134\\
379	0.132239752495037\\
380	0.132156733514788\\
381	0.132109518972577\\
382	0.132056340020041\\
383	0.132031778003963\\
384	0.132020847641739\\
385	0.131842100494781\\
386	0.131785036840638\\
387	0.131746001606234\\
388	0.131648852376691\\
389	0.131365483788511\\
390	0.13127919772273\\
391	0.131065983263384\\
392	0.130975560551875\\
393	0.13080517776407\\
394	0.130804290820442\\
395	0.130594738622714\\
396	0.130558385952344\\
397	0.130416155369366\\
398	0.130370419396967\\
399	0.130170360642926\\
400	0.130168115657181\\
401	0.130121338036461\\
402	0.130092936129217\\
403	0.129934047020008\\
404	0.129831703629685\\
405	0.129817890171898\\
406	0.129724200765545\\
407	0.1296853341352\\
408	0.129519437792099\\
409	0.129504774629623\\
410	0.129478986328658\\
411	0.129235021735844\\
412	0.128985123425858\\
413	0.128912489608311\\
414	0.128848797690386\\
415	0.128580547558486\\
416	0.128533183392514\\
417	0.128409750285957\\
418	0.128316751474159\\
419	0.128306521606215\\
420	0.128266639083655\\
421	0.128242920832202\\
422	0.12814718406389\\
423	0.127904607082499\\
424	0.127599816710865\\
425	0.127565905700823\\
426	0.127478295286511\\
427	0.127399295005225\\
428	0.127395475633836\\
429	0.127293494293171\\
430	0.127217550197837\\
431	0.127110152984741\\
432	0.126825894940118\\
433	0.126696029584609\\
434	0.126648785400534\\
435	0.126529995002678\\
436	0.12651771824179\\
437	0.126486505569673\\
438	0.126450351489742\\
439	0.126300254335418\\
440	0.126281578746607\\
441	0.126214238336462\\
442	0.126129402565574\\
443	0.126102050082102\\
444	0.125917057471356\\
445	0.125795880800637\\
446	0.125792941786833\\
447	0.125773222239769\\
448	0.125729024167436\\
449	0.125721523650519\\
450	0.125711051478912\\
451	0.125705166336403\\
452	0.125459506300486\\
453	0.125456616813016\\
454	0.125409399348565\\
455	0.125399828203609\\
456	0.125354037780656\\
457	0.125221554252944\\
458	0.125095289365099\\
459	0.1250889861246\\
460	0.125017328018828\\
461	0.125012069800757\\
462	0.124835433454849\\
463	0.124769990458574\\
464	0.1244372556575\\
465	0.124420130769641\\
466	0.124236636801986\\
467	0.124173630765065\\
468	0.124063224477776\\
469	0.123745716540102\\
470	0.12365487262703\\
471	0.12358901355056\\
472	0.123516968530514\\
473	0.123484346673558\\
474	0.123475726678591\\
475	0.123210635822196\\
476	0.123176672717356\\
477	0.123171646625388\\
478	0.123161568629983\\
479	0.12313243371201\\
480	0.123074736272342\\
481	0.123071122699771\\
482	0.122975714594042\\
483	0.122969780788035\\
484	0.122943139140191\\
485	0.12279071808454\\
486	0.1226442650603\\
487	0.122640388705997\\
488	0.122564583039818\\
489	0.122396318709921\\
490	0.122343902412197\\
491	0.122273865488949\\
492	0.122245014978555\\
493	0.122190607715895\\
494	0.122114539549861\\
495	0.122092980554429\\
496	0.122009768485414\\
497	0.121959019259086\\
498	0.12189894718328\\
499	0.121767192601178\\
500	0.12172745941799\\
};
\addlegendentry{$\tilde{\Omega}$}
\addplot [color=mycolor1]
  table[row sep=crcr]{%
1	1\\
2	0.907359729116752\\
3	0.875136338009864\\
4	0.82717035466366\\
5	0.805232327554549\\
6	0.661796146713053\\
7	0.653944839973685\\
8	0.646823717306886\\
9	0.597915635177717\\
10	0.553236189979551\\
11	0.52037376299636\\
12	0.506455755704317\\
13	0.504356272984807\\
14	0.496250676407639\\
15	0.481710334050695\\
16	0.479410314035688\\
17	0.471329049720831\\
18	0.461907080576924\\
19	0.461272825609178\\
20	0.454029345049822\\
21	0.451351831857072\\
22	0.436372110391954\\
23	0.422022679882722\\
24	0.419673168549666\\
25	0.41837101729888\\
26	0.408693450907793\\
27	0.401529403361864\\
28	0.381273779860918\\
29	0.379282552184499\\
30	0.376608730509028\\
31	0.373678078777682\\
32	0.372633424387557\\
33	0.369469983274666\\
34	0.363022355520138\\
35	0.362538037223719\\
36	0.360618318344784\\
37	0.358888747399487\\
38	0.357403906928491\\
39	0.353653115041144\\
40	0.34443343855282\\
41	0.340039548838846\\
42	0.338377614385052\\
43	0.337113544682719\\
44	0.336849523946851\\
45	0.335955370418017\\
46	0.334622964932299\\
47	0.334253980821259\\
48	0.333268408491074\\
49	0.333211575439517\\
50	0.332013551657382\\
51	0.329300409205181\\
52	0.328123879016745\\
53	0.327896817079765\\
54	0.326989513715787\\
55	0.326886826196304\\
56	0.326257587197176\\
57	0.326137199474119\\
58	0.325450043606049\\
59	0.324586925704534\\
60	0.324329641246347\\
61	0.320170183547541\\
62	0.319519830434951\\
63	0.31897384488224\\
64	0.318820637307568\\
65	0.315828867717394\\
66	0.312776137413361\\
67	0.310441084895288\\
68	0.307126484473009\\
69	0.306544443203258\\
70	0.305950178391953\\
71	0.305747149913348\\
72	0.304683234040474\\
73	0.302465838478631\\
74	0.297836063745356\\
75	0.29737259558851\\
76	0.297041744772325\\
77	0.29461595042185\\
78	0.293957762986499\\
79	0.29297305175\\
80	0.292636877649419\\
81	0.291870025446621\\
82	0.289675609166945\\
83	0.288066168404214\\
84	0.287663827321652\\
85	0.286571004427976\\
86	0.282439587463232\\
87	0.279156493560023\\
88	0.279063859418265\\
89	0.278444509858018\\
90	0.278126136384332\\
91	0.274908069753494\\
92	0.27422527950177\\
93	0.271263894000811\\
94	0.270426376665699\\
95	0.270137341520376\\
96	0.269522309826063\\
97	0.26896918597068\\
98	0.268711871988691\\
99	0.26855590777105\\
100	0.267584976533029\\
101	0.267028628422734\\
102	0.266649672486328\\
103	0.266575939802921\\
104	0.26455280701157\\
105	0.263037263953541\\
106	0.26228842290887\\
107	0.261820572533038\\
108	0.261616553427973\\
109	0.261162041959897\\
110	0.260173460316343\\
111	0.259370953206733\\
112	0.258986058760738\\
113	0.258457698623582\\
114	0.258326821893565\\
115	0.2568011033128\\
116	0.256761785646428\\
117	0.256682078786134\\
118	0.256326272591943\\
119	0.256173217159023\\
120	0.255991578181596\\
121	0.25587701955187\\
122	0.255184061149676\\
123	0.254755716064671\\
124	0.25457091096716\\
125	0.253939581126435\\
126	0.253903391718213\\
127	0.252899068941047\\
128	0.251914112051091\\
129	0.251315263197815\\
130	0.250745716699386\\
131	0.250028004567539\\
132	0.249601108672438\\
133	0.249521013134052\\
134	0.24915633676764\\
135	0.248897605227377\\
136	0.248390405159723\\
137	0.248376472325089\\
138	0.247707463925782\\
139	0.247315163849941\\
140	0.244421154153394\\
141	0.243907906906669\\
142	0.243563216232206\\
143	0.243410971470462\\
144	0.243309846524385\\
145	0.242153065347107\\
146	0.241868379849088\\
147	0.241478196531382\\
148	0.241043626038388\\
149	0.241034097791632\\
150	0.240840139137112\\
151	0.240457080428708\\
152	0.239353244055451\\
153	0.239309589794929\\
154	0.239304865721962\\
155	0.239198548228233\\
156	0.239091641258354\\
157	0.238892679680272\\
158	0.238613179843259\\
159	0.238537973055654\\
160	0.236794622083792\\
161	0.236366790367805\\
162	0.236194285886491\\
163	0.23613191601569\\
164	0.236072189251969\\
165	0.235989437151132\\
166	0.235602424150957\\
167	0.235464652177335\\
168	0.235213579859504\\
169	0.235126277063404\\
170	0.234529405978384\\
171	0.234493257305512\\
172	0.234479220060367\\
173	0.234470104895408\\
174	0.234326476704108\\
175	0.234194517531555\\
176	0.233458850707839\\
177	0.233451291134297\\
178	0.232999181127\\
179	0.232952179639932\\
180	0.232510242514537\\
181	0.232427774411513\\
182	0.231872131894644\\
183	0.231083731935129\\
184	0.230937226654335\\
185	0.230840600870681\\
186	0.230346961284635\\
187	0.230040828614093\\
188	0.230028099768064\\
189	0.229919026012511\\
190	0.229658250183163\\
191	0.229570806600189\\
192	0.229545941176445\\
193	0.229154518157443\\
194	0.229145264380648\\
195	0.228987208588502\\
196	0.22871710364523\\
197	0.22731094004713\\
198	0.227218455840482\\
199	0.22677818385061\\
200	0.22675424304611\\
201	0.226554470102054\\
202	0.226415946028464\\
203	0.226229599780142\\
204	0.225892611969148\\
205	0.225808700214747\\
206	0.225589750507872\\
207	0.225064207511971\\
208	0.224321261395191\\
209	0.224171363509658\\
210	0.223551980270442\\
211	0.223025757862232\\
212	0.222843937598708\\
213	0.222814150662042\\
214	0.22274332349995\\
215	0.222557784445428\\
216	0.222476167104766\\
217	0.222207214092095\\
218	0.221134729364903\\
219	0.22048804037702\\
220	0.220419410004285\\
221	0.220350471480876\\
222	0.22011552369593\\
223	0.219732558538349\\
224	0.219579485094733\\
225	0.219472730278355\\
226	0.21912909570731\\
227	0.219037484940381\\
228	0.218724893880178\\
229	0.218706998341939\\
230	0.218511280988025\\
231	0.218185285188377\\
232	0.218108025058085\\
233	0.21778800080114\\
234	0.217180219671391\\
235	0.216978152662307\\
236	0.216938994068917\\
237	0.216886413892087\\
238	0.216694676668685\\
239	0.216162377271583\\
240	0.215712537574646\\
241	0.215641815704308\\
242	0.215620705804114\\
243	0.215600858813566\\
244	0.215097703016069\\
245	0.215087528821621\\
246	0.214824146861366\\
247	0.214727226627707\\
248	0.214595445234117\\
249	0.213756708709382\\
250	0.213402084390737\\
251	0.213139325931009\\
252	0.212367902485366\\
253	0.212068551791897\\
254	0.21196877312904\\
255	0.211896071170428\\
256	0.21182634361526\\
257	0.211778930100345\\
258	0.211513725893263\\
259	0.211200940328139\\
260	0.210117777112626\\
261	0.209912648963626\\
262	0.209503424102989\\
263	0.20902397533509\\
264	0.208790845477285\\
265	0.208658047800118\\
266	0.208258206564088\\
267	0.207760991487752\\
268	0.207316687592786\\
269	0.207139368536121\\
270	0.20677409784631\\
271	0.206231226236437\\
272	0.206093134117407\\
273	0.206017027859423\\
274	0.205986684455037\\
275	0.205925695418869\\
276	0.205631203719327\\
277	0.205307389442789\\
278	0.205245926587879\\
279	0.205212890650868\\
280	0.204974379380161\\
281	0.204817128564925\\
282	0.204671057901453\\
283	0.204657011782893\\
284	0.204607841695975\\
285	0.204540131996224\\
286	0.204181764281012\\
287	0.203762201578768\\
288	0.203335107005062\\
289	0.202850776386928\\
290	0.202782864993154\\
291	0.202761037746107\\
292	0.201944460793932\\
293	0.201489279256031\\
294	0.201248366791996\\
295	0.201015412876631\\
296	0.200483843634239\\
297	0.200433020652038\\
298	0.200360991913301\\
299	0.199930602591037\\
300	0.199751363555267\\
301	0.199663190101085\\
302	0.199318773534395\\
303	0.198348736224222\\
304	0.198330252880056\\
305	0.19831165366873\\
306	0.198280218726647\\
307	0.198156497364435\\
308	0.197701046594416\\
309	0.197446234563577\\
310	0.197363337672381\\
311	0.196809185875726\\
312	0.196606502510994\\
313	0.196464789977071\\
314	0.19608127775408\\
315	0.195806004485135\\
316	0.195687275115615\\
317	0.195413349351213\\
318	0.195399544033735\\
319	0.195348103965034\\
320	0.195265412479254\\
321	0.195200072396012\\
322	0.195072182164574\\
323	0.194780360430498\\
324	0.194303942072976\\
325	0.193876693260561\\
326	0.193858599156723\\
327	0.193711581174375\\
328	0.193412966918523\\
329	0.193297892597937\\
330	0.193189537632803\\
331	0.19318856256336\\
332	0.193145254584048\\
333	0.192871401796676\\
334	0.192717291000804\\
335	0.192202995269387\\
336	0.192158946088593\\
337	0.191593496534683\\
338	0.191573990203179\\
339	0.191390526524861\\
340	0.191382152897045\\
341	0.191097036695974\\
342	0.191061500170074\\
343	0.191044907724634\\
344	0.190828427671123\\
345	0.190584517096259\\
346	0.190472972820767\\
347	0.190340969739555\\
348	0.190323672810517\\
349	0.190288501136739\\
350	0.190192013507516\\
351	0.190188082794959\\
352	0.1900584081689\\
353	0.190034077686157\\
354	0.189923781733299\\
355	0.189736569678319\\
356	0.189715936863916\\
357	0.189653951260696\\
358	0.18941866658938\\
359	0.189334500226427\\
360	0.189252724782572\\
361	0.189107596416308\\
362	0.189060982343338\\
363	0.1889307128174\\
364	0.188903101856751\\
365	0.188814514537873\\
366	0.188735931858294\\
367	0.188595067927639\\
368	0.188492864933567\\
369	0.188468202553931\\
370	0.188321097629158\\
371	0.18820404124626\\
372	0.188111722612275\\
373	0.187997503116882\\
374	0.187990369665056\\
375	0.187920224434799\\
376	0.187718127169794\\
377	0.187543092755226\\
378	0.187485861795446\\
379	0.187401608218191\\
380	0.18722392187086\\
381	0.187029004619495\\
382	0.186965383577801\\
383	0.186918652115488\\
384	0.186872834024827\\
385	0.186854334337644\\
386	0.186794364912669\\
387	0.186705144814707\\
388	0.186700042809237\\
389	0.186606418110577\\
390	0.186457692341865\\
391	0.186253118933206\\
392	0.186194614056685\\
393	0.186182049289527\\
394	0.186046085340541\\
395	0.185891556034303\\
396	0.18584982376575\\
397	0.185773856966936\\
398	0.185737226382795\\
399	0.185724619085395\\
400	0.185611974018219\\
401	0.185229295719992\\
402	0.185153733380923\\
403	0.185073015508838\\
404	0.185013730677223\\
405	0.184804352730911\\
406	0.184723995627602\\
407	0.184495122512899\\
408	0.184244241518908\\
409	0.184119671486079\\
410	0.18401002518456\\
411	0.183956622148926\\
412	0.183887678056262\\
413	0.183853404333885\\
414	0.183842172627893\\
415	0.183732937003904\\
416	0.183645560019837\\
417	0.183600302847197\\
418	0.183559516709351\\
419	0.183483113830873\\
420	0.183355402494792\\
421	0.183074825258673\\
422	0.183040597113924\\
423	0.182897725286591\\
424	0.182895594726377\\
425	0.182823170075422\\
426	0.182790584470385\\
427	0.18270267400294\\
428	0.18259123614733\\
429	0.182292663536705\\
430	0.18201568647281\\
431	0.181967415529443\\
432	0.181954487769918\\
433	0.181814001381591\\
434	0.18177741911243\\
435	0.181520896904405\\
436	0.181355548323948\\
437	0.181354106787618\\
438	0.181214635858717\\
439	0.181180400240883\\
440	0.180981166332487\\
441	0.180859676113523\\
442	0.180759458728191\\
443	0.180696632722353\\
444	0.180537142946967\\
445	0.180534303506629\\
446	0.180327071039854\\
447	0.179887230540182\\
448	0.179879958825075\\
449	0.179822035258509\\
450	0.179767344902581\\
451	0.179638651332832\\
452	0.179619265517597\\
453	0.179536280151778\\
454	0.179521417925954\\
455	0.179512954850642\\
456	0.179499530234669\\
457	0.179485810491447\\
458	0.179463422882258\\
459	0.17892948094961\\
460	0.178913003278613\\
461	0.178901716941189\\
462	0.178814123416092\\
463	0.178808798770102\\
464	0.178766727005141\\
465	0.178682860627585\\
466	0.178586297549793\\
467	0.178538094163478\\
468	0.178528654980151\\
469	0.178268305437442\\
470	0.178039717401922\\
471	0.178017651884325\\
472	0.177979663275524\\
473	0.177945587020315\\
474	0.17766929497287\\
475	0.177632154819761\\
476	0.177293219415744\\
477	0.177136098106171\\
478	0.177069233655874\\
479	0.176967931237002\\
480	0.176922221288706\\
481	0.176918375914283\\
482	0.176891357338066\\
483	0.176833145192126\\
484	0.176749117800582\\
485	0.176645491091475\\
486	0.176611359607514\\
487	0.176554711761983\\
488	0.176495804830501\\
489	0.176383312681136\\
490	0.176343535870971\\
491	0.176341769004069\\
492	0.176315265642105\\
493	0.176275967997192\\
494	0.176001042037802\\
495	0.1759267213073\\
496	0.175911495345495\\
497	0.175749969177157\\
498	0.175652677331175\\
499	0.175572508855177\\
500	0.175565377521543\\
};
\addlegendentry{$\Omega$}

\addplot [color=black, dashed]
  table[row sep=crcr]{%
1	0.7\\
2	0.5886274906776\\
3	0.531884979956115\\
4	0.494974746830583\\
5	0.468118213483495\\
6	0.447260172972391\\
7	0.430351707065885\\
8	0.416222490250952\\
9	0.404145188432738\\
10	0.393638927633244\\
11	0.384370340743279\\
12	0.37609947613824\\
13	0.368648271493549\\
14	0.361881207770019\\
15	0.35569292370823\\
16	0.35\\
17	0.344735342353817\\
18	0.339844240195126\\
19	0.335281537810502\\
20	0.331009563151112\\
21	0.32699658440974\\
22	0.323215641660872\\
23	0.319643649847712\\
24	0.316260701263446\\
25	0.313049516849971\\
26	0.309995009988403\\
27	0.307083936355582\\
28	0.304304610361494\\
29	0.301646673621234\\
30	0.299100904477364\\
31	0.296659060190515\\
32	0.2943137453388\\
33	0.292058301402805\\
34	0.289886713596586\\
35	0.287793531830364\\
36	0.285773803324704\\
37	0.283823014887139\\
38	0.281937043245604\\
39	0.280112112134575\\
40	0.278344755068468\\
41	0.276631782927083\\
42	0.274970255630359\\
43	0.273357457302661\\
44	0.271790874426556\\
45	0.270268176567263\\
46	0.268787199315554\\
47	0.267345929151667\\
48	0.265942489978057\\
49	0.264575131106459\\
50	0.263242216516048\\
51	0.261942215225749\\
52	0.260673692645788\\
53	0.259435302792156\\
54	0.258225781263408\\
55	0.257043938892551\\
56	0.255888655998159\\
57	0.254758877168563\\
58	0.253653606521303\\
59	0.252571903387142\\
60	0.251512878374159\\
61	0.250475689772718\\
62	0.24945954026674\\
63	0.248463673920704\\
64	0.247487373415292\\
65	0.246529957507627\\
66	0.245590778694708\\
67	0.244669221060986\\
68	0.24376469829305\\
69	0.242876651846214\\
70	0.242004549249359\\
71	0.241147882535798\\
72	0.240306166789164\\
73	0.239478938794425\\
74	0.238665755785097\\
75	0.237866194278597\\
76	0.23707984899246\\
77	0.236306331834808\\
78	0.235545270963111\\
79	0.234796309905789\\
80	0.234059106741748\\
81	0.233333333333333\\
82	0.232618674608628\\
83	0.231914827889353\\
84	0.231221502260966\\
85	0.230538417981843\\
86	0.229865305928678\\
87	0.229201907075513\\
88	0.228547972003964\\
89	0.227903260442477\\
90	0.227267540832569\\
91	0.226640589920195\\
92	0.226022192370535\\
93	0.225412140404617\\
94	0.224810233456311\\
95	0.224216277848359\\
96	0.223630086486195\\
97	0.223051478568395\\
98	0.222480279312703\\
99	0.22191631969664\\
100	0.221359436211787\\
101	0.22080947063088\\
102	0.220266269786949\\
103	0.219729685363745\\
104	0.219199573696791\\
105	0.218675795584418\\
106	0.218158216108186\\
107	0.217646704462157\\
108	0.21714113379049\\
109	0.216641381032887\\
110	0.216147326777441\\
111	0.215658855120471\\
112	0.215175853532942\\
113	0.214698212733121\\
114	0.214225826565106\\
115	0.213758591882936\\
116	0.21329640843994\\
117	0.212839178783091\\
118	0.212386808152055\\
119	0.211939204382725\\
120	0.211496277814974\\
121	0.211057941204435\\
122	0.21062410963808\\
123	0.210194700453423\\
124	0.209769633161141\\
125	0.209348829370971\\
126	0.208932212720688\\
127	0.208519708808039\\
128	0.208111245125476\\
129	0.207706750997544\\
130	0.207306157520814\\
131	0.206909397506229\\
132	0.206516405423748\\
133	0.206127117349178\\
134	0.205741470913102\\
135	0.205359405251792\\
136	0.204980860960029\\
137	0.204605780045729\\
138	0.204234105886306\\
139	0.20386578318668\\
140	0.203500757938877\\
141	0.203138977383119\\
142	0.202780389970376\\
143	0.20242494532628\\
144	0.202072594216369\\
145	0.201723288512587\\
146	0.201376981160996\\
147	0.201033626150636\\
148	0.200693178483506\\
149	0.200355594145592\\
150	0.200020830078916\\
151	0.199688844154567\\
152	0.199359595146652\\
153	0.199033042707154\\
154	0.198709147341645\\
155	0.198387870385824\\
156	0.198069173982845\\
157	0.197753021061408\\
158	0.197439375314578\\
159	0.197128201179306\\
160	0.196819463816622\\
161	0.196513129092475\\
162	0.1962091635592\\
163	0.195907534437576\\
164	0.195608209599466\\
165	0.195311157551008\\
166	0.195016347416339\\
167	0.194723748921832\\
168	0.194433332380825\\
169	0.19414506867883\\
170	0.193858929259194\\
171	0.193574886109205\\
172	0.193292911746624\\
173	0.193012979206621\\
174	0.192735062029114\\
175	0.192459134246479\\
176	0.192185170371639\\
177	0.191913145386497\\
178	0.191643034730713\\
179	0.191374814290817\\
180	0.191108460389635\\
181	0.190843949776018\\
182	0.190581259614879\\
183	0.190320367477509\\
184	0.190061251332169\\
185	0.189803889534953\\
186	0.18954826082091\\
187	0.18929434429541\\
188	0.189042119425762\\
189	0.188791566033052\\
190	0.188542664284216\\
191	0.188295394684333\\
192	0.18804973806912\\
193	0.18780567559764\\
194	0.187563188745204\\
195	0.187322259296465\\
196	0.187082869338697\\
197	0.186845001255254\\
198	0.186608637719202\\
199	0.186373761687123\\
200	0.186140356393075\\
201	0.185908405342715\\
202	0.185677892307577\\
203	0.185448801319491\\
204	0.185221116665153\\
205	0.184994822880829\\
206	0.1847699047472\\
207	0.184546347284331\\
208	0.184324135746774\\
209	0.18410325561879\\
210	0.183883692609691\\
211	0.183665432649299\\
212	0.183448461883519\\
213	0.183232766670017\\
214	0.183018333574012\\
215	0.182805149364161\\
216	0.18259320100855\\
217	0.182382475670785\\
218	0.182172960706169\\
219	0.181964643657978\\
220	0.181757512253824\\
221	0.181551554402102\\
222	0.181346758188525\\
223	0.181143111872737\\
224	0.18094060388501\\
225	0.180739222823013\\
226	0.18053895744866\\
227	0.180339796685032\\
228	0.180141729613362\\
229	0.179944745470101\\
230	0.179748833644042\\
231	0.179553983673514\\
232	0.179360185243638\\
233	0.179167428183642\\
234	0.178975702464245\\
235	0.178784998195091\\
236	0.178595305622241\\
237	0.178406615125729\\
238	0.178218917217158\\
239	0.178032202537362\\
240	0.177846461854115\\
241	0.177661686059884\\
242	0.177477866169638\\
243	0.177294993318705\\
244	0.177113058760667\\
245	0.17693205386531\\
246	0.176751970116612\\
247	0.176572799110772\\
248	0.176394532554291\\
249	0.176217162262078\\
250	0.176040680155611\\
251	0.175865078261127\\
252	0.175690348707853\\
253	0.175516483726272\\
254	0.175343475646429\\
255	0.175171316896267\\
256	0.175\\
257	0.174829517576519\\
258	0.17465986233783\\
259	0.174491027087523\\
260	0.174323004719274\\
261	0.174155788215375\\
262	0.173989370645294\\
263	0.173823745164267\\
264	0.173658905011913\\
265	0.173494843510877\\
266	0.173331554065506\\
267	0.173169030160543\\
268	0.173007265359854\\
269	0.17284625330517\\
270	0.172685987714867\\
271	0.172526462382756\\
272	0.172367671176908\\
273	0.172209608038493\\
274	0.172052266980644\\
275	0.171895642087346\\
276	0.171739727512342\\
277	0.171584517478062\\
278	0.171430006274571\\
279	0.171276188258535\\
280	0.171123057852215\\
281	0.170970609542467\\
282	0.17081883787977\\
283	0.17066773747727\\
284	0.170517303009839\\
285	0.170367529213156\\
286	0.170218410882798\\
287	0.170069942873357\\
288	0.169922120097563\\
289	0.169774937525433\\
290	0.169628390183426\\
291	0.169482473153618\\
292	0.169337181572896\\
293	0.169192510632156\\
294	0.169048455575525\\
295	0.168905011699593\\
296	0.168762174352654\\
297	0.168619938933975\\
298	0.168478300893056\\
299	0.168337255728926\\
300	0.168196798989433\\
301	0.168056926270557\\
302	0.167917633215733\\
303	0.16777891551518\\
304	0.167640768905251\\
305	0.167503189167787\\
306	0.167366172129485\\
307	0.167229713661275\\
308	0.167093809677712\\
309	0.16695845613637\\
310	0.166823649037258\\
311	0.166689384422234\\
312	0.166555658374438\\
313	0.166422467017728\\
314	0.16628980651613\\
315	0.166157673073294\\
316	0.166026062931961\\
317	0.165894972373441\\
318	0.165764397717092\\
319	0.165634335319815\\
320	0.165504781575556\\
321	0.165375732914813\\
322	0.165247185804153\\
323	0.165119136745739\\
324	0.164991582276861\\
325	0.164864518969475\\
326	0.164737943429751\\
327	0.164611852297631\\
328	0.164486242246388\\
329	0.164361109982193\\
330	0.164236452243696\\
331	0.164112265801604\\
332	0.163988547458273\\
333	0.163865294047301\\
334	0.163742502433133\\
335	0.163620169510667\\
336	0.16349829220487\\
337	0.163376867470396\\
338	0.163255892291215\\
339	0.163135363680242\\
340	0.16301527867898\\
341	0.162895634357156\\
342	0.162776427812377\\
343	0.162657656169779\\
344	0.162539316581689\\
345	0.162421406227289\\
346	0.162303922312287\\
347	0.16218686206859\\
348	0.162070222753984\\
349	0.161954001651819\\
350	0.161838196070698\\
351	0.161722803344171\\
352	0.161607820830436\\
353	0.161493245912035\\
354	0.16137907599557\\
355	0.161265308511408\\
356	0.1611519409134\\
357	0.161038970678601\\
358	0.160926395306994\\
359	0.160814212321218\\
360	0.1607024192663\\
361	0.160591013709393\\
362	0.160479993239513\\
363	0.160369355467285\\
364	0.160259098024689\\
365	0.160149218564814\\
366	0.160039714761607\\
367	0.159930584309638\\
368	0.159821824923856\\
369	0.159713434339358\\
370	0.159605410311154\\
371	0.159497750613942\\
372	0.159390453041879\\
373	0.159283515408364\\
374	0.159176935545813\\
375	0.15907071130545\\
376	0.158964840557088\\
377	0.158859321188924\\
378	0.158754151107328\\
379	0.158649328236643\\
380	0.158544850518982\\
381	0.158440715914027\\
382	0.158336922398839\\
383	0.158233467967659\\
384	0.158130350631723\\
385	0.158027568419069\\
386	0.157925119374358\\
387	0.157823001558683\\
388	0.157721213049398\\
389	0.157619751939934\\
390	0.157518616339625\\
391	0.157417804373535\\
392	0.157317314182289\\
393	0.157217143921904\\
394	0.157117291763619\\
395	0.157017755893736\\
396	0.156918534513456\\
397	0.15681962583872\\
398	0.156721028100052\\
399	0.156622739542399\\
400	0.156524758424985\\
401	0.156427083021155\\
402	0.156329711618224\\
403	0.156232642517334\\
404	0.156135874033307\\
405	0.156039404494498\\
406	0.155943232242659\\
407	0.155847355632792\\
408	0.155751773033017\\
409	0.155656482824434\\
410	0.155561483400985\\
411	0.155466773169326\\
412	0.155372350548691\\
413	0.155278213970766\\
414	0.155184361879561\\
415	0.15509079273128\\
416	0.154997504994202\\
417	0.15490449714855\\
418	0.154811767686379\\
419	0.154719315111448\\
420	0.154627137939105\\
421	0.15453523469617\\
422	0.154443603920818\\
423	0.154352244162466\\
424	0.154261153981659\\
425	0.154170331949959\\
426	0.154079776649838\\
427	0.153989486674563\\
428	0.153899460628096\\
429	0.15380969712498\\
430	0.153720194790243\\
431	0.153630952259287\\
432	0.153541968177791\\
433	0.153453241201606\\
434	0.153364769996661\\
435	0.153276553238856\\
436	0.153188589613973\\
437	0.153100877817575\\
438	0.153013416554913\\
439	0.152926204540831\\
440	0.152839240499673\\
441	0.152752523165195\\
442	0.152666051280467\\
443	0.152579823597792\\
444	0.152493838878612\\
445	0.152408095893423\\
446	0.152322593421687\\
447	0.152237330251748\\
448	0.152152305180747\\
449	0.15206751701454\\
450	0.151982964567614\\
451	0.151898646663005\\
452	0.151814562132221\\
453	0.151730709815159\\
454	0.151647088560027\\
455	0.151563697223267\\
456	0.15148053466948\\
457	0.151397599771345\\
458	0.15131489140955\\
459	0.151232408472713\\
460	0.151150149857311\\
461	0.151068114467609\\
462	0.150986301215582\\
463	0.150904709020854\\
464	0.150823336810617\\
465	0.150742183519572\\
466	0.150661248089852\\
467	0.15058052947096\\
468	0.150500026619699\\
469	0.150419738500108\\
470	0.150339664083394\\
471	0.150259802347868\\
472	0.150180152278884\\
473	0.150100712868772\\
474	0.150021483116774\\
475	0.14994246202899\\
476	0.149863648618306\\
477	0.149785041904343\\
478	0.149706640913391\\
479	0.149628444678352\\
480	0.149550452238682\\
481	0.149472662640331\\
482	0.149395074935687\\
483	0.149317688183521\\
484	0.149240501448927\\
485	0.149163513803271\\
486	0.14908672432413\\
487	0.149010132095246\\
488	0.148933736206465\\
489	0.148857535753688\\
490	0.148781529838817\\
491	0.148705717569701\\
492	0.14863009806009\\
493	0.14855467042958\\
494	0.148479433803563\\
495	0.148404387313176\\
496	0.148329530095258\\
497	0.148254861292291\\
498	0.148180380052363\\
499	0.148106085529111\\
500	0.148031976881679\\
};
\addlegendentry{$\text{n}^{\text{-0.25}}$}

\end{axis}
\end{tikzpicture}%

%% file: Figures2/mat1_P_decay.tex
%
%
\definecolor{mycolor1}{rgb}{1.00000,0.00000,1.00000}%
\begin{tikzpicture}

\begin{axis}[%
width=0.951\fwidth,
height=0.75\fwidth,
at={(0\fwidth,0\fwidth)},
scale only axis,
xmode=log,
xmin=1,
xmax=1000,
xminorticks=true,
ymode=log,
ymin=0.001,
ymax=1,
yminorticks=true,
axis background/.style={fill=white},
axis x line*=bottom,
axis y line*=left,
legend style={legend cell align=left, align=left, draw=white!15!black}
]
\addplot [color=blue]
  table[row sep=crcr]{%
1	1\\
2	0.695308601435107\\
3	0.405293120537644\\
4	0.356709338987698\\
5	0.210831817752758\\
6	0.121995311876435\\
7	0.0876588759743974\\
8	0.081416972861305\\
9	0.0812503478586047\\
10	0.0696172058295399\\
11	0.0476497269816948\\
12	0.0437715892469305\\
13	0.0427897266352408\\
14	0.0349421064938817\\
15	0.0322996246841403\\
16	0.0308915519295107\\
17	0.0307373681247506\\
18	0.029805787772126\\
19	0.0290553228815669\\
20	0.0261032334879502\\
21	0.0233251161520292\\
22	0.0231163393267178\\
23	0.0230376409247239\\
24	0.0211449060916959\\
25	0.0200788173268542\\
26	0.0196243100439911\\
27	0.0180913648261125\\
28	0.0174666606403918\\
29	0.0160378323614761\\
30	0.015372674137515\\
31	0.0148191424583231\\
32	0.0136102313611754\\
33	0.0133710402655329\\
34	0.0133321735713524\\
35	0.0129503143350031\\
36	0.0126488075086121\\
37	0.0124127985710412\\
38	0.0123112209071347\\
39	0.0117346703069173\\
40	0.0113839030783287\\
41	0.0110973617642599\\
42	0.0108732879722607\\
43	0.0104255641269875\\
44	0.0103362127683724\\
45	0.0101099675663879\\
46	0.0101060464461084\\
47	0.00989082716269531\\
48	0.00965561247032261\\
49	0.00946937573674439\\
50	0.00930988970397589\\
51	0.00915901470358357\\
52	0.00912923870070145\\
53	0.00844781766471664\\
54	0.00839152581596898\\
55	0.00836338854877835\\
56	0.00813003362826283\\
57	0.00801052688964631\\
58	0.00790931154948492\\
59	0.00789476702415393\\
60	0.00783457986259111\\
61	0.00779574583073173\\
62	0.00765700488330101\\
63	0.00754406872397687\\
64	0.00743659345559655\\
65	0.0072469510451135\\
66	0.00715792079494989\\
67	0.0071406417335852\\
68	0.00702412237401564\\
69	0.00687516730445704\\
70	0.00673324910382099\\
71	0.00660591923917582\\
72	0.00655919819168515\\
73	0.00653575027526759\\
74	0.00651901183442353\\
75	0.00640918436356824\\
76	0.00621645124530947\\
77	0.00613867569452557\\
78	0.00588517512431276\\
79	0.00568856693315462\\
80	0.00556155832441124\\
81	0.00546003695821562\\
82	0.00545640976427354\\
83	0.00544730938388914\\
84	0.00539153740604172\\
85	0.00538289779530183\\
86	0.00531928225323468\\
87	0.00527280571866853\\
88	0.00520776368182585\\
89	0.00517292145330141\\
90	0.0050625722626726\\
91	0.00505378985542705\\
92	0.00496503145212145\\
93	0.00492040931030548\\
94	0.00489597582909784\\
95	0.00482846952613919\\
96	0.00482103041776987\\
97	0.00477720374711138\\
98	0.00467007070586245\\
99	0.00463551235421869\\
100	0.00461717359126864\\
101	0.0045039459324577\\
102	0.00448356643797454\\
103	0.00447987566104854\\
104	0.00438733241129564\\
105	0.00436821963144065\\
106	0.00426135291329177\\
107	0.00420999407686577\\
108	0.00416926017866092\\
109	0.00415010565085407\\
110	0.00404291630064644\\
111	0.00399699214060304\\
112	0.00398214461921922\\
113	0.00396412625273007\\
114	0.00396330651913336\\
115	0.00396158936959567\\
116	0.00392447631874761\\
117	0.00388828179803073\\
118	0.0038670527816564\\
119	0.0038646688128981\\
120	0.00384398466258038\\
121	0.00380808668331192\\
122	0.00377265355871135\\
123	0.00375499840117518\\
124	0.00374984925052748\\
125	0.0037054815255476\\
126	0.00368828208369715\\
127	0.00368220833686542\\
128	0.00367915512309447\\
129	0.00367424206141954\\
130	0.00366323396310596\\
131	0.00365767091991966\\
132	0.00353594475561912\\
133	0.00352426260675875\\
134	0.00351582915776419\\
135	0.00348215879281443\\
136	0.00346000303417102\\
137	0.00345164563013704\\
138	0.00341223625204369\\
139	0.00340163056280249\\
140	0.00338109749842287\\
141	0.00337223505160164\\
142	0.00335745988601232\\
143	0.00331836631633726\\
144	0.00330939048105125\\
145	0.00327680635676253\\
146	0.00325805214371114\\
147	0.00320258250617019\\
148	0.00319569946215283\\
149	0.00319235715169138\\
150	0.00317009334380265\\
151	0.00316870845187734\\
152	0.00316575664731961\\
153	0.00316428900282302\\
154	0.00315916173468463\\
155	0.00312707046747915\\
156	0.00311949834231351\\
157	0.00311202828777035\\
158	0.00310904960043992\\
159	0.00307896464367983\\
160	0.00305760360019079\\
161	0.00304082701370616\\
162	0.00303844572019493\\
163	0.0030251691526939\\
164	0.00301194275071872\\
165	0.00300898388713589\\
166	0.0029171639216046\\
167	0.00288002930166492\\
168	0.00287380778438157\\
169	0.0028616180959961\\
170	0.002861160851943\\
171	0.00285628660631314\\
172	0.00285440212186238\\
173	0.00285439699832232\\
174	0.00280701878762162\\
175	0.00278860950894037\\
176	0.00278546052956957\\
177	0.00278528395951163\\
178	0.00277055389544087\\
179	0.00274193553691651\\
180	0.00273675490964999\\
181	0.002734906444844\\
182	0.00273260612509724\\
183	0.00271639080212321\\
184	0.00271282459153361\\
185	0.00266845521711159\\
186	0.00266430433187425\\
187	0.00265238240787563\\
188	0.00264464131492445\\
189	0.00262944887916765\\
190	0.00260558573215858\\
191	0.00255920040706367\\
192	0.00254140393973919\\
193	0.00252980197760362\\
194	0.00252524089751974\\
195	0.00251494717429771\\
196	0.00251087082335333\\
197	0.00250252543532518\\
198	0.00250237164132192\\
199	0.00248437440237581\\
200	0.00247838407052477\\
201	0.00247614603601733\\
202	0.00245308596017\\
203	0.00244205874819008\\
204	0.00243241440129797\\
205	0.00242924448318645\\
206	0.0024149729698076\\
207	0.00241429586301012\\
208	0.00240659469196506\\
209	0.00240073026434202\\
210	0.00239524958760069\\
211	0.00236499988562283\\
212	0.00235835890130093\\
213	0.00235168756987612\\
214	0.00233699139037421\\
215	0.00232003488771984\\
216	0.00229643382487568\\
217	0.00228719982561591\\
218	0.00227394711135822\\
219	0.00226335878134581\\
220	0.00225541726479894\\
221	0.00224588731969403\\
222	0.00224495337692313\\
223	0.0022332879425469\\
224	0.00222585633630277\\
225	0.00222359800716305\\
226	0.00222033808924619\\
227	0.00220932751046966\\
228	0.00219887337457567\\
229	0.00219802909991411\\
230	0.00219173489651454\\
231	0.00218310907049397\\
232	0.00217265009228465\\
233	0.00216620403724137\\
234	0.00215604908449877\\
235	0.00215117862095794\\
236	0.00211051044188659\\
237	0.00210830778232541\\
238	0.00208803815581918\\
239	0.00208800256276398\\
240	0.00208602902398537\\
241	0.0020824282586456\\
242	0.00206528873671096\\
243	0.00206144788746081\\
244	0.00206061784127679\\
245	0.00205922382647365\\
246	0.00205849066243104\\
247	0.00204269295479571\\
248	0.00203549191146759\\
249	0.002024020826746\\
250	0.00202050773429641\\
251	0.00201901847065577\\
252	0.00201182033305546\\
253	0.00201122222484097\\
254	0.00198020018189568\\
255	0.00197968451760727\\
256	0.00197753628828236\\
257	0.00197737448806204\\
258	0.00196473382424994\\
259	0.00196413538056458\\
260	0.00196398160975956\\
261	0.0019487565553274\\
262	0.00194475029097626\\
263	0.00193214196118645\\
264	0.00193029563853715\\
265	0.00192965873956776\\
266	0.00192579405292091\\
267	0.00192332396297318\\
268	0.00191826343644186\\
269	0.00191109464966503\\
270	0.00190785423771928\\
271	0.00190783428186109\\
272	0.00189885228154335\\
273	0.00188756418766221\\
274	0.00188291386711232\\
275	0.00187998131215642\\
276	0.00187612836320349\\
277	0.00187288169694734\\
278	0.00187235596152712\\
279	0.00186671001188187\\
280	0.00186623356712244\\
281	0.0018631979834161\\
282	0.00185980525577513\\
283	0.00185957308625959\\
284	0.0018575227067056\\
285	0.00185031359971005\\
286	0.00184568241176277\\
287	0.00184553195544728\\
288	0.00184251137853052\\
289	0.00184168416650482\\
290	0.00184141965693833\\
291	0.00183404435881664\\
292	0.0018245216249182\\
293	0.00182085906594638\\
294	0.00181600652983209\\
295	0.00181316627047203\\
296	0.00180768772669484\\
297	0.00179157345175657\\
298	0.00178987627418284\\
299	0.00176095482516887\\
300	0.00175610541079208\\
301	0.00175020674321258\\
302	0.00174901655383318\\
303	0.00174250690203683\\
304	0.00172372364228596\\
305	0.00171491013908922\\
306	0.00170369284705708\\
307	0.00170344652887584\\
308	0.00170116673785189\\
309	0.0017007238877074\\
310	0.00169826998677505\\
311	0.00169681628481127\\
312	0.00168357895824529\\
313	0.00168231233701615\\
314	0.00167668154650485\\
315	0.00167564324212391\\
316	0.00167207526181715\\
317	0.0016696810871979\\
318	0.00166674732522864\\
319	0.00165800564303321\\
320	0.00165609991265679\\
321	0.00165354593861058\\
322	0.00164650305867377\\
323	0.0016349257827301\\
324	0.0016290915691082\\
325	0.00160767157976255\\
326	0.00160699740790739\\
327	0.00160653378917278\\
328	0.0016015914318649\\
329	0.00160157332051062\\
330	0.001600289653573\\
331	0.00159598768388501\\
332	0.00159510187838663\\
333	0.00159411720266435\\
334	0.00159068471373657\\
335	0.00158785604670964\\
336	0.00158177643021002\\
337	0.00157568621471018\\
338	0.00157516085560256\\
339	0.00157346000658698\\
340	0.00156566358772107\\
341	0.00156051415442527\\
342	0.00155788900057\\
343	0.00155182047896146\\
344	0.00154811841308322\\
345	0.00154226923127141\\
346	0.0015397601254616\\
347	0.00153671167983933\\
348	0.00152554849254535\\
349	0.00151964653006158\\
350	0.00151951708922988\\
351	0.00151419826330852\\
352	0.00151124130253661\\
353	0.00149948066561477\\
354	0.0014970956772955\\
355	0.00149567549746025\\
356	0.00149508981049014\\
357	0.00149249313494799\\
358	0.00149054572753278\\
359	0.00148281721932399\\
360	0.00148262715514952\\
361	0.00148216968755621\\
362	0.00148161816528899\\
363	0.00148019501928822\\
364	0.00147141952799712\\
365	0.00146752346674034\\
366	0.00145759842078063\\
367	0.00145546726029048\\
368	0.00144725293753296\\
369	0.00144342311426883\\
370	0.00144166057244274\\
371	0.00144039142774319\\
372	0.00143978598881309\\
373	0.0014372839228161\\
374	0.00143700365435864\\
375	0.00143394860567977\\
376	0.00143299424431872\\
377	0.00142767922245062\\
378	0.00142704966291207\\
379	0.0014243708629812\\
380	0.00142379833884284\\
381	0.00141258613051129\\
382	0.00140738325498107\\
383	0.0014069171078195\\
384	0.00140391365992743\\
385	0.00140204956504523\\
386	0.00139759165638938\\
387	0.00139181426770371\\
388	0.0013901557918998\\
389	0.0013862153800119\\
390	0.00138595966092559\\
391	0.00138462367226182\\
392	0.00138059277000848\\
393	0.00137940349745335\\
394	0.00137806050279951\\
395	0.00137360118049006\\
396	0.00137339242584546\\
397	0.0013690527543219\\
398	0.00136775084701276\\
399	0.00136194091197605\\
400	0.00135551708869411\\
401	0.00135259742586647\\
402	0.00135016501984211\\
403	0.00134956897177221\\
404	0.00134809567149111\\
405	0.00134340785081146\\
406	0.0013399414910117\\
407	0.00133910888920808\\
408	0.00133897889574803\\
409	0.00133778353069835\\
410	0.00133180114763773\\
411	0.00133076217148969\\
412	0.00133015539615515\\
413	0.00132744970408215\\
414	0.00132160601192494\\
415	0.00130387621943856\\
416	0.001302691832457\\
417	0.00130163471668477\\
418	0.00129872445298502\\
419	0.00129798130371719\\
420	0.00129472014265818\\
421	0.0012930958108983\\
422	0.00129196253235023\\
423	0.00127851079051853\\
424	0.00127380725854868\\
425	0.00127101540663072\\
426	0.00127031697141045\\
427	0.00127019423081446\\
428	0.00127002970856618\\
429	0.00126700433752462\\
430	0.00126658331255196\\
431	0.0012652142441147\\
432	0.00126448488669287\\
433	0.0012621668262857\\
434	0.00126195438598092\\
435	0.00126170733433581\\
436	0.00125269777203908\\
437	0.00124950329718357\\
438	0.00124152056647481\\
439	0.00123705164028659\\
440	0.00123477907056293\\
441	0.0012329461374636\\
442	0.00123248244027718\\
443	0.0012322581577137\\
444	0.00123051465542909\\
445	0.001229059150875\\
446	0.001225307934693\\
447	0.00122016665310814\\
448	0.00121890590979302\\
449	0.00121862370116748\\
450	0.00121814412688663\\
451	0.00121723035184629\\
452	0.00121059288065334\\
453	0.00120932516226935\\
454	0.00120872081507633\\
455	0.00120818104559948\\
456	0.00120661674151612\\
457	0.0012055994774699\\
458	0.00120257777926857\\
459	0.00120222147443888\\
460	0.00120106761931665\\
461	0.00120067242546754\\
462	0.00120022042984621\\
463	0.00119934478101498\\
464	0.00118841381465944\\
465	0.00118808395548167\\
466	0.00118742605560843\\
467	0.00118607597007837\\
468	0.00118005020840954\\
469	0.00117961199072659\\
470	0.0011782431068659\\
471	0.00117682852154436\\
472	0.00117413182221985\\
473	0.00117409944700278\\
474	0.00117407358380764\\
475	0.0011720871132136\\
476	0.00117028315971336\\
477	0.00116840124201353\\
478	0.00116818807214752\\
479	0.00116683546520045\\
480	0.001164672803291\\
481	0.00116372813997142\\
482	0.00116311171065115\\
483	0.00116232941177569\\
484	0.00116000766763407\\
485	0.00115860203940352\\
486	0.00115505950730286\\
487	0.00115490583952129\\
488	0.00115393439660724\\
489	0.00115121515875475\\
490	0.00115060089883646\\
491	0.00114781310105642\\
492	0.00114645482484443\\
493	0.00114569418146878\\
494	0.00114501821119259\\
495	0.00114469017754257\\
496	0.00113838318478258\\
497	0.00113787585177094\\
498	0.001134193832448\\
499	0.00113280523110462\\
500	0.00113237874574337\\
};
\addlegendentry{$\tilde{\Omega}$}

\addplot [color=mycolor1]
  table[row sep=crcr]{%
1	1\\
2	0.621660686407085\\
3	0.568331435653954\\
4	0.435776116512709\\
5	0.403629957193977\\
6	0.199247928996266\\
7	0.195074259692233\\
8	0.187408356915233\\
9	0.145128317234516\\
10	0.114401655851915\\
11	0.104460533302866\\
12	0.093207154309417\\
13	0.082886622318443\\
14	0.0722389797073831\\
15	0.0719802019917469\\
16	0.0708336199760953\\
17	0.0667875298934407\\
18	0.0652735838299285\\
19	0.0621751650211139\\
20	0.0580121338201961\\
21	0.0576389899104046\\
22	0.0557344980198137\\
23	0.053346843586839\\
24	0.0497833057556191\\
25	0.0444888838178297\\
26	0.0429819748770946\\
27	0.0415426514437886\\
28	0.0391573333563109\\
29	0.0387518906381731\\
30	0.035483479317765\\
31	0.0349080625378429\\
32	0.0335504288849486\\
33	0.0333023754680099\\
34	0.0332117785329226\\
35	0.0313730210156106\\
36	0.0298244145685056\\
37	0.029746298679222\\
38	0.0291449121396247\\
39	0.0284038921167201\\
40	0.0272899063897125\\
41	0.0262357598176777\\
42	0.0256321999931245\\
43	0.0254093384967391\\
44	0.024158087577322\\
45	0.0237542751187146\\
46	0.0237452856083441\\
47	0.0236919730792302\\
48	0.0234179505837346\\
49	0.023351940109041\\
50	0.0232858662996166\\
51	0.022277880550473\\
52	0.0220023411474464\\
53	0.0214129845331616\\
54	0.0212134047955477\\
55	0.0211246419435347\\
56	0.0208504616809588\\
57	0.0208485329844174\\
58	0.0207763869553584\\
59	0.0206955371601589\\
60	0.0204745002777088\\
61	0.0201220103950263\\
62	0.0198061929562688\\
63	0.0195441056883742\\
64	0.0190152188202468\\
65	0.018907250691899\\
66	0.018728652126793\\
67	0.0186742872626716\\
68	0.0184127785254787\\
69	0.0183721449832267\\
70	0.0180282553668705\\
71	0.017976752750666\\
72	0.017891235525553\\
73	0.0173930869270278\\
74	0.0171831624023591\\
75	0.0171331720361043\\
76	0.0165161386078068\\
77	0.0161544949888696\\
78	0.0160988293094014\\
79	0.0160898696883249\\
80	0.015891537287043\\
81	0.0151455257687466\\
82	0.015024154333511\\
83	0.0149363802807126\\
84	0.014195282002965\\
85	0.0141745777721587\\
86	0.0139956866189884\\
87	0.0138835091352697\\
88	0.0138030974520977\\
89	0.01377408229984\\
90	0.0134119521773959\\
91	0.0132869097581342\\
92	0.0132118960281854\\
93	0.0129773687969724\\
94	0.0129292618941689\\
95	0.0128367204255221\\
96	0.0126495242453985\\
97	0.0123273378662956\\
98	0.0122205243379035\\
99	0.0120158236587616\\
100	0.0119572880743206\\
101	0.0119464845350665\\
102	0.0119007651604349\\
103	0.0118288556420916\\
104	0.0117744558900918\\
105	0.011760198155016\\
106	0.0116964387119818\\
107	0.0116820689413957\\
108	0.0114210066958799\\
109	0.0114112706052163\\
110	0.0113514386507896\\
111	0.011335389222333\\
112	0.0112801961967832\\
113	0.0111190406623851\\
114	0.0110981637338246\\
115	0.0110958501570599\\
116	0.0110738628440938\\
117	0.0110690079197593\\
118	0.010748483707954\\
119	0.0107071733332171\\
120	0.0105925487513831\\
121	0.0104687215678529\\
122	0.0103814480944058\\
123	0.0103714244414613\\
124	0.0103189699104212\\
125	0.0102571212214069\\
126	0.0102386549846018\\
127	0.0101501171430093\\
128	0.0101458351688898\\
129	0.0101436642083615\\
130	0.00985076197444249\\
131	0.00983493809703975\\
132	0.00978152276938448\\
133	0.00961034472515579\\
134	0.00956972226852353\\
135	0.00950673120919612\\
136	0.00947482623929132\\
137	0.00934551056323549\\
138	0.00934480926994857\\
139	0.00934176010774877\\
140	0.00919023813435921\\
141	0.00913526088455443\\
142	0.00904279878581673\\
143	0.00902337018896185\\
144	0.0090038653462282\\
145	0.00898245851059527\\
146	0.00897829327413151\\
147	0.00896111859290449\\
148	0.0089499243559497\\
149	0.00885871228508262\\
150	0.00877899416396785\\
151	0.00874474354971814\\
152	0.00872993713544473\\
153	0.00867814233894383\\
154	0.00858209533588447\\
155	0.00856539854324303\\
156	0.00854565241091229\\
157	0.00850644401903677\\
158	0.00846587584657175\\
159	0.00845613046920358\\
160	0.00840460494829283\\
161	0.00833445706981093\\
162	0.00831280155100422\\
163	0.00826907785402646\\
164	0.00821827787912436\\
165	0.00820757717256512\\
166	0.00814673729238472\\
167	0.00810807830345772\\
168	0.00809922521829028\\
169	0.00804407181862856\\
170	0.00790442934720928\\
171	0.00785806426364805\\
172	0.00785701227097094\\
173	0.00783032093977746\\
174	0.00781862493623203\\
175	0.0077533820392743\\
176	0.00774784898270796\\
177	0.00773324536541219\\
178	0.00772942375621383\\
179	0.00769824597907199\\
180	0.00763785332688498\\
181	0.00760598538855167\\
182	0.00753763140384041\\
183	0.00747938511880954\\
184	0.00745739257799661\\
185	0.00743298969169692\\
186	0.0074327554386365\\
187	0.00737097450245293\\
188	0.00736317000921276\\
189	0.00731206892418926\\
190	0.00730916098858651\\
191	0.00729951922474055\\
192	0.00728982447502733\\
193	0.0072779107912528\\
194	0.00719769360504706\\
195	0.00719096665969739\\
196	0.00716442587410815\\
197	0.00714695406700961\\
198	0.00709155269881739\\
199	0.00701775168719003\\
200	0.00698946188090506\\
201	0.00698846577003595\\
202	0.00698441920233883\\
203	0.00696564908982269\\
204	0.00694407295236002\\
205	0.00691207637434913\\
206	0.00687690112694578\\
207	0.00686817680915531\\
208	0.00686417866560732\\
209	0.00685094246817969\\
210	0.00681893724276145\\
211	0.00679847834104923\\
212	0.00676021089905883\\
213	0.00672326500045215\\
214	0.00670886654729177\\
215	0.00670739292068952\\
216	0.00670429809513303\\
217	0.00668982725648605\\
218	0.00666531555548644\\
219	0.00665884130791636\\
220	0.00663454050082904\\
221	0.00660979193712677\\
222	0.00657321573650716\\
223	0.00656416980514016\\
224	0.00656176712079882\\
225	0.00655674681012247\\
226	0.006535627068884\\
227	0.00653096345904189\\
228	0.00652634155752516\\
229	0.00651370499716484\\
230	0.00643934461963131\\
231	0.00643081954512693\\
232	0.0064177195706773\\
233	0.00638268879523091\\
234	0.00637849165775053\\
235	0.00636879813143641\\
236	0.00636853040399709\\
237	0.0063618188633739\\
238	0.00633127271835397\\
239	0.00629849217201299\\
240	0.00628303225127267\\
241	0.00623267456413867\\
242	0.006225432898135\\
243	0.00622194792982873\\
244	0.00622135035665064\\
245	0.00619929512749425\\
246	0.00618838364644812\\
247	0.00618635074709031\\
248	0.00613128074353267\\
249	0.00605803207076117\\
250	0.00602740987952\\
251	0.00600205671476952\\
252	0.00598327141866975\\
253	0.00596581725369601\\
254	0.00595195427390916\\
255	0.00595190804057349\\
256	0.00593566305464997\\
257	0.00591076836341931\\
258	0.0059076628096747\\
259	0.0058980149211029\\
260	0.00589263562824553\\
261	0.00588724191816871\\
262	0.00581229002894378\\
263	0.00575029489108414\\
264	0.00574934193075284\\
265	0.00572981112545491\\
266	0.00570363335181517\\
267	0.00569162398416477\\
268	0.0056146395130145\\
269	0.005614206611093\\
270	0.00559721163327414\\
271	0.00553862877277613\\
272	0.00550599102556895\\
273	0.00549451537188755\\
274	0.00547997175795968\\
275	0.005416237404313\\
276	0.00541187477322841\\
277	0.00538196807168721\\
278	0.00537453265566885\\
279	0.00537442357978986\\
280	0.00535300492205451\\
281	0.00528536069690277\\
282	0.00527854802889211\\
283	0.0052599830917887\\
284	0.00525633719157522\\
285	0.00525344929171667\\
286	0.00522893638895498\\
287	0.00522236343207086\\
288	0.00518252167851464\\
289	0.00517121994666884\\
290	0.00516587761464575\\
291	0.00516249644589202\\
292	0.00515576309902953\\
293	0.00511485161683667\\
294	0.00510753202288656\\
295	0.00504871418939101\\
296	0.00501774501537239\\
297	0.00501611486956861\\
298	0.00495547658643525\\
299	0.00494292390407196\\
300	0.0049342527838558\\
301	0.00492936631342748\\
302	0.00487910063871518\\
303	0.00487605252653482\\
304	0.00486535512346324\\
305	0.00486401869701972\\
306	0.00486136796553039\\
307	0.00485576163593841\\
308	0.00485090657155359\\
309	0.00480863135100645\\
310	0.00480804485351602\\
311	0.00478734378168092\\
312	0.00477887113626828\\
313	0.00475816304215708\\
314	0.00475675149654274\\
315	0.00474470563430073\\
316	0.00474305258651435\\
317	0.00473509914209601\\
318	0.00472577690234318\\
319	0.00471324337075799\\
320	0.00469278704037089\\
321	0.00467853796794424\\
322	0.0046580355227749\\
323	0.00464223277274618\\
324	0.00462028452315271\\
325	0.00461124597628706\\
326	0.004604991324959\\
327	0.00460139826670349\\
328	0.00458569856022347\\
329	0.00457199778904658\\
330	0.00456428851358094\\
331	0.00456215065337299\\
332	0.00454863106364067\\
333	0.00454495319786484\\
334	0.00452478197231109\\
335	0.00451029964711238\\
336	0.00450606664108659\\
337	0.00450264029884227\\
338	0.00448988519371348\\
339	0.00448464147091323\\
340	0.00447167918225828\\
341	0.00447167077615443\\
342	0.00441605035278103\\
343	0.00441345577065863\\
344	0.00440077000399485\\
345	0.0043935589758457\\
346	0.0043740034064288\\
347	0.004364670597155\\
348	0.00436272815393234\\
349	0.00435871243677103\\
350	0.00435066685478436\\
351	0.00433531978695343\\
352	0.00433000429057832\\
353	0.0042952582605509\\
354	0.00429421130838207\\
355	0.00429080653177316\\
356	0.00427086782357533\\
357	0.0042585827219094\\
358	0.0042387854577627\\
359	0.00423234026334485\\
360	0.00422913967858207\\
361	0.00422571808556686\\
362	0.00422443591376423\\
363	0.00420670799747694\\
364	0.0041832142449\\
365	0.00417929107879905\\
366	0.00417544397324874\\
367	0.0041670743988965\\
368	0.00416063545025538\\
369	0.00415690578839041\\
370	0.00415284895176489\\
371	0.00414662079765676\\
372	0.00414639618885017\\
373	0.00413316933205987\\
374	0.0041286905561139\\
375	0.00411947950306184\\
376	0.00410213127928398\\
377	0.00410053728505476\\
378	0.00409464406965997\\
379	0.00409210419383469\\
380	0.00409157808270225\\
381	0.00408853986135761\\
382	0.00407909713942664\\
383	0.00407627971892372\\
384	0.0040730242124599\\
385	0.00406430777113995\\
386	0.00405439563013806\\
387	0.00405211481737401\\
388	0.00404963376774834\\
389	0.0040467645640338\\
390	0.00402697918598748\\
391	0.0040244695845281\\
392	0.00401556446325712\\
393	0.00399823285825545\\
394	0.00398869322336305\\
395	0.00397650679357845\\
396	0.00396698932795235\\
397	0.00396655664879835\\
398	0.00395497483604844\\
399	0.00394444962974519\\
400	0.00393945933992445\\
401	0.00393482770727949\\
402	0.00393478553217877\\
403	0.00391505970355051\\
404	0.00390745279342925\\
405	0.00389514371160706\\
406	0.00389491666342944\\
407	0.0038825558761301\\
408	0.00385986549059149\\
409	0.00385881067632632\\
410	0.00385488892114371\\
411	0.00385412229610134\\
412	0.00385363597625219\\
413	0.00385292637976235\\
414	0.00384736467832666\\
415	0.00384107303696924\\
416	0.00383516867874279\\
417	0.00382935482421252\\
418	0.00382311206721227\\
419	0.00381984279762856\\
420	0.00381831585237591\\
421	0.00381793915393482\\
422	0.00379812031993936\\
423	0.00379676830556107\\
424	0.00379081237238322\\
425	0.00378841452198447\\
426	0.00378826595863164\\
427	0.00378589900941503\\
428	0.00377201987899435\\
429	0.00374817768100849\\
430	0.00372251928693264\\
431	0.00370893549202185\\
432	0.00370684003727094\\
433	0.00370592336628918\\
434	0.00369524854258724\\
435	0.00369262000114948\\
436	0.00369070256238189\\
437	0.00368707555887542\\
438	0.00368455344769576\\
439	0.00367812004057568\\
440	0.00367578653692519\\
441	0.00366596216288244\\
442	0.0036647482935527\\
443	0.00366220697401177\\
444	0.00365113017967833\\
445	0.00364163256258411\\
446	0.00363832849317612\\
447	0.00363462459928783\\
448	0.00362430145007698\\
449	0.00360763879051287\\
450	0.00360517386430013\\
451	0.0036029761742857\\
452	0.00360104123063744\\
453	0.00359793521836472\\
454	0.00359289326695657\\
455	0.00358168927197312\\
456	0.00357941648778794\\
457	0.00357478950884807\\
458	0.00356859535093621\\
459	0.00355847003630954\\
460	0.00355712541535141\\
461	0.00355693994192574\\
462	0.00354298261325435\\
463	0.00353610724657031\\
464	0.00353422863090511\\
465	0.0035318445149657\\
466	0.00353004983875104\\
467	0.00352537142148943\\
468	0.00352457571187059\\
469	0.00351717381378123\\
470	0.00350373363171385\\
471	0.00350222112692\\
472	0.00350040623334945\\
473	0.00350035836784963\\
474	0.00349087353987905\\
475	0.00348685566651451\\
476	0.0034806583032245\\
477	0.00347468038314564\\
478	0.00347341067113927\\
479	0.00347274032879209\\
480	0.00346714703678597\\
481	0.00345043164479916\\
482	0.00345010374681221\\
483	0.00343932515280031\\
484	0.00343674386353109\\
485	0.00343590302187359\\
486	0.00343302360487513\\
487	0.00343053071864187\\
488	0.00342106630002061\\
489	0.00341853360637201\\
490	0.00341805912252697\\
491	0.00341022379902005\\
492	0.00340755720165197\\
493	0.00340685457870172\\
494	0.00340521346690307\\
495	0.00340026352456056\\
496	0.00339877778898692\\
497	0.00339873084252934\\
498	0.00338826900698298\\
499	0.00338327120166074\\
500	0.00338226305948838\\
};
\addlegendentry{$\Omega$}

\addplot [color=black, dashed]
  table[row sep=crcr]{%
1	0.2\\
2	0.118920711500272\\
3	0.0877382675301662\\
4	0.0707106781186548\\
5	0.0598139512488488\\
6	0.0521694860024429\\
7	0.0464736160485082\\
8	0.0420448207626857\\
9	0.0384900179459751\\
10	0.0355655882007785\\
11	0.033112005215234\\
12	0.03102016197007\\
13	0.0292127526479376\\
14	0.0276333774323953\\
15	0.0262398622835391\\
16	0.025\\
17	0.0238887433513992\\
18	0.022886301598968\\
19	0.0219768171571563\\
20	0.0211474252688113\\
21	0.0203875727897912\\
22	0.0196885160969818\\
23	0.019042945244303\\
24	0.018444698661672\\
25	0.0178885438199983\\
26	0.017370006648871\\
27	0.0168852374589243\\
28	0.01643090452708\\
29	0.0160041088765937\\
30	0.0156023154621381\\
31	0.0152232971727133\\
32	0.014865088937534\\
33	0.0145259498601723\\
34	0.0142043317809789\\
35	0.0138988530234159\\
36	0.0136082763487954\\
37	0.0133314903489711\\
38	0.013067493664202\\
39	0.0128153815355866\\
40	0.0125743342968294\\
41	0.01234360748505\\
42	0.0121225233096278\\
43	0.0119104632652605\\
44	0.0117068617131882\\
45	0.0115112002849393\\
46	0.0113230029875662\\
47	0.0111418319093442\\
48	0.0109672834412708\\
49	0.0107989849431208\\
50	0.01063659179389\\
51	0.0104797847756227\\
52	0.010328267747241\\
53	0.0101817655713513\\
54	0.0100400222623316\\
55	0.00990279932847818\\
56	0.009769874284767\\
57	0.00964103931598061\\
58	0.00951610007266171\\
59	0.00939487458466453\\
60	0.0092771922790458\\
61	0.00916289309072538\\
62	0.00905182665579575\\
63	0.00894385157860717\\
64	0.00883883476483184\\
65	0.00873665081364208\\
66	0.00863718146294484\\
67	0.00854031508231733\\
68	0.00844594620889971\\
69	0.0083539751220348\\
70	0.00826430745291165\\
71	0.00817685382587915\\
72	0.00809152952845539\\
73	0.0080082542073747\\
74	0.00792695158829329\\
75	0.00784754921702057\\
76	0.00776997822036102\\
77	0.00769417308484493\\
78	0.00762007145179703\\
79	0.00754761392734514\\
80	0.0074767439061061\\
81	0.00740740740740741\\
82	0.00733955292301113\\
83	0.00727313127540321\\
84	0.00720809548579784\\
85	0.00714440065108426\\
86	0.00708200382901318\\
87	0.00702086393098242\\
88	0.00696094162183815\\
89	0.00690219922615895\\
90	0.00684460064053561\\
91	0.00678811125140145\\
92	0.00673269785800538\\
93	0.00667832860015372\\
94	0.00662497289037823\\
95	0.00657260135021504\\
96	0.00652118575030536\\
97	0.00647069895405154\\
98	0.00642111486458324\\
99	0.00637240837480789\\
100	0.00632455532033676\\
101	0.00627753243509446\\
102	0.00623131730943385\\
103	0.00618588835059217\\
104	0.00614122474533604\\
105	0.00609730642465463\\
106	0.00605411403037034\\
107	0.00601162888354585\\
108	0.00596983295457523\\
109	0.00592870883485457\\
110	0.00588823970993521\\
111	0.00584840933406929\\
112	0.00580920200606352\\
113	0.00577060254636307\\
114	0.00573259627529255\\
115	0.00569516899238618\\
116	0.00565830695674361\\
117	0.00562199686835224\\
118	0.00558622585032065\\
119	0.00555098143197139\\
120	0.00551625153274479\\
121	0.00548202444686843\\
122	0.00544828882874995\\
123	0.00541503367905338\\
124	0.0053822483314218\\
125	0.00534992243981138\\
126	0.00531804596640398\\
127	0.00528660917006761\\
128	0.00525560259533572\\
129	0.00522501706187823\\
130	0.00519484365443874\\
131	0.00516507371321383\\
132	0.00513569882465181\\
133	0.00510671081264973\\
134	0.00507810173012841\\
135	0.00504986385096666\\
136	0.0050219896622769\\
137	0.00499447185700514\\
138	0.00496730332683976\\
139	0.0049404771554137\\
140	0.00491398661178628\\
141	0.00488782514419091\\
142	0.00486198637403636\\
143	0.00483646409014925\\
144	0.00481125224324688\\
145	0.00478634494062922\\
146	0.00476173644108023\\
147	0.00473742114996884\\
148	0.00471339361454025\\
149	0.00468964851938932\\
150	0.00466618068210744\\
151	0.0046429850490954\\
152	0.00462005669153476\\
153	0.00459739080151073\\
154	0.00457498268828001\\
155	0.00455282777467707\\
156	0.00453092159365307\\
157	0.00450925978494168\\
158	0.00448783809184624\\
159	0.00446665235814326\\
160	0.00444569852509731\\
161	0.0044249726285825\\
162	0.00440447079630637\\
163	0.00438418924513162\\
164	0.00436412427849193\\
165	0.00434427228389784\\
166	0.00432462973052916\\
167	0.00430519316691022\\
168	0.00428595921866489\\
169	0.00426692458634792\\
170	0.00424808604334974\\
171	0.00422944043387178\\
172	0.0042109846709695\\
173	0.00419271573466047\\
174	0.00417463067009513\\
175	0.00415672658578758\\
176	0.00413900065190425\\
177	0.00412145009860828\\
178	0.00410407221445725\\
179	0.00408686434485262\\
180	0.00406982389053856\\
181	0.00405294830614863\\
182	0.00403623509879832\\
183	0.00401968182672187\\
184	0.00400328609795179\\
185	0.00398704556903929\\
186	0.00397095794381448\\
187	0.00395502097218459\\
188	0.00393923244896897\\
189	0.00392359021276967\\
190	0.00390809214487611\\
191	0.00389273616820277\\
192	0.00387752024625875\\
193	0.00386244238214803\\
194	0.00384750061759938\\
195	0.00383269303202487\\
196	0.00381801774160606\\
197	0.00380347289840688\\
198	0.00378905668951223\\
199	0.00377476733619155\\
200	0.00376060309308639\\
201	0.00374656224742135\\
202	0.00373264311823734\\
203	0.00371884405564668\\
204	0.00370516344010918\\
205	0.00369159968172852\\
206	0.00367815121956833\\
207	0.00366481652098721\\
208	0.0036515940809922\\
209	0.00363848242161001\\
210	0.00362548009127554\\
211	0.00361258566423698\\
212	0.0035997977399771\\
213	0.00358711494265024\\
214	0.0035745359205343\\
215	0.00356205934549749\\
216	0.00354968391247929\\
217	0.0035374083389851\\
218	0.00352523136459427\\
219	0.0035131517504811\\
220	0.00350116827894826\\
221	0.00348927975297244\\
222	0.00347748499576176\\
223	0.00346578285032459\\
224	0.00345417217904941\\
225	0.00344265186329548\\
226	0.00343122080299389\\
227	0.00341987791625871\\
228	0.003408622139008\\
229	0.00339745242459427\\
230	0.00338636774344426\\
231	0.00337536708270764\\
232	0.00336444944591444\\
233	0.00335361385264094\\
234	0.00334285933818375\\
235	0.00333218495324191\\
236	0.00332158976360672\\
237	0.00331107284985908\\
238	0.00330063330707418\\
239	0.00329027024453328\\
240	0.00327998278544238\\
241	0.00326977006665767\\
242	0.0032596312384174\\
243	0.00324956546408023\\
244	0.00323957191986964\\
245	0.00322964979462439\\
246	0.00321979828955482\\
247	0.00321001661800478\\
248	0.00320030400521916\\
249	0.00319065968811673\\
250	0.0031810829150682\\
251	0.00317157294567944\\
252	0.00316212905057957\\
253	0.00315275051121388\\
254	0.00314343661964152\\
255	0.00313418667833762\\
256	0.003125\\
257	0.00311587590736008\\
258	0.0031068137329981\\
259	0.00309781281916238\\
260	0.00308887251759264\\
261	0.00307999218934715\\
262	0.0030711712046337\\
263	0.00306240894264433\\
264	0.00305370479139352\\
265	0.00304505814756006\\
266	0.00303646841633219\\
267	0.0030279350112562\\
268	0.00301945735408817\\
269	0.00301103487464896\\
270	0.0030026670106823\\
271	0.00299435320771577\\
272	0.0029860929189249\\
273	0.00297788560499996\\
274	0.00296973073401568\\
275	0.00296162778130363\\
276	0.00295357622932726\\
277	0.00294557556755955\\
278	0.00293762529236319\\
279	0.00292972490687324\\
280	0.00292187392088218\\
281	0.00291407185072732\\
282	0.00290631821918052\\
283	0.00289861255534013\\
284	0.00289095439452516\\
285	0.00288334327817149\\
286	0.00287577875373033\\
287	0.00286826037456852\\
288	0.002860787699871\\
289	0.00285336029454509\\
290	0.00284597772912677\\
291	0.0028386395796887\\
292	0.00283134542775017\\
293	0.00282409486018875\\
294	0.00281688746915366\\
295	0.00280972285198086\\
296	0.00280260061110983\\
297	0.00279552035400183\\
298	0.00278848169305988\\
299	0.00278148424555021\\
300	0.00277452763352521\\
301	0.00276761148374789\\
302	0.00276073542761776\\
303	0.00275389910109806\\
304	0.00274710214464453\\
305	0.00274034420313531\\
306	0.00273362492580231\\
307	0.00272694396616384\\
308	0.00272030098195843\\
309	0.00271369563507997\\
310	0.00270712759151399\\
311	0.00270059652127514\\
312	0.00269410209834585\\
313	0.00268764400061607\\
314	0.00268122190982414\\
315	0.00267483551149874\\
316	0.00266848449490189\\
317	0.00266216855297298\\
318	0.00265588738227383\\
319	0.00264964068293472\\
320	0.00264342815860141\\
321	0.0026372495163831\\
322	0.0026311044668013\\
323	0.00262499272373966\\
324	0.00261891400439462\\
325	0.00261286802922702\\
326	0.00260685452191447\\
327	0.00260087320930465\\
328	0.00259492382136936\\
329	0.00258900609115939\\
330	0.00258311975476022\\
331	0.0025772645512484\\
332	0.00257144022264879\\
333	0.00256564651389246\\
334	0.00255988317277537\\
335	0.00255414994991769\\
336	0.0025484465987239\\
337	0.00254277287534351\\
338	0.00253712853863249\\
339	0.00253151335011531\\
340	0.00252592707394763\\
341	0.00252036947687971\\
342	0.00251484032822026\\
343	0.00250933939980108\\
344	0.00250386646594216\\
345	0.00249842130341743\\
346	0.00249300369142105\\
347	0.00248761341153427\\
348	0.00248225024769285\\
349	0.00247691398615499\\
350	0.00247160441546978\\
351	0.00246632132644622\\
352	0.00246106451212272\\
353	0.00245583376773704\\
354	0.00245062889069681\\
355	0.00244544968055049\\
356	0.00244029593895877\\
357	0.00243516746966645\\
358	0.00243006407847483\\
359	0.00242498557321444\\
360	0.00241993176371826\\
361	0.00241490246179539\\
362	0.00240989748120509\\
363	0.00240491663763127\\
364	0.00239995974865733\\
365	0.00239502663374149\\
366	0.00239011711419239\\
367	0.00238523101314517\\
368	0.00238036815553787\\
369	0.00237552836808823\\
370	0.0023707114792708\\
371	0.00236591731929449\\
372	0.00236114572008038\\
373	0.00235639651523992\\
374	0.00235166954005344\\
375	0.00234696463144905\\
376	0.00234228162798175\\
377	0.00233762036981297\\
378	0.00233298069869037\\
379	0.00232836245792795\\
380	0.00232376549238646\\
381	0.00231918964845411\\
382	0.00231463477402758\\
383	0.0023101007184933\\
384	0.002305587332709\\
385	0.0023010944689856\\
386	0.00229662198106925\\
387	0.00229216972412377\\
388	0.00228773755471327\\
389	0.00228332533078503\\
390	0.00227893291165266\\
391	0.00227456015797951\\
392	0.00227020693176228\\
393	0.00226587309631492\\
394	0.00226155851625274\\
395	0.00225726305747677\\
396	0.0022529865871583\\
397	0.0022487289737237\\
398	0.00224449008683943\\
399	0.00224026979739727\\
400	0.00223606797749979\\
401	0.00223188450044596\\
402	0.00222771924071703\\
403	0.00222357207396261\\
404	0.0022194428769869\\
405	0.00221533152773514\\
406	0.0022112379052803\\
407	0.00220716188980987\\
408	0.0022031033626129\\
409	0.00219906220606718\\
410	0.00219503830362667\\
411	0.00219103153980901\\
412	0.0021870418001833\\
413	0.00218306897135793\\
414	0.00217911294096875\\
415	0.00217517359766723\\
416	0.00217125083110887\\
417	0.00216734453194181\\
418	0.00216345459179548\\
419	0.00215958090326952\\
420	0.0021557233599228\\
421	0.00215188185626254\\
422	0.00214805628773372\\
423	0.00214424655070848\\
424	0.00214045254247574\\
425	0.00213667416123099\\
426	0.00213291130606612\\
427	0.0021291638769595\\
428	0.00212543177476609\\
429	0.00212171490120778\\
430	0.00211801315886378\\
431	0.00211432645116115\\
432	0.00211065468236554\\
433	0.00210699775757191\\
434	0.00210335558269552\\
435	0.00209972806446292\\
436	0.00209611511040313\\
437	0.00209251662883892\\
438	0.00208893252887819\\
439	0.00208536272040548\\
440	0.00208180711407355\\
441	0.00207826562129517\\
442	0.00207473815423488\\
443	0.00207122462580098\\
444	0.00206772494963755\\
445	0.00206423904011657\\
446	0.00206076681233021\\
447	0.00205730818208314\\
448	0.00205386306588501\\
449	0.00205043138094293\\
450	0.00204701304515418\\
451	0.00204360797709888\\
452	0.00204021609603284\\
453	0.00203683732188047\\
454	0.00203347157522777\\
455	0.00203011877731544\\
456	0.00202677885003205\\
457	0.0020234517159073\\
458	0.00202013729810538\\
459	0.00201683552041837\\
460	0.00201354630725981\\
461	0.00201026958365826\\
462	0.00200700527525095\\
463	0.00200375330827761\\
464	0.00200051360957421\\
465	0.00199728610656694\\
466	0.00199407072726615\\
467	0.00199086740026042\\
468	0.0019876760547107\\
469	0.00198449662034449\\
470	0.00198132902745014\\
471	0.00197817320687116\\
472	0.00197502909000066\\
473	0.0019718966087758\\
474	0.00196877569567238\\
475	0.0019656662836994\\
476	0.00196256830639379\\
477	0.00195948169781511\\
478	0.00195640639254036\\
479	0.00195334232565889\\
480	0.00195028943276726\\
481	0.0019472476499643\\
482	0.00194421691384611\\
483	0.00194119716150117\\
484	0.00193818833050555\\
485	0.00193519035891808\\
486	0.00193220318527566\\
487	0.00192922674858861\\
488	0.001926260988336\\
489	0.00192330584446118\\
490	0.0019203612573672\\
491	0.00191742716791243\\
492	0.00191450351740609\\
493	0.00191159024760398\\
494	0.00190868730070413\\
495	0.00190579461934257\\
496	0.00190291214658917\\
497	0.00190003982594341\\
498	0.00189717760133039\\
499	0.00189432541709667\\
500	0.00189148321800635\\
};
\addlegendentry{$\text{n}^{\text{-0.75}}$}

\end{axis}

\begin{axis}[%
width=1.227\fwidth,
height=0.92\fwidth,
at={(-0.16\fwidth,-0.101\fwidth)},
scale only axis,
xmin=0,
xmax=1,
ymin=0,
ymax=1,
axis line style={draw=none},
ticks=none,
axis x line*=bottom,
axis y line*=left
]
\end{axis}
\end{tikzpicture}%

%% file: Figures2/mat0_ctrs1.tex
%
%
\begin{tikzpicture}

\begin{axis}[%
width=0.951\fwidth,
height=0.951\fwidth,
at={(0\fwidth,0\fwidth)},
scale only axis,
xmin=0,
xmax=1,
ymin=0,
ymax=1,
axis background/.style={fill=white},
axis x line*=bottom,
axis y line*=left
]
\addplot [color=blue, draw=none, mark size=0.3pt, mark=*, mark options={solid, blue}, forget plot]
  table[row sep=crcr]{%
0.995722752978051	0.382251502704207\\
1.04110721474138e-05	0.999999999992543\\
0.997775336634236	0.999426382262247\\
0.99999999996281	1.08210087645925e-05\\
0.500220660781389	0.866102807508819\\
0.771607980771598	0.636494361332982\\
0.287593802677985	0.99998114508176\\
0.707154772885129	0.999968672367714\\
0.999962962575976	0.722220090335072\\
0.981506666834023	0.191963542729713\\
0.650310203063883	0.760068391745644\\
0.873277473742803	0.487225909550443\\
0.499264679865392	0.999996360019031\\
0.844350904404326	0.85474631783123\\
0.999952693283778	0.560463882179471\\
0.144053406276475	0.989609493855127\\
0.851868188652606	0.998347569615258\\
0.999538648481042	0.857043280722375\\
0.377643195090444	0.926084616255401\\
0.653977348008379	0.888445577172628\\
0.868966563393808	0.713606595440108\\
0.954425918774874	0.298555646677744\\
0.999873926366359	0.0968419430839814\\
0.603472961717588	0.999045741209527\\
0.758458346698702	0.779503958461685\\
0.893091765135985	0.602776887237896\\
0.914868393236804	0.404280853568298\\
0.40070630091337	0.999867502591093\\
0.761493945478424	0.910273089553111\\
0.999312129751898	0.471937256460902\\
0.579297397828287	0.816474704617646\\
0.922578270807843	0.921753266795947\\
0.922904065568327	0.795523431059932\\
0.71101953704443	0.703357482571218\\
0.999764002533823	0.641388096204215\\
0.826510287135934	0.563103761310246\\
0.219938951192186	0.975999432651662\\
0.072524663198803	0.999935882929476\\
0.564168438701186	0.923598429513243\\
0.999850389556458	0.260845425089989\\
0.782190778764861	0.999377159159823\\
0.918837045417813	0.999317166865778\\
0.440030664002387	0.898224608302101\\
0.999019781506167	0.92995081425466\\
0.708268214408936	0.836441593500628\\
0.311644304042778	0.950651361681066\\
0.932886232561113	0.529131247384506\\
0.999511935552741	0.792591042877765\\
0.796541854006531	0.711177252151377\\
0.932717486956925	0.676740545900903\\
0.837360976256754	0.926477067745089\\
0.835066887516908	0.783845902535004\\
0.660683272254724	0.952157588981758\\
0.842687724232836	0.647183055074801\\
0.488864812686491	0.939896616924856\\
0.922361731233361	0.858194902848544\\
0.999946763688338	0.326466936488348\\
0.779229554395354	0.844492734847719\\
0.348142849460148	0.999911905798396\\
0.936271386012482	0.351675045938644\\
0.641738316243756	0.825608767978679\\
0.938188237753512	0.459439994111685\\
0.999928675503902	0.148133208567509\\
0.935397951714783	0.740256100872857\\
0.556334971291657	0.999716670892214\\
0.951468616622583	0.601703089510277\\
0.998946683458061	0.0483677283986069\\
0.587593904530789	0.872871273380834\\
0.711543411919307	0.926297961242318\\
0.447528519835078	0.999740015631537\\
0.704795062383108	0.767672942896271\\
0.652269478418121	0.999186499046242\\
0.187880673936914	0.999901525867762\\
0.878381282930551	0.545080866873693\\
0.999918511269771	0.4267088292345\\
0.539125841123935	0.842340086601154\\
0.800893120344367	0.599342591847682\\
0.611659138280469	0.93800071309832\\
0.739712740373502	0.673386481942518\\
0.99985448215306	0.517565816233068\\
0.426267591983396	0.95033845664695\\
0.884310668202863	0.951092730428628\\
0.788388253020638	0.951310701391585\\
0.955162765897368	0.961597199914576\\
0.87730168863922	0.815014636963335\\
0.970739110847835	0.240278775711072\\
0.895027843724978	0.446144295313728\\
0.266808592342729	0.964169592214943\\
0.879748115158317	0.891372337560052\\
0.61240151162164	0.790565126236871\\
0.890119378264527	0.657495003041762\\
0.964049352953962	0.889661319429306\\
0.884121761720641	0.760145343031855\\
0.75364919480208	0.730717739294967\\
0.722854701166947	0.877681033456199\\
0.805390449885668	0.888932541634462\\
0.998708035655104	0.680579262959245\\
0.959978933420616	0.826278563024975\\
0.740094100333506	0.966518975294634\\
0.524203526447244	0.904078116695403\\
0.534584995932072	0.959035029567354\\
0.804296874895811	0.753412216008487\\
0.24630559016099	0.999866116428395\\
0.998801374489908	0.601582739296214\\
0.680057033545375	0.73390497100803\\
0.95849206060177	0.410087109579949\\
0.803797895703023	0.669136503731734\\
0.108811304341127	0.99983718107003\\
0.0361542711984134	0.99941327014462\\
0.804498788864631	0.81106964972675\\
0.852310330161302	0.523873096790277\\
0.969736303813857	0.762575226575571\\
0.850222966786418	0.601018381944602\\
0.680054166461161	0.802264316792194\\
0.99996990893607	0.217942352276501\\
0.3610293972197	0.961981572527826\\
0.952379456362834	0.642771408081293\\
0.922467940358182	0.570386156770856\\
0.958962251217259	0.496274613443717\\
0.744602163272434	0.819745521190777\\
0.956355259393321	0.998076195509016\\
0.824735296357988	0.970072887966247\\
0.959935923729409	0.708209216910305\\
0.83706257050958	0.695385342192113\\
0.999991185117421	0.964117990653652\\
0.673931316408529	0.857217294543979\\
0.471530649842088	0.882345312979161\\
0.911531521456547	0.494400828803014\\
0.908649826873955	0.709646149180061\\
0.624646262172647	0.862636939164517\\
0.579703442726851	0.964442226701731\\
0.999984560039945	0.898059764495758\\
0.747023633375743	0.999327175892322\\
0.963371298462445	0.554511036799986\\
0.839864529651068	0.742660767012087\\
0.409455900525381	0.912801347680098\\
0.88644148555306	0.997474651825888\\
0.466118503021314	0.967319835750965\\
0.999807630399141	0.294364072582076\\
0.971998380820403	0.353264184109554\\
0.614607377598032	0.902633129402239\\
0.694276712716888	0.965370245964396\\
0.678340061610069	0.913189932797339\\
0.884563188902934	0.849857527199278\\
0.551214213873794	0.877478960262018\\
0.813550934258738	0.999887444108818\\
0.917711771678405	0.628214981085377\\
0.75829163972648	0.87183967383853\\
0.998645450876951	0.822327674683595\\
0.999059780197432	0.75198997145481\\
0.344542802815952	0.938931888077088\\
0.318437730176034	0.999736320976546\\
0.767660612282893	0.691755614454851\\
0.919021067167589	0.96133509528505\\
0.959319700315939	0.92604936149375\\
0.626149568462369	0.970166537468527\\
0.971381099270547	0.449398942960247\\
0.842891423567272	0.893065646044496\\
0.45866061178356	0.925084237220853\\
0.838932214493675	0.820834070124506\\
0.813324546146494	0.631757603876509\\
0.396043372203675	0.96526829076104\\
0.962872351451377	0.270754049154505\\
0.721257995971307	0.739656043499112\\
0.721427284195929	0.796304237636103\\
0.918247631818547	0.890256476705283\\
0.604966765786839	0.836121115158978\\
0.992462286945625	0.12345814854931\\
0.955840690465024	0.792269579163347\\
0.926486995937164	0.376534347974092\\
0.809856550526385	0.852996736901836\\
0.798761099575875	0.92226542321535\\
0.869498301618707	0.575591579065093\\
0.867363055384842	0.678685107370487\\
0.916992103526143	0.826038242674757\\
0.957080053745275	0.858066521889592\\
0.528959090664245	0.997832015208329\\
0.8720194796313	0.921701644209133\\
0.969155344866464	0.671845759341644\\
0.946495555530305	0.323055891643086\\
0.872085067709747	0.630127060731808\\
0.644562321058137	0.923360847023224\\
0.746948799514969	0.936477360614464\\
0.648299414112676	0.791947739691371\\
0.790120542497727	0.780267795603792\\
0.999736764520248	0.177624321441442\\
0.854329295316765	0.963571131127733\\
0.917272357439321	0.763584729506301\\
0.897322112886312	0.521851939416043\\
0.491746970758834	0.90695052098451\\
0.679208014236144	0.999231486097394\\
0.867580208589739	0.783866285008566\\
0.692635311833133	0.882141879182487\\
0.50391143434294	0.968403909878215\\
0.172798687887541	0.984958462523443\\
0.929186333021296	0.430860770855182\\
0.999649193747633	0.356220905565718\\
0.969746644450582	0.525267195386055\\
0.374070164935095	0.998790306927938\\
0.570014243148241	0.845062346781095\\
0.214599209338368	0.999943695191615\\
0.960132852111177	0.382074789962566\\
0.52415553499338	0.93212618446585\\
0.585043488123663	0.904257168035289\\
0.773882598556561	0.809918942170716\\
0.775926024815903	0.746545581462782\\
0.997334462012393	0.0730470797095623\\
0.97583626655025	0.622084316936301\\
0.738527350326824	0.848854573965588\\
0.472014931554949	0.9999604320274\\
0.999980639615213	0.0241146849618559\\
0.426262694907644	0.980273324502921\\
0.975205750687332	0.581794955579872\\
0.33035727089885	0.972658909673013\\
0.89531307597534	0.734977254967915\\
0.768196625181482	0.971412224930885\\
0.730634421817387	0.904409757865872\\
0.771313813114513	0.663332576702628\\
0.675399898959477	0.829758362138484\\
0.901969399007127	0.683024654205337\\
0.678955574630005	0.772524209952397\\
0.738375241581832	0.70528784830293\\
0.971520384736753	0.734834792665032\\
0.736825938326127	0.763662209719001\\
0.822614296188202	0.720696798468938\\
0.294313279539434	0.971439280592102\\
0.923075916744388	0.597783067373676\\
0.975505394617949	0.317298731745368\\
0.897219795314055	0.790126130768597\\
0.904797758660822	0.55336490938179\\
0.90853975051206	0.468301044679142\\
0.243511285065684	0.970053028323733\\
0.84040283882877	0.542824158582766\\
0.58136967354353	0.999718406650425\\
0.62928948690761	0.999088469423372\\
0.525093432416614	0.869657679475877\\
0.826665632410632	0.59230162014951\\
0.779323193806094	0.886772216200223\\
0.978067403578046	0.978310104083801\\
0.78716389537757	0.617197563282212\\
0.896799524695079	0.920312434406895\\
0.984548811246035	0.494298013088579\\
0.999595382178957	0.405068890044657\\
0.976485537019454	0.216117092326736\\
0.935482567979149	0.709348251029626\\
0.718714210136939	0.952176679278542\\
0.919897700667686	0.654360590910957\\
0.815759001762777	0.944248781374953\\
0.558223566724921	0.949408831569663\\
0.831022761439605	0.670510598452977\\
0.64928744040632	0.854317417760359\\
0.864261443998429	0.746205175367802\\
0.682642886872737	0.939639681031592\\
0.866734442644716	0.870458690131459\\
0.935510842214575	0.488503889213019\\
0.986816862711629	0.16261264993563\\
0.40173797516438	0.939247174196297\\
0.895115275455184	0.974118475789681\\
0.663027019907747	0.976125032338328\\
0.978500324118511	0.945677889960506\\
0.587155365735739	0.937533865113589\\
0.943505720202572	0.768875845215916\\
0.942725959278003	0.904937076617656\\
0.159715039154019	0.99993859891199\\
0.797898883613309	0.976031336798516\\
0.551168185032169	0.903496997958921\\
0.999106014726708	0.448823776800405\\
0.618324929007138	0.81399918189672\\
0.905371406821745	0.425406984630198\\
0.884512157578301	0.467881748355684\\
0.98341088816056	0.700282419675187\\
0.983277915533538	0.874115965956376\\
0.946800294742947	0.572297440709076\\
0.897595918503779	0.872246228416285\\
0.977702681781348	0.804706105874622\\
0.839684483634678	0.622720879336421\\
0.632300397751806	0.777484597043943\\
0.432596759501859	0.924646844607549\\
0.697031595646067	0.719091274395867\\
0.195841244558118	0.98075049035139\\
0.872529444671944	0.513185696438192\\
0.936654025995097	0.943876721278608\\
0.862193710981786	0.833126081126085\\
0.964872811922594	0.472521615138134\\
0.602531298381441	0.971485068948694\\
0.93558904113065	0.982338864132488\\
0.998848462836495	0.537907943298277\\
0.554064096236205	0.975282872561101\\
0.896098696108893	0.576483834229695\\
0.701536616029592	0.81134506087972\\
0.82664941814383	0.87544848054448\\
0.719657189225265	0.982088061574581\\
0.981101484249106	0.839540654938654\\
0.771538662858948	0.935230482665517\\
0.999197599195505	0.240339201001621\\
0.704843141319332	0.860715963927427\\
0.450646592472672	0.949310147767106\\
0.557710269133163	0.830158624819087\\
0.268322604893211	0.999917013106778\\
0.773634240807298	0.718626093080494\\
0.978765370595757	0.90960096967077\\
0.940713857719827	0.835659273983039\\
0.422812750357904	0.999760133864241\\
0.813691089085091	0.779077953462823\\
0.976174857252157	0.424501441116708\\
0.808380660847477	0.690908766257498\\
0.634114651451682	0.946616123472884\\
0.820096355173583	0.908584716846705\\
0.940323934294369	0.875953888728801\\
0.977024334845303	0.650056371094131\\
0.702311935382306	0.903782179194699\\
0.816504799410589	0.832410903520827\\
0.630330823165284	0.884796260071659\\
0.999275554980705	0.77179707531851\\
0.849923200294249	0.567675896883627\\
0.794175461733499	0.645324498040772\\
0.979708994840183	0.285146590884041\\
0.128872029021967	0.999970758343114\\
0.938211044957906	0.399844431506986\\
0.858247301876266	0.940467574288043\\
0.894959753693703	0.629567183109942\\
0.84717758534677	0.764557278393636\\
0.853685497817903	0.802028231394057\\
0.757188487321398	0.654127673193635\\
0.865110956222955	0.653288390729389\\
0.940314829979286	0.622026316240391\\
0.870341758204271	0.980177281716646\\
0.788618618504555	0.866557355232481\\
0.951388023306825	0.434985098640067\\
0.504588881249694	0.887979772668217\\
0.0910914312363061	0.996005469181981\\
0.99953876368653	0.583110096118012\\
0.847311133845211	0.719130614736406\\
0.88353845806286	0.698246580724428\\
0.938561451814575	0.549076707521145\\
0.894223898356114	0.827621069690422\\
0.465664943405363	0.903828238786878\\
0.598604543277233	0.803350768601083\\
0.665523526438188	0.748003881411592\\
0.744740760785214	0.795220367782936\\
0.939946780050816	0.812409261040542\\
0.905727594296581	0.942009773378951\\
0.831907327067985	0.999502965139806\\
0.758623230196124	0.836171430570944\\
0.871171199744886	0.603835176461012\\
0.977353574370921	0.998670645355605\\
0.375224715702288	0.977949666188614\\
0.697768510597818	0.749195649164666\\
0.054745103918368	0.998568340077339\\
0.9444858372461	0.512632654039604\\
0.484865882172178	0.980550498333823\\
0.951686190070421	0.685749823334019\\
0.601029106710155	0.856687237090834\\
0.378042474241516	0.947986150585475\\
0.999772583392146	0.624488848058968\\
0.802595381742339	0.732967756362169\\
0.748840818578599	0.891014930815248\\
0.0181425131462162	0.999967997694735\\
0.97724733935119	0.397064892784714\\
0.998859110631514	0.659049709835495\\
0.660860592927892	0.812193656196027\\
0.856509214066297	0.908184923319461\\
0.288046955877139	0.957777364607648\\
0.757404351041308	0.756148139669906\\
0.915663895308855	0.731411810366412\\
0.999994292763519	0.491917759056528\\
0.448409084311644	0.976444587375445\\
0.569882790609263	0.887455029535315\\
0.999777893106448	0.196559185043605\\
0.504957604707887	0.924983443659055\\
0.724213410054921	0.822663207503455\\
0.52462964367673	0.978145707455749\\
0.791368917926996	0.827697677560761\\
0.605859071824042	0.883051490295543\\
0.825669655759042	0.804494235120512\\
0.523122198079608	0.853529053462841\\
0.785964930265102	0.685354278234232\\
0.782714530439675	0.908924024036783\\
0.671863531847876	0.882363366348805\\
0.812590702134411	0.584179915641321\\
0.54356313720651	0.927601086084082\\
0.959942973184166	0.337754327072222\\
0.697232457857176	0.789049257371758\\
0.72912062704093	0.688748500575242\\
0.890991146213225	0.499088269897415\\
0.99986108618389	0.878737975671702\\
0.977979707182867	0.782489867965371\\
0.917833017167552	0.5139473531905\\
0.625495792906458	0.838523739965854\\
0.824349512818321	0.755919334039686\\
0.512470673921907	0.949201028940748\\
0.900242172129443	0.900261052987778\\
0.603564222317936	0.9198234367401\\
0.818176336012307	0.652808877711617\\
0.917238543483659	0.448142044337643\\
0.638590382228108	0.904055158153581\\
0.945481404791447	0.660198209776515\\
0.555728355670723	0.859715091042187\\
0.765356238019985	0.999547519246358\\
0.834788960481539	0.951794080904856\\
0.953645501474338	0.74506287317247\\
0.951682515811825	0.363815610079065\\
0.820238619082429	0.611453436845812\\
0.731187541226305	0.723189406913412\\
0.975033021090494	0.601543237678398\\
0.350131493452407	0.979796735075201\\
0.981794346466683	0.547941110464993\\
0.327463098984776	0.945136460760755\\
0.906141101675605	0.847390173890199\\
0.662172600740154	0.929893166068355\\
0.979863704639428	0.259064917045994\\
0.758645310424178	0.953355982701706\\
0.31072961929105	0.980471802260594\\
0.643459011425607	0.977651982100951\\
0.693141571544269	0.984257185351673\\
0.361288117287016	0.933079467696531\\
0.732417443154699	0.925782309246458\\
0.473057106208934	0.948939814753214\\
0.902296415259213	0.807462376839579\\
0.724037547541737	0.998129860522829\\
0.860935283996034	0.549382671748166\\
0.855581590401794	0.693772026270897\\
0.272710445657393	0.981856501526875\\
0.913167667016046	0.537717533134634\\
0.920378726962988	0.691223783526078\\
0.979383882255996	0.371182153402521\\
0.778142544616369	0.765043428975049\\
0.842870207661883	0.979706361904426\\
0.689535908630167	0.844274805236405\\
0.997822534710111	0.703031750472668\\
0.984802158692219	0.338845213469633\\
0.934784137921278	0.999318413074245\\
0.663529460225064	0.77937550788975\\
0.999526091174973	0.978675945689718\\
0.99993737978069	0.947877521553488\\
0.726507435122497	0.778591358804681\\
0.407665383161305	0.980971616395375\\
0.575662005703856	0.982216828814474\\
0.479648751255724	0.922474880050457\\
0.877218473386255	0.729985785828188\\
0.634942756525701	0.804966016164476\\
0.851092686218093	0.669258999456181\\
0.900681960679851	0.770479308802728\\
0.774434703702505	0.790992897188578\\
0.999746061930336	0.840644079320422\\
0.625842990669719	0.923033145699674\\
0.587448528016772	0.835443710585088\\
0.531752596409797	0.887306447338239\\
0.869734979983338	0.998620331445367\\
0.948024859045041	0.724880830645125\\
0.39333080931742	0.920454106451127\\
0.999892091238296	0.277970387443584\\
0.739669724765331	0.867666812198494\\
0.700231815984262	0.944096566277341\\
0.660250649877511	0.907327463752719\\
0.953910651418917	0.537003772558074\\
0.425575529840365	0.905180946659443\\
0.8325113403711	0.842605295609895\\
0.914367287805037	0.978840932501745\\
0.602371148943835	0.953593668175437\\
0.756315298113661	0.708233303780265\\
0.657465632322438	0.835980843057893\\
0.485314802579136	0.87515198178827\\
0.455100969744691	0.890735221762293\\
0.846480109496632	0.875140883390196\\
0.749713044545888	0.981814710947914\\
0.99965746556619	0.310426070421867\\
0.982802530074059	0.465023303165187\\
0.909754842972902	0.611469037297053\\
0.979024764005539	0.71783294961281\\
0.229761386550447	0.999617392371436\\
0.958357601447747	0.979687186409118\\
0.649334813189156	0.871416456609744\\
0.762705102046422	0.854545999155662\\
0.756927805155903	0.678487106631338\\
0.677448537253293	0.963440245550078\\
0.979153413009283	0.510969592197924\\
0.722086516761586	0.852765916462859\\
0.974581991678846	0.857544845011396\\
0.929835121417098	0.780339247591742\\
0.740340523657783	0.744886626460246\\
0.88330960054132	0.561823104603407\\
0.902949483667732	0.997525378519099\\
0.863085796321297	0.852242681365761\\
0.864613984253237	0.503628023651399\\
0.796786356637145	0.79635173420789\\
0.843221777028606	0.583114813949463\\
0.906882390443296	0.74852710901756\\
0.958735580401735	0.621234733748795\\
0.856673199566737	0.620784294878811\\
0.931826775789119	0.641881313143192\\
0.343136429772402	0.958020531901331\\
0.805537717015842	0.959520876816853\\
0.937225850987177	0.588580892208546\\
0.957072017370594	0.94242113528156\\
0.926041558059914	0.472734333261062\\
0.908463880056175	0.665896778020678\\
0.877897675280149	0.79604764626601\\
0.488404278768768	0.959460740826952\\
0.982668076585307	0.749208569702582\\
};
\end{axis}
\end{tikzpicture}%

%% file: Figures2/mat1_ctrs1.tex
%
%
\begin{tikzpicture}

\begin{axis}[%
width=0.951\fwidth,
height=0.951\fwidth,
at={(0\fwidth,0\fwidth)},
scale only axis,
xmin=0,
xmax=1,
ymin=0,
ymax=1,
axis background/.style={fill=white},
axis x line*=bottom,
axis y line*=left
]
\addplot [color=blue, draw=none, mark size=0.3pt, mark=*, mark options={solid, blue}, forget plot]
  table[row sep=crcr]{%
0.812161773817706	0.859779776397228\\
0.999999999958788	4.38751249836998e-05\\
1.0628197832063e-05	0.999999999979544\\
0.997979082246083	0.994966648194319\\
0.9999044676976	0.463628587267867\\
0.40643453967391	0.999864127577685\\
0.7763917609825	0.63025430215355\\
0.999848800637272	0.725336877771778\\
0.545451784589755	0.838556722008749\\
0.974159547837151	0.226209193975112\\
0.665918455490686	0.999828331240069\\
0.195562845036048	0.980732134304544\\
0.884131235564515	0.467399859025123\\
0.848789066581898	0.999018649383008\\
0.665153386529364	0.747551347038837\\
0.373799927121973	0.927573615169497\\
0.999589355490058	0.862203834740405\\
0.999960298327967	0.109542754761074\\
0.924807453299858	0.607569323396026\\
0.999534647560627	0.334806209721664\\
0.534943780635409	0.999907311915417\\
0.868588810335526	0.740351302954242\\
0.666897304094394	0.879795811807691\\
0.0966928629930779	0.999973296863189\\
0.292126986694096	0.999939953805822\\
0.932092454532269	0.363920670084634\\
0.999794971313529	0.586751878793253\\
0.9170429803151	0.91702083601089\\
0.760403841690682	0.756023909325207\\
0.465286970688763	0.885378710161868\\
0.758906486114566	0.958076739184673\\
0.832977769048971	0.5534233286525\\
0.57431011250341	0.924845594035489\\
0.927176402371053	0.810556247056236\\
0.939290240286668	0.523679353520506\\
0.927256852218464	0.999248890775929\\
0.849508869324489	0.65456108074821\\
0.289159707248482	0.957535831110105\\
0.722280256199584	0.694113735415735\\
0.999915967599608	0.187596015449544\\
0.728490408559995	0.832753518417157\\
0.937035621166502	0.690136527608351\\
0.60504821233134	0.796888941393288\\
0.473443428542637	0.95551507955487\\
0.998237428744996	0.934324741112884\\
0.738539220803347	0.999878580866833\\
0.998994078308574	0.653058124596717\\
0.999593942156705	0.269780140792293\\
0.839749561367447	0.929385229492619\\
0.949887065357398	0.42618824013381\\
0.999668211944095	0.792028166622156\\
0.59646467152968	0.999742946883565\\
0.953186668744542	0.302986238866541\\
0.838204173131413	0.8002943491188\\
0.68412963052276	0.939594224371696\\
0.998749522620262	0.0505805355268238\\
0.799719024450092	0.699872783643763\\
0.22870138614264	0.999955044757416\\
0.99973746213345	0.398621881071093\\
0.88305730747667	0.86067613676681\\
0.666205286283495	0.810636380642046\\
0.155096287240775	0.999953046528811\\
0.602383192751552	0.864023025798016\\
0.746759726804055	0.895619770456799\\
0.348033691625578	0.999984987096749\\
0.999140947723059	0.523156566740376\\
0.471618157555548	0.999882819389048\\
0.519611917812952	0.900665963714321\\
0.88513618120933	0.559132003465534\\
0.942107362217717	0.752208678713586\\
0.0442158731078374	0.999055578491466\\
0.791691062464429	0.999447888207644\\
0.628509866119364	0.944796751201814\\
0.947825627435645	0.872021855108225\\
0.988217311916828	0.15394485424438\\
0.834111556783291	0.604759215512763\\
0.90972252133275	0.416873420657091\\
0.419006221573984	0.908672353750956\\
0.956384020227943	0.961101465298954\\
0.782130879678673	0.80735196327921\\
0.711814643741206	0.767967360432163\\
0.887175757291118	0.962636042616947\\
0.860647699286166	0.509322203435266\\
0.414694760778117	0.955369410090806\\
0.897032007843316	0.655498247237766\\
0.816818022485291	0.74898705337804\\
0.530648104213178	0.955399692249018\\
0.938812855378781	0.475708474051942\\
0.247271081752021	0.969033386598395\\
0.794288382769096	0.915517390378313\\
0.962359037565199	0.635350114930983\\
0.960299847031393	0.566396375864534\\
0.888234445348362	0.784107347245502\\
0.332032720420718	0.943466936021796\\
0.507654640378952	0.862613495307375\\
0.749310903292595	0.663889347374197\\
0.875048765919866	0.696372206560828\\
0.969829036835869	0.824663678062979\\
0.814676737359091	0.962291017478657\\
0.970524869627924	0.365742516331\\
0.624292649508103	0.90151014470483\\
0.129836734569013	0.991574745132075\\
0.804202058076457	0.595289671855543\\
0.969744063751392	0.907611310419531\\
0.765565599917117	0.855580444028437\\
0.708164354181112	0.972051162732322\\
0.878024908783654	0.608106244078125\\
0.757061002926498	0.70967103352111\\
0.560828061988579	0.879297484038879\\
0.370011130360139	0.968595192252197\\
0.707419653551374	0.874651604425304\\
0.635126026051858	0.772409074143811\\
0.855184949822166	0.890337469991462\\
0.573619890537711	0.968222121381336\\
0.900707782554321	0.509779842594118\\
0.977487630219422	0.691270877627728\\
0.803790551722634	0.656000424138105\\
0.633351464227721	0.837004062234138\\
0.963987023225725	0.26659647106841\\
0.909068151024829	0.728155736950848\\
0.962300008563406	0.999646101209664\\
0.697280732073654	0.717249558351451\\
0.855337615084876	0.835013534951662\\
0.723817453876279	0.927265091896184\\
0.889335837104631	0.999472597912238\\
0.973177695591072	0.493727586003959\\
0.977906116521284	0.761935671912124\\
0.460281024414086	0.920716549680318\\
0.999930134975898	0.229860788439852\\
0.632437999905491	0.983645995674391\\
0.838822182937919	0.707204989938519\\
0.575503369611611	0.819843037972778\\
0.740663878583519	0.793407373161073\\
0.68446467871575	0.840685371371908\\
0.922090432404736	0.563833771624062\\
0.878661402776646	0.919734330294817\\
0.320558553910775	0.9775131466664\\
0.440963179836792	0.983305285602599\\
0.998557837070138	0.895024048688632\\
0.916316312846809	0.84340250569687\\
0.984373073401173	0.432750770599896\\
0.999021565062132	0.618315192232373\\
0.996559570973609	0.083427656966469\\
0.984782907597405	0.303977274059033\\
0.501940353967443	0.984467004083435\\
0.962279106195952	0.725465911400437\\
0.705136638526393	0.999968855018968\\
0.725990665556191	0.733907001048803\\
0.49890901856627	0.92896055699693\\
0.94427124150053	0.334209251371457\\
0.922998965568238	0.954569389931234\\
0.699572098084133	0.802560579816959\\
0.798664705650255	0.776127609223668\\
0.999369051425209	0.825092603453358\\
0.99941194288416	0.969782500498016\\
0.887129492248509	0.819384727421135\\
0.659591913175275	0.915778385316277\\
0.933387174472506	0.651858373899762\\
0.913910522186015	0.451394431872569\\
0.66325705371787	0.96650719510405\\
0.812282795095192	0.826851025454393\\
0.908413411958725	0.884425647144792\\
0.95655805142274	0.789361163934549\\
0.999457078781913	0.555049799311331\\
0.981356716642313	0.192853146306785\\
0.261960481054202	0.999982361841146\\
0.851688993936607	0.967118603845376\\
0.948050643603054	0.394040250610081\\
0.964352913836708	0.599739646336053\\
0.673802978713239	0.777619533158995\\
0.848947974902745	0.767984203171946\\
0.855507050636891	0.577284390167085\\
0.786100627972395	0.885357243673685\\
0.192879763908978	0.999822601071517\\
0.786954355416846	0.733186640244196\\
0.999896691791852	0.695073280833221\\
0.694351468979807	0.906486243955511\\
0.821071776919329	0.895220293608792\\
0.999550422051739	0.36875406570662\\
0.54263784622006	0.921844618143252\\
0.561594741286123	0.999866489591713\\
0.915237605501235	0.770253119218066\\
0.970738295930549	0.533554923945398\\
0.777385860326672	0.674371961851299\\
0.760893282864919	0.926767489216071\\
0.072814096702197	0.997350091954397\\
0.600652924828251	0.830420239081767\\
0.999986318853137	0.0231679540833199\\
0.965881710855801	0.458596185565427\\
0.589329249918712	0.895808908686142\\
0.862213415683043	0.537934098278571\\
0.810772012082693	0.625909079570294\\
0.819960455586188	0.998569502254516\\
0.602176006046377	0.958463808625411\\
0.90645462200565	0.690441246150911\\
0.376262821446455	0.999696493964327\\
0.634132879432731	0.803885819995753\\
0.949334815989117	0.929557945198895\\
0.168426643620231	0.985805318631487\\
0.637990258061427	0.869808174242595\\
0.830140175036485	0.677418457278163\\
0.490628756556677	0.894447852268692\\
0.999442337585237	0.763805692340642\\
0.439973215778264	0.943747728469452\\
0.973539981773897	0.858630551659969\\
0.785903135841528	0.969337655476023\\
0.999376337466854	0.492994404351707\\
0.572184393383514	0.852136908579508\\
0.92128618553219	0.389433093876328\\
0.999975094818857	0.142720272860794\\
0.89776258039052	0.585994331561008\\
0.970843874792118	0.331266157306859\\
0.841922144750939	0.861790891419827\\
0.757626695225834	0.823371902227041\\
0.96621192755265	0.663171718273938\\
0.73408197790988	0.86447391513752\\
0.902143732925015	0.627383573631641\\
0.53397469608611	0.868400947118801\\
0.992235400658564	0.124837545221174\\
0.908209325409572	0.483461413996034\\
0.399912017956309	0.931791426209217\\
0.605737153219907	0.925039848320089\\
0.694283711533391	0.747775005928215\\
0.945287568216475	0.838868282875275\\
0.763316726663026	0.998628335003625\\
0.854857536357567	0.627454569659684\\
0.97521010495351	0.402046679470606\\
0.907906143950673	0.53887396123785\\
0.628572574772108	0.999760814409015\\
0.276393358454891	0.979278252820902\\
0.443866417314433	0.999921252422922\\
0.735172708757819	0.969269508856318\\
0.937589820361319	0.899291982828895\\
0.0204758864411249	0.999977643334132\\
0.796000017934004	0.941543803437004\\
0.891489062799567	0.755280765077345\\
0.977421303416092	0.945243649202155\\
0.322684958021194	0.999979569047879\\
0.78577685823728	0.837970916448459\\
0.8638485435899	0.798589896936364\\
0.821863274649739	0.573930730313993\\
0.905876178534044	0.981286340640541\\
0.768722037215587	0.783206339579329\\
0.934773559596191	0.720194555969172\\
0.658800811336821	0.846128870951538\\
0.871618270670821	0.6702768624485\\
0.22087408284854	0.975320411765159\\
0.54416926121137	0.976344678284067\\
0.445353204105496	0.895758646987808\\
0.395595383804464	0.976686109015278\\
0.89710474770151	0.442736182912048\\
0.977225665117478	0.979242713445452\\
0.83806910621732	0.735661815074139\\
0.500780157629164	0.999957738235694\\
0.984368852331566	0.251729114879191\\
0.883719067635374	0.89129999124494\\
0.999665328761919	0.425042627591187\\
0.557389669455499	0.946007522357581\\
0.864869224648011	0.942475857009459\\
0.87274095340933	0.488373251509881\\
0.9398267117892	0.585600199535179\\
0.355646360130159	0.948836448438971\\
0.721211334657903	0.898868250513705\\
0.999677880011144	0.308032193454113\\
0.939138662180438	0.979847189865906\\
0.501930769075906	0.955437471748163\\
0.934316190498778	0.49986769329535\\
0.685282676325789	0.981402999946667\\
0.811924499061748	0.797965652538614\\
0.899927406793311	0.937290403502681\\
0.813708764835354	0.72089736349377\\
0.876966845534247	0.639871190469613\\
0.739024498065957	0.762428110957875\\
0.75079364072014	0.686145339273527\\
0.975250302030166	0.882255247549591\\
0.655782517880116	0.940633498643626\\
0.884985535192674	0.719646646496409\\
0.343798910236027	0.974186546232847\\
0.978798764014709	0.800159909323103\\
0.881738872543426	0.525036226873022\\
0.939359752757478	0.625948291405426\\
0.87199664019424	0.982105197330799\\
0.828929869324393	0.642162830526106\\
0.710657979549429	0.946664905977904\\
0.555390976868652	0.901334994421153\\
0.983107218468945	0.573282618186094\\
0.708799959085746	0.845243906435882\\
0.93988965370963	0.448798157951171\\
0.82462246757713	0.776120739797868\\
0.943558127559785	0.546832969395926\\
0.818611613947799	0.923645197732402\\
0.470025964339473	0.978066516543878\\
0.310681398994096	0.950735321776091\\
0.907008099655014	0.799737126562157\\
0.75260376156858	0.734179855714973\\
0.485375951544545	0.875517774563048\\
0.38592490782392	0.950287758037693\\
0.773837392588722	0.653496785779007\\
0.60585010190912	0.980970192820524\\
0.434063086469436	0.922077735853924\\
0.977844633276042	0.282141945833661\\
0.99944862401326	0.674383271176486\\
0.923742499732919	0.864592681504219\\
0.933489030132906	0.785334516714044\\
0.999926851979766	0.206604282849764\\
0.683518241897912	0.864554700747555\\
0.719294509927397	0.81181256539727\\
0.739048553406879	0.944264042676912\\
0.832974482815125	0.834397163930308\\
0.930479428824645	0.415068807091468\\
0.981675149180997	0.739033487128562\\
0.652975178093984	0.788200593054137\\
0.770439444128141	0.902254822727252\\
0.828862971066088	0.980203025053667\\
0.636271846150611	0.920743296005241\\
0.120597904611893	0.999987309999613\\
0.791034854666153	0.612626108066425\\
0.852073586343389	0.685701806248118\\
0.732012994510486	0.71123566290326\\
0.920115744059158	0.669863015376463\\
0.848468873958445	0.530105260235674\\
0.983616560251102	0.633764454220409\\
0.959772347724245	0.511880943176978\\
0.521527256111963	0.930783028506439\\
0.529625037947559	0.848321498456816\\
0.858713730457764	0.716689670714762\\
0.645617012618369	0.894077806345324\\
0.851768323556022	0.91197103922189\\
0.778936436674831	0.705474262093427\\
0.75964846426114	0.877275907032288\\
0.784396796297099	0.759855316881245\\
0.952464703909583	0.354335286443834\\
0.958380729405552	0.70208686204159\\
0.951540747327738	0.81256293871736\\
0.759842572996564	0.980458614424935\\
0.418914664711527	0.980806684986279\\
0.583939070578578	0.944072814721986\\
0.8717127109642	0.769585562696159\\
0.788535628350624	0.861206339146236\\
0.855399469375405	0.600106658960225\\
0.89049046809767	0.842110717490472\\
0.83188866588411	0.950132222773501\\
0.919219437913752	0.747269107550878\\
0.983180707149068	0.71405108262013\\
0.612000958962238	0.882607218251289\\
0.269418538174791	0.963092656902612\\
0.999920659053725	0.166444505382403\\
0.985547883179855	0.603781646461961\\
0.681770670858267	0.734536361577718\\
0.999353892084693	0.914209595404261\\
0.80614685875947	0.678928701115798\\
0.69445857063915	0.822613752437394\\
0.87589726103443	0.581388403606176\\
0.522940496554921	0.979220444973205\\
0.863053605799433	0.862926221665929\\
0.108421194460489	0.994228250649929\\
0.983200638113532	0.473955269180866\\
0.983798647944466	0.348827529597717\\
0.989718841923902	0.841450062901772\\
0.245685244310063	0.986845954607893\\
0.94848851155347	0.673165915510914\\
0.450201143455243	0.962622012871793\\
0.999978586169802	0.0690279641623435\\
0.480092011112058	0.914158381735864\\
0.618301425669026	0.852645663763891\\
0.614904290127139	0.817342314735895\\
0.299369739085087	0.977946227941157\\
0.648332157178671	0.761550302501442\\
0.717759119591455	0.78684200071512\\
0.982760463397075	0.512838755121941\\
0.640959521755534	0.963931999459432\\
0.899572396241339	0.906517834144805\\
0.579755185920009	0.988288331547977\\
0.687714660330625	0.959054734793593\\
0.582767239435629	0.874083553583743\\
0.80151041293645	0.984121726540852\\
0.650128659425109	0.824281424201052\\
0.688160840888675	0.887333157584608\\
0.984200002165893	0.383100573260234\\
0.513252992991802	0.880355675916612\\
0.732814622608015	0.680853514051692\\
0.391379978531319	0.920313814343221\\
0.867344620258058	0.999647557640783\\
0.956169157271188	0.767981587215302\\
0.776642884357093	0.941590132887048\\
0.747324691346461	0.845885655836087\\
0.999916533188363	0.28981511470849\\
0.647377055887239	0.999260561260165\\
0.999936233024924	0.443191313120789\\
0.693097659083724	0.776908474749079\\
0.591984767582158	0.806548806809739\\
0.917253682387274	0.639899769250329\\
0.979463347166871	0.924659125232591\\
0.538416433775967	0.889125052878634\\
0.969497123169304	0.245153654438322\\
0.909606737857201	0.82118342338168\\
0.915253759878226	0.706649700494158\\
0.855262485241997	0.556068976720802\\
0.979983491829432	0.550713064211174\\
0.830565021731875	0.879202282421811\\
0.742295072378958	0.916621558222521\\
0.477520419262894	0.935097927121044\\
0.723248986751578	0.989141582309953\\
0.918627825868424	0.517776090085143\\
0.95695322851742	0.483031283654973\\
0.930615382177878	0.936498662832238\\
0.804849346044514	0.881759952739061\\
0.792842647818567	0.636815935465596\\
0.958369935675402	0.28568456418797\\
0.901722908798712	0.606027334063242\\
0.848376372349647	0.816919927318226\\
0.92363126546824	0.434481612836449\\
0.888883761995234	0.492769705697547\\
0.55072254625457	0.856841346909115\\
0.70098109802914	0.925884438657681\\
0.35561933665853	0.934882246350674\\
0.984734022342303	0.17443505808656\\
0.982948855883479	0.780909546739161\\
0.759261469354235	0.802175132618355\\
0.210180129866008	0.999885351745376\\
0.902714058574322	0.864787482475668\\
0.917515675107912	0.587521925886086\\
0.950569675869275	0.375592459497396\\
0.426821127425028	0.999840704925211\\
0.821871776646435	0.695965360005505\\
0.892295939483112	0.676985749641145\\
0.620739734860603	0.784814613874309\\
0.673855177158227	0.901548226582622\\
0.99990883752734	0.74743810085664\\
0.835491339901258	0.584274328378083\\
0.981205955664998	0.996312128873921\\
0.999914486812125	0.251136178673526\\
0.94772225804001	0.608736724174745\\
0.361293223365855	0.986994289303544\\
0.911101547638863	0.999220188231553\\
0.984157533150201	0.663805702300605\\
0.954788008448695	0.892813205275883\\
0.708111376779149	0.730097425830464\\
0.680173238275736	0.79718701194546\\
0.173892771236148	0.9997115137225\\
0.656346293880989	0.983651753567653\\
0.740298251048047	0.814009545312203\\
0.904075542942958	0.565348469010309\\
0.761995527206503	0.648279568431606\\
0.963436086608468	0.439081971762063\\
0.419934612725792	0.935339780026517\\
0.970701661727463	0.61833789859158\\
0.620245311259843	0.96461738164556\\
0.957344175793848	0.747123886334482\\
0.68168861921504	0.999496492919437\\
0.328859285161948	0.960590805141848\\
0.835293030656878	0.754968041063215\\
0.590421416953909	0.846819693183608\\
0.769508162176781	0.727406057178125\\
0.870051287730536	0.822935583336281\\
0.998109464911682	0.953586919617361\\
0.559633256405	0.829878748264491\\
0.832912414938398	0.623885822880821\\
0.463555342670296	0.902866677749007\\
0.719713209700668	0.750839694061492\\
0.48915625862891	0.969975711103237\\
0.835942300082793	0.906021732651758\\
0.604575376371231	0.905299837771055\\
0.731109869634327	0.882086481380023\\
0.677551171648266	0.921357647110074\\
0.966832326095011	0.418613358704445\\
0.966352950333851	0.312723354181352\\
0.503097838238309	0.911370722049123\\
0.926426766928025	0.542684137720435\\
0.710743632173644	0.703483049463454\\
0.961472045612566	0.84632314835932\\
0.67635883586618	0.75867844374101\\
0.928395704277541	0.88460455072531\\
0.825612154566171	0.815079841267244\\
0.941678237274635	0.951073068433943\\
0.946515933833874	0.999851470227715\\
0.778796652947098	0.987370801255604\\
0.901603696727399	0.464599277918931\\
0.518074224718528	0.998949279705913\\
0.671569839095543	0.828215413700251\\
0.815404951387823	0.606876366403209\\
0.797066495796905	0.819094668119956\\
0.895276448217557	0.705610239061916\\
0.457548769773803	0.942431058724644\\
0.984572206184096	0.453404693068318\\
0.957466696521751	0.983474894314314\\
0.751472692612071	0.777405889537161\\
0.823380670666027	0.658831844663681\\
0.905258927068622	0.956531576209028\\
0.573847442319578	0.904629099296737\\
0.656526631692984	0.863952130318322\\
0.430836509327254	0.965328066766725\\
0.0584977357280048	0.99996861959939\\
0.892879413314684	0.737967814945072\\
0.9224686692013	0.468025535726467\\
0.98537106827094	0.213623380725134\\
0.883858132979803	0.80199261799196\\
0.999379383370924	0.879243251530782\\
0.799190481858763	0.750825060188967\\
0.870692533976182	0.844542810409589\\
};
\end{axis}
\end{tikzpicture}%

%% file: Figures2/mat0_ctrs2.tex
%
%
\definecolor{mycolor1}{rgb}{1.00000,0.00000,1.00000}%
\begin{tikzpicture}

\begin{axis}[%
width=0.951\fwidth,
height=0.951\fwidth,
at={(0\fwidth,0\fwidth)},
scale only axis,
xmin=0,
xmax=1,
ymin=0,
ymax=1,
axis background/.style={fill=white},
axis x line*=bottom,
axis y line*=left
]
\addplot [color=mycolor1, draw=none, mark size=0.3pt, mark=*, mark options={solid, mycolor1}, forget plot]
  table[row sep=crcr]{%
0.465485547297182	0.312528604449066\\
0.994707292339675	0.998834981792273\\
0.00473869004015193	0.999123835669222\\
0.998451938329319	7.91230608497839e-05\\
0.00263753861992211	0.00115444221968963\\
0.502298784000462	0.999944734671007\\
0.999584857817467	0.510627727222282\\
0.000299044245616109	0.524438208099088\\
0.550204218237345	4.23074360417219e-05\\
0.610602285794635	0.654531414294995\\
0.264706347983162	0.723462907986533\\
0.803475265835525	0.261548088502759\\
0.266018598979793	0.00138136290190183\\
0.00138794978424339	0.256955103115965\\
0.747633306568571	0.998872618490131\\
0.999369405562693	0.751594270635137\\
0.242753738648554	0.997787508053737\\
0.00075777986264991	0.770381478203838\\
0.247751941731316	0.423804397482314\\
0.999467508697527	0.244734822104217\\
0.775877048503926	0.00019981507164113\\
0.743673498645949	0.483202853538003\\
0.782404757744126	0.776320935766111\\
0.23924048507991	0.203177583006758\\
0.427916172701057	0.555832587194051\\
0.465282412527936	0.808806565457813\\
0.61615767659761	0.195352558449429\\
0.414187011654655	0.119329853309076\\
0.157149747837863	0.594255067079087\\
0.587596090271212	0.429065170293112\\
0.855828274277702	0.616666632341524\\
0.151989725771967	0.853542736036079\\
0.634181873384796	0.861490217912878\\
0.0958001545215814	0.380632073556591\\
0.106349040769406	0.113181682150347\\
0.899818171586743	0.391335965411939\\
0.885608354456259	0.118685400649244\\
0.88887879756248	0.889948323943883\\
0.349571014932026	0.900333715795156\\
0.404499674997207	0.00248265671835324\\
0.00155957927474737	0.650441781836824\\
0.333398998127229	0.308800904624973\\
0.4088095077464	0.688819348471317\\
0.685688694029991	0.330336794638112\\
0.133207591852054	0.00138917533478611\\
0.71322135599442	0.110691334680056\\
0.619926993159437	0.999739968167022\\
0.999444533728415	0.633170621638787\\
0.294519774881432	0.559489460036731\\
0.127575300819662	0.997804033178214\\
0.000300705428341486	0.134101540967145\\
0.000115146132991906	0.881114621064563\\
0.998680091608952	0.363902375872301\\
0.999057104805084	0.865614380654591\\
0.871595809440537	0.999648930736071\\
0.999154195708493	0.126437545815339\\
0.416187005351803	0.432746814889373\\
0.385520759348072	0.996834977987177\\
0.00161302511907102	0.382323790247282\\
0.729124282728025	0.631902038447433\\
0.660006232188301	0.000717135891001086\\
0.130045378316656	0.72717914396052\\
0.138324787303857	0.265602362853494\\
0.563461780890659	0.54624877680376\\
0.88215599089285	0.00119288012183527\\
0.503320136531213	0.190245333365675\\
0.573896097969258	0.757831228773078\\
0.868748103217363	0.502104630460815\\
0.887794335376553	0.743016256964462\\
0.1347824510427	0.485969416519155\\
0.29755462312607	0.112079644516933\\
0.752432294709108	0.887751016656412\\
0.26311855676909	0.837935919732744\\
0.575684289109641	0.306145131054281\\
0.522505935783857	0.897391031755135\\
0.789206381813916	0.379705458543871\\
0.569301363581184	0.0989742040693095\\
0.899444650790697	0.248008910343381\\
0.379179444724825	0.218205641477992\\
0.683866083948299	0.747470953746824\\
0.501554415824822	0.653914036837047\\
0.22412481660179	0.325298905934633\\
0.360218503418396	0.785743800502557\\
0.662252316379513	0.537241575784786\\
0.719241768122035	0.2134715626572\\
0.19818524567325	0.092035612603059\\
0.332005624097973	0.469955209444117\\
0.317554695447415	0.644741069318972\\
0.802962732596977	0.14346648061781\\
0.500669207068472	0.467759261573296\\
0.0832815323587746	0.914182110873513\\
0.0702856314039678	0.588939697634966\\
0.0807619963378481	0.202478312135329\\
0.680291620963377	0.42103018937013\\
0.212366751993907	0.918858708495954\\
0.481385046740532	0.0620000418900469\\
0.217947645744788	0.652314277458912\\
0.436657811245552	0.912793631493021\\
0.80290742809634	0.68743349888764\\
0.218973292240061	0.519085599394687\\
0.931738942824727	0.574165867212414\\
0.788532331001348	0.558043615911041\\
0.0797747659728711	0.803917510607305\\
0.997884066291917	0.438859667380643\\
0.678848503571871	0.937715268758306\\
0.0463089835655519	0.458864076639204\\
0.513314686933298	0.377611743305822\\
0.311573959556869	0.999112538837971\\
0.921682744884128	0.814119134542602\\
0.347174501982462	0.384020773534701\\
0.19719058856529	0.775169912068094\\
0.931532704321012	0.67676100604285\\
0.33511508562334	0.00243966977794319\\
0.81878599179873	0.934829946461723\\
0.861589809466912	0.323615734224654\\
0.0712730550555127	0.682348084620941\\
0.171665610863407	0.401382837649826\\
0.0557736640145967	0.0547694963227018\\
0.485868698852996	0.733904065939731\\
0.946318318800929	0.0597311933550546\\
0.0580462371004229	0.309464241896814\\
0.951323249523167	0.93856488753699\\
0.162427134902741	0.175514756859815\\
0.813047204012472	0.0647115226798226\\
0.827516969193738	0.839318822373074\\
0.642308932200198	0.0833930459742754\\
0.940357919529258	0.180043735454302\\
0.701308217447521	0.823347754417258\\
0.300794764985899	0.234369182920609\\
0.202005352763198	0.000367454700616943\\
0.555614711604801	0.831436257492838\\
0.644583305549602	0.262218709867562\\
0.821963765014195	0.44412266606488\\
0.584553524176316	0.937616292820524\\
0.382018854355247	0.608600397378117\\
0.472580561207945	0.000909423312703717\\
0.947379886486226	0.309856282440177\\
0.450766244988691	0.241720156242742\\
0.497917807063051	0.575856864809341\\
0.931473086775523	0.466928104180751\\
0.753963994529989	0.310219933628605\\
0.359035641838234	0.070576724000773\\
0.542123843670763	0.24679879802406\\
0.406384320679699	0.345219836816892\\
0.660023648383747	0.607316454819497\\
0.847908812484919	0.196391087901339\\
0.283579777108947	0.923500981705825\\
0.620051422386238	0.3635374846841\\
0.0651495440244046	0.997583289258195\\
0.399297511615968	0.845800489031822\\
0.340043620642587	0.718007000007832\\
0.00186963089614323	0.19753276137814\\
0.720978173502085	0.038722079829928\\
0.000117436160280704	0.071845636577671\\
0.00111533310684309	0.585684905928018\\
0.368917382032205	0.523891986525977\\
0.612313515704933	0.489249874479006\\
0.345882368710409	0.160321693834927\\
0.728649678044599	0.702562302207395\\
0.00181502057878569	0.939404863010783\\
0.676944463541967	0.999597304303561\\
0.804572386473276	0.999643482303492\\
0.44478121603241	0.998800794796798\\
0.141729313628747	0.659032937059684\\
0.00272047871039027	0.70794086000537\\
0.499622040665244	0.125612476995636\\
0.152827998016012	0.333345542439308\\
0.99959293803091	0.565418288038091\\
0.188111885528045	0.999683413870681\\
0.930075342049809	0.997933505786784\\
0.281455142118583	0.361145188764598\\
0.20978163051506	0.259092258535869\\
0.0802141241066541	0.522444461048393\\
0.0766607989464986	0.00176898873552334\\
0.000250358343799495	0.447310868581405\\
0.626260208502154	0.717558690342186\\
0.667634066570216	0.156780279150321\\
0.23556952834446	0.589006713337104\\
0.00206822705056608	0.318577353861647\\
0.999062303842151	0.931528049844917\\
0.561412135113395	0.612703230663406\\
0.999625586053072	0.684571586807529\\
0.566271764030753	0.999455200537945\\
0.998771730598153	0.0612440680659287\\
0.147795091294377	0.927505187567352\\
0.723049933229145	0.56271521568011\\
0.292784261878004	0.782756076118635\\
0.995318692706744	0.807353690587848\\
0.628349155971386	0.796437254582859\\
0.999685975915982	0.183303696289052\\
0.418143818138916	0.757404456258771\\
0.00100953863110886	0.829198931481258\\
0.258323741371305	0.0616198650628031\\
0.440771306721295	0.492079041961112\\
0.606700298602444	0.00313392517141409\\
0.5541192310114	0.696447377662944\\
0.736513343687951	0.401428088208405\\
0.236266384508577	0.143285920164239\\
0.999714565538838	0.304722282490204\\
0.274636532645296	0.491094286644701\\
0.445181663694027	0.62484342746657\\
0.563448625414145	0.166321881537378\\
0.938168855026491	0.00292803545005316\\
0.143825709685675	0.0620052344221319\\
0.43818222598452	0.176588277549715\\
0.392942708209197	0.281892273812215\\
0.864709628895431	0.684087741640124\\
0.777054794928802	0.201991277775485\\
0.196564704570952	0.710160883190126\\
0.791800564536593	0.621607954500401\\
0.829808577409769	0.000180445698097609\\
0.844066793450155	0.776946369921882\\
0.192110980082019	0.460779126247261\\
0.323734817276477	0.84368373174952\\
0.0642586934495772	0.743255272456397\\
0.45889762987439	0.393394728057414\\
0.85678569025431	0.558129597076604\\
0.0675663978404042	0.862094155735215\\
0.0546711024067899	0.141736171541389\\
0.675148339083362	0.670767371373807\\
0.681222133804409	0.48059765095056\\
0.208475115670576	0.849015899615442\\
0.26987160339983	0.292131473672957\\
0.880314178798086	0.0630489363566186\\
0.138236911938046	0.790413201310733\\
0.156547988291744	0.5391591244371\\
0.708140418377965	0.270152463774703\\
0.944475007530125	0.747315575832853\\
0.493123503990929	0.945647379301138\\
0.108491253020128	0.437319872399675\\
0.741665074846381	0.942234514381678\\
0.478903002949793	0.863119194497952\\
0.804077210020803	0.501579718324571\\
0.550905353047133	0.48997206284842\\
0.520467473816037	0.310509171470043\\
0.421833973100463	0.0605327157831849\\
0.759457227995966	0.827566934450651\\
0.0694226261330501	0.256186425534299\\
0.378110560806967	0.944322885208286\\
0.944576554306808	0.875702501835868\\
0.522234173128137	0.780308062055049\\
0.953617862438762	0.395803636885102\\
0.76685973486166	0.0884542503676811\\
0.614488314363723	0.572402281163263\\
0.888139659318473	0.946596935790448\\
0.844371041435005	0.384796360117501\\
0.541543693870525	0.0542405794468678\\
0.693843524123314	0.881290499735091\\
0.920255171122297	0.624566924476018\\
0.932092476411773	0.524414751894227\\
0.940809757382976	0.120184236156848\\
0.290119864972355	0.175292233559591\\
0.717603030620545	0.000271190815876454\\
0.271765140282986	0.664201725759879\\
0.739445684003932	0.155536539147491\\
0.565856212649684	0.372084532358976\\
0.305423684659334	0.42459289607413\\
0.584442357348251	0.882107333943607\\
0.48987837922208	0.522133860848332\\
0.8064906751876	0.325584457872912\\
0.880066473492282	0.446268208317692\\
0.616236184729242	0.13647988987175\\
0.3411036096216	0.578965960900295\\
0.0976970671641348	0.631892852214752\\
0.735359920450759	0.763350110651192\\
0.950160395426442	0.242718013760471\\
0.0511049321670323	0.397422110854856\\
0.59531925964105	0.2458819151072\\
0.217001636194895	0.375315328768931\\
0.812242215418778	0.885412996364752\\
0.853417130630944	0.267276740850449\\
0.135562507250794	0.21397725786242\\
0.458456925666178	0.685311585737298\\
0.308373386522334	0.0494375932366498\\
0.535300865796532	0.424892700388412\\
0.632808213524473	0.91910858278739\\
0.363688928293934	0.664323425115996\\
0.380098827503178	0.466868886081846\\
0.155309360860184	0.12523139309174\\
0.633977492134826	0.311704180552028\\
0.106503621637027	0.319244525876141\\
0.597878070063626	0.0518784652636133\\
0.823058819371461	0.731532501608604\\
0.891823539533549	0.171992264701588\\
0.878275868105913	0.836781261020294\\
0.667050103666031	0.215484459826555\\
0.913871673471012	0.342141988630595\\
0.77267883085394	0.439838551014203\\
0.0495416624175775	0.950773731114649\\
0.244027933044533	0.786516799248419\\
0.635531818080032	0.434656803977125\\
0.326379251941649	0.954862358660358\\
0.495405024873268	0.257661724178643\\
0.34588048748588	0.255432226661463\\
0.731515218692939	0.354227738888107\\
0.28018844834748	0.612144689723497\\
0.835705788546628	0.111583400353935\\
0.250924166816749	0.887012837543729\\
0.111440876505682	0.564904435082222\\
0.67680870592299	0.0499567288670227\\
0.199948661610717	0.0454284568274071\\
0.0461309869576505	0.635207158838586\\
0.356361109512533	0.116398610335804\\
0.672643501753082	0.377061920053792\\
0.190641064783511	0.214175530411257\\
0.756976436583057	0.253568741420238\\
0.109500125636614	0.163441815380386\\
0.539823935209775	0.954372738036101\\
0.396260256966039	0.394245800442769\\
0.246152095987156	0.952764702188861\\
0.117045057105054	0.87975292694257\\
0.0399168114697276	0.547489253160103\\
0.3190986523709	0.516646761981902\\
0.107048397170392	0.957025365525322\\
0.10125892157709	0.0512914292843564\\
0.456595772323158	0.111533323544202\\
0.415652022241145	0.803014643276844\\
0.391194021359396	0.169242250523491\\
0.1838752246178	0.299822089652898\\
0.77152311421775	0.722024340675566\\
0.468597501452753	0.440657637368715\\
0.430159690213107	0.955949478399729\\
0.71368448491287	0.51859032961731\\
0.046425582440093	0.0975746251216089\\
0.0466589066291	0.351569426092747\\
0.19384615979716	0.564357324814782\\
0.394959212741649	0.890984025753723\\
0.0367717580627455	0.797546572035361\\
0.186134776263954	0.954862291596798\\
0.757531607512357	0.663380353215487\\
0.38244776067641	0.735210234137076\\
0.965842003510126	0.607555172437244\\
0.511026754415761	0.832231771864331\\
0.902823488228442	0.293264439159659\\
0.184295638232043	0.626587568820451\\
0.507946809520319	0.00103817718269206\\
0.67075883819067	0.793112253040151\\
0.0377369735680816	0.217555742944029\\
0.00103709037172273	0.484818868906398\\
0.301806131950866	0.883161158655649\\
0.364280390487287	0.337328475854401\\
0.55317178352202	0.653112770731133\\
0.596718251274939	0.821822941273753\\
0.258661211611112	0.243156662878255\\
0.238211483762808	0.471285253313191\\
0.302848497132232	0.697149002191072\\
0.256112484586145	0.541134898021629\\
0.175951986375398	0.889721755732888\\
0.780271462312763	0.958470523455184\\
0.527062251738247	0.729708223281745\\
0.557537548888294	0.207446629389583\\
0.773857274715948	0.0432083155045117\\
0.973029317817499	0.478196501774248\\
0.756647273355893	0.592187217563065\\
0.0359106727088235	0.901357315896797\\
0.720683769842967	0.445219232061631\\
0.471543373088317	0.352051268086717\\
0.084206094690332	0.477551796439101\\
0.332012893473236	0.20390084215398\\
0.356404221790615	0.42782137711359\\
0.9604976923822	0.705952177394549\\
0.830029336425993	0.651092764004345\\
0.374116009542241	0.0307419203135092\\
0.138964385178311	0.376375595894603\\
0.639899273469974	0.962247359805802\\
0.180034179825748	0.812835397632991\\
0.88898948675692	0.78572898060093\\
0.515323967239619	0.611830211743649\\
0.244175992090264	0.101004553882784\\
0.435259856242163	0.286723769194943\\
0.89081196265831	0.586163994502461\\
0.196224539827205	0.153115964227628\\
0.954256020548893	0.792493119519606\\
0.612084480522089	0.531526833768471\\
0.11094482893123	0.833593014663794\\
0.84309971229529	0.967107638550538\\
0.324680110820892	0.757955015587181\\
0.438845426133738	0.857611808225667\\
0.757816753076021	0.52972756198902\\
0.153575122362538	0.440537133926617\\
0.541309779694287	0.132518801404218\\
0.695772239909199	0.590627196791036\\
0.345257607531511	0.999562406849052\\
0.852706932989177	0.908437281066399\\
0.168973394953465	0.744031680754975\\
0.709547829469638	0.972141032032137\\
0.820321199673942	0.582686335854976\\
0.100230965659268	0.76113573104772\\
0.462650386419707	0.76596152134327\\
0.669342548079372	0.111172085646411\\
0.408430130369797	0.645586877076102\\
0.613135385715537	0.614482218483977\\
0.627626484822438	0.756269082314073\\
0.816710276478677	0.223429421125071\\
0.112091617401073	0.690047757958887\\
0.0415894685054375	0.00698887801604653\\
0.228613694489369	0.742933420520848\\
0.382921699222378	0.560722148833208\\
0.847123671961932	0.155096100877976\\
0.455118500760769	0.58287175084997\\
0.364230114444258	0.827425104150441\\
0.998699184643709	0.406761449395734\\
0.847832886313042	0.0386790457190999\\
0.177636220086206	0.502026757673543\\
0.95681005581494	0.350507118282944\\
0.890880882362735	0.655113520189273\\
0.420192610904201	0.219624902771535\\
0.318170868797458	0.350478030793557\\
0.671566322572956	0.709321653203607\\
0.931356050847974	0.430088074458652\\
0.964489593713183	0.839692513889051\\
0.51817982533839	0.0897876301599398\\
0.900830452418545	0.210036681467407\\
0.000945323516492458	0.03531047709177\\
0.594181447464781	0.697984574125458\\
0.47411820062575	0.156858169744998\\
0.907381447126692	0.708346015878521\\
0.524598884138337	0.537913769445683\\
0.406248185978639	0.506380707871406\\
0.440890804707979	0.0256158132913933\\
0.80665694714656	0.807912549030139\\
0.17441464404301	0.681800316542986\\
0.702227308463962	0.174532507725375\\
0.0406573750696388	0.709488186098041\\
0.237038918829446	0.69337241440699\\
0.911071674441612	0.0356700798086367\\
0.475615702757179	0.904615128796293\\
0.896034934967738	0.540237819903155\\
0.437768775677029	0.722342482600089\\
0.281068880612343	0.978531896007339\\
0.662394282896471	0.838255227998626\\
0.966704674943021	0.538230038836458\\
0.505643927830094	0.693409979615911\\
0.962773177406408	0.65175257005128\\
0.999195348399362	0.968345956491528\\
0.0292210319419568	0.172679749760062\\
0.0444240792322873	0.505265259478242\\
0.974164124621297	0.0930242027742446\\
0.92179960973552	0.916257599500173\\
0.544433400876983	0.339754077683643\\
0.208179173175588	0.421546225885154\\
0.162969723388317	0.0238349398778173\\
0.82750702251509	0.535063442431919\\
0.102286109754894	0.275062794586376\\
0.0328441731510739	0.276358204943855\\
0.308472766075252	0.281791280878153\\
0.637636635235666	0.031854241103333\\
0.724220317905597	0.85558134465093\\
0.305299144000395	0.00126418701985709\\
0.76356387120715	0.126415407831906\\
0.783754185696395	0.859503389830675\\
0.183155966174566	0.352787963977442\\
0.471276452287408	0.206763919642928\\
0.969714051457664	0.151843874105752\\
0.615255343098893	0.398423216947773\\
0.672731096292038	0.295827666558502\\
0.0375231678903707	0.842731858325891\\
0.919149192818193	0.0874346962354655\\
0.977139915161084	0.897348936646719\\
0.779850626099767	0.914985530583297\\
0.235142480355746	0.0296491600008001\\
0.275271300189538	0.396248110262916\\
0.837061109832825	0.478755967575626\\
0.688465703702846	0.636764394987091\\
0.726879105054603	0.0746784581977076\\
0.713629136883373	0.916715979455755\\
0.970536168777792	0.0251991384490422\\
0.972612947505003	0.278031724379304\\
0.573115764289205	0.462827453510287\\
0.343848910454647	0.619830257085556\\
0.286400022136657	0.4575132837506\\
0.968432834490611	0.203947687726643\\
0.541520670370021	0.867042422712632\\
0.563992960072798	0.793758916828031\\
0.262937932262766	0.330289077188614\\
0.647128676115068	0.489784086253384\\
0.170706706193161	0.248014187699864\\
0.121681254024657	0.52286511172705\\
0.00127480125289803	0.415250579280591\\
0.874839140189358	0.359044532990609\\
0.321543204040408	0.809235859843788\\
0.542510937495875	0.580312996953477\\
0.734081356737421	0.800390449203202\\
0.960233769323874	0.977304852751262\\
0.49956089452982	0.416517020716283\\
0.903816201472364	0.490446989148041\\
0.854469553395804	0.419690006527736\\
0.126072177590821	0.612042850404913\\
0.608051776470887	0.283696398228301\\
0.60923808510263	0.0968690329901211\\
0.395560103862148	0.086434315172963\\
0.0720811062154245	0.42957799766053\\
0.825890767255645	0.29529531882833\\
0.456990591400942	0.532243527006148\\
0.911517143683532	0.858108498062984\\
0.274150240658807	0.142033949519437\\
0.59311558927553	0.337515303894339\\
0.323956015730617	0.0881649377323284\\
0.151560452561946	0.973005771930725\\
0.24404101352851	0.628948440705863\\
};
\end{axis}
\end{tikzpicture}%

%% file: Figures2/mat1_ctrs2.tex
%
%
\definecolor{mycolor1}{rgb}{1.00000,0.00000,1.00000}%
\begin{tikzpicture}

\begin{axis}[%
width=0.951\fwidth,
height=0.951\fwidth,
at={(0\fwidth,0\fwidth)},
scale only axis,
xmin=0,
xmax=1,
ymin=0,
ymax=1,
axis background/.style={fill=white},
axis x line*=bottom,
axis y line*=left
]
\addplot [color=mycolor1, draw=none, mark size=0.3pt, mark=*, mark options={solid, mycolor1}, forget plot]
  table[row sep=crcr]{%
0.903983461642628	0.694629043325264\\
0.00351992456462402	0.00379331012736783\\
0.000713876950142911	0.998165941569912\\
0.996607118333609	0.000484940148920643\\
0.999751571728732	0.998909440581492\\
0.356238295959646	0.443564958043876\\
0.500182243367562	0.998976553161055\\
0.516516041791371	0.000342885613561328\\
0.000173595557909745	0.548192311796471\\
0.999068430922368	0.339673311728269\\
0.694445949784035	0.307849243856784\\
0.254658537340227	0.784317083771163\\
0.19509147818975	0.192540610311149\\
0.597007339122287	0.69954235003579\\
0.771758035970324	0.999705788452457\\
0.000931298430831728	0.27317431516175\\
0.762744789362053	0.00177259538717878\\
0.230689790884205	0.999426398344241\\
0.260471658156072	0.00168767610725518\\
0.000836682757433205	0.780183533132439\\
0.462383565245924	0.217489266086195\\
0.779823358006518	0.51285298275736\\
0.999873206639899	0.56625292464026\\
0.885598837432774	0.161452148449243\\
0.183926926901849	0.554684134351027\\
0.998871669924614	0.837234335423091\\
0.552804292218659	0.490746323099473\\
0.405490147735502	0.664838026174323\\
0.736082938089169	0.834328034106947\\
0.444401750555452	0.849820991966569\\
0.145771651967286	0.380530915902714\\
0.644436184733753	0.129576065156245\\
0.999143964955769	0.156416816871865\\
0.110365451936159	0.892035246099378\\
0.856037522329175	0.368611778366202\\
0.0752266151089651	0.108591581321819\\
0.366741123876393	0.0980115782677201\\
0.612236492288422	0.912834999090739\\
0.11044736329754	0.69211479561773\\
0.886914012414093	0.912511058712416\\
0.315598405021774	0.291725248946798\\
0.524682419234118	0.347653908881911\\
0.744999179119218	0.66169901328827\\
0.00110634171156798	0.411340803308687\\
0.363561144345525	0.998133059927412\\
0.127702760975506	0.000388491323495166\\
0.886370095001962	0.00175773736154916\\
0.917593465217201	0.487662007381921\\
0.284659031568255	0.629367307567212\\
0.629987100036627	6.03984097952148e-05\\
0.673256393370241	0.439057864742255\\
0.999564695159691	0.701113144929624\\
0.641564599964665	0.999526064283819\\
0.763214461253658	0.181499172699129\\
0.39860521581625	2.27987227616744e-05\\
0.298282107627448	0.898886673744146\\
0.116428454441942	0.999887831029692\\
0.435981485596233	0.548408925724702\\
0.636815772900176	0.579394716392919\\
0.000647247808942764	0.667894095512699\\
0.0010490530538424	0.143384683002758\\
0.514958557548591	0.108804949998234\\
0.000644843544219009	0.894291855764683\\
0.588673796643993	0.241952912807036\\
0.8973586463289	0.999405435927705\\
0.245089638707102	0.446455670200944\\
0.999690323116915	0.448602541288243\\
0.105670683529319	0.261058527415316\\
0.850219588188222	0.792440199786543\\
0.546301841923957	0.798189462411745\\
0.91118384655446	0.264314942853102\\
0.0902077452833491	0.48561930200342\\
0.365155704127036	0.771087814477356\\
0.218915384353767	0.0897167452030772\\
0.51966835758487	0.625335329966285\\
0.416588595400912	0.346675189151189\\
0.856062717112652	0.596454205894579\\
0.805992828148716	0.0814654022659985\\
0.319239495618291	0.186664779360092\\
0.138436676730697	0.794074419862774\\
0.666132449698696	0.772649127495786\\
0.951584450398905	0.0761202357913132\\
0.315147546056578	0.536213292087712\\
0.214466163241217	0.307042013687114\\
0.79842499474832	0.28075269595794\\
0.766030070975044	0.40284644809908\\
0.0799103106142215	0.594445267631647\\
0.474097862084233	0.733930177589293\\
0.464715997680308	0.43820664264157\\
0.203094299020775	0.895196827015424\\
0.721734063424086	0.928422002229816\\
0.20597746909848	0.688579955891694\\
0.998839006894363	0.929312900953135\\
0.99870054426097	0.245264923998975\\
0.428639197191293	0.937384255771143\\
0.613059632487169	0.364899106069133\\
0.7045607989254	0.0659892120489053\\
0.783853740261411	0.741779841030619\\
0.0573968611311909	0.348482410582684\\
0.519189357157408	0.908430275668143\\
0.940677569031177	0.783703200366189\\
0.804565459971446	0.903191263943715\\
0.677924191799537	0.215715200429729\\
0.290144909263407	0.372452106434859\\
0.0635829075668507	0.193246914762121\\
0.720009261977344	0.565979176285241\\
0.946710891416686	0.619045699730507\\
0.937423157015444	0.393040931100791\\
0.446951088586558	0.0685450730721856\\
0.583229007433201	0.0605871095557581\\
0.308175332035866	0.709974294248832\\
0.385916990715906	0.249962247393005\\
0.0600579224614193	0.822576018825144\\
3.94424640670632e-05	0.0732014117650198\\
0.054834011704485	0.952019196978098\\
0.844487141975471	0.455018885524513\\
0.670974519893558	0.66594509281971\\
0.417210775019852	0.156238829828238\\
0.363777313515745	0.861237288463063\\
0.301118548265015	0.0596326198773051\\
0.632675015486513	0.842900413156349\\
0.561540049640195	0.171221185127254\\
0.168972677990901	0.469747067653214\\
0.573976253218784	0.998902795667943\\
0.14859847596438	0.119113094859141\\
0.366252507581565	0.598765471790311\\
0.640906647529599	0.506072831663937\\
0.826015827210684	0.668863265423287\\
0.560278973764638	0.421365973309227\\
0.0629980683993554	0.0269590698043359\\
0.0499008874954073	0.734394307018543\\
0.998861006586265	0.0660779187723384\\
0.501584400874884	0.278528155835568\\
0.158236114543181	0.627514256691636\\
0.934007022165694	0.859621885019053\\
0.840559243534507	0.219039284333445\\
0.00427547432676412	0.474673377562097\\
0.303844583437649	0.999922363778849\\
0.260128623626398	0.237357611393272\\
0.551185257306449	0.563024537784488\\
0.330954577737527	0.00185573385683491\\
0.165445101921799	0.956590019921809\\
0.259239633340461	0.147741974864337\\
0.875511870273651	0.0741333730816611\\
0.0795102705449515	0.420163319523917\\
0.191537777122774	0.00017293311512645\\
0.834158171228741	0.974375317400937\\
0.244908049933556	0.529939076840504\\
0.786761853562977	0.592056953199019\\
0.941991529973377	0.196615366965119\\
0.399021433216827	0.493660961441216\\
0.99919099467355	0.631329502156195\\
0.68691324642204	0.000354740314948199\\
0.432861617542783	0.99826571057016\\
0.944422771875769	0.956454108073583\\
0.859267386996754	0.526762540888412\\
0.0015238589876776	0.341974880074618\\
0.300495205784358	0.83086579735718\\
0.707273505384521	0.999923966294981\\
0.719884083203057	0.131488386137702\\
0.144583228210061	0.0556431998159249\\
0.828734747273048	0.000276126761305839\\
0.216576455892506	0.381737283432878\\
0.717282232844498	0.483789586989995\\
0.528744861597627	0.701787209835637\\
0.198355153894113	0.82423845071369\\
0.58977858947104	0.631825358260851\\
0.000510793794711972	0.209227700108788\\
0.702131419419006	0.370685651522814\\
0.932080571298807	0.324384039004136\\
0.486063865702739	0.505963585475074\\
0.997665119091303	0.777625199424479\\
0.62351124976737	0.297348175203815\\
0.718211587140325	0.730369558307291\\
0.354506347657207	0.360972721007995\\
0.450275073458819	0.620515113999881\\
0.352384979774985	0.935302950620049\\
0.932276284504012	0.552801803315328\\
0.147550101073389	0.310207946814781\\
0.77299133416123	0.340314192242476\\
0.178790507444434	0.74726005898115\\
0.739427381987562	0.248713968540859\\
0.803686410965016	0.833817102589154\\
0.130134698112755	0.192608555471077\\
0.431849538038393	0.788127991629875\\
0.00383020572832815	0.609946702951229\\
0.255440905050473	0.945897794179419\\
0.226296346287864	0.604996571847706\\
0.566234096688133	0.861629667798643\\
0.816868144138703	0.147280542472696\\
0.462144193851068	7.21731953057247e-06\\
0.996519509936893	0.505812049272487\\
0.599966449974671	0.771392030527322\\
0.0568776758865239	0.656497609059253\\
0.858291173948791	0.304832587900698\\
0.948990851689448	0.00355099031859296\\
0.489270815785357	0.165355770104731\\
0.678056117860485	0.886290672891999\\
0.440439755816306	0.291262769913803\\
0.178102909242635	0.249033183011925\\
0.120820671423638	0.543612095012114\\
0.871193095145035	0.852299142968874\\
0.302748007065065	0.470406142900816\\
0.0519793816597705	0.534586630059874\\
0.0450105222168479	0.999828493896942\\
0.400859907138811	0.723685890832214\\
0.855017693731967	0.729977592883453\\
0.47277942600624	0.379116322844285\\
0.643981363440447	0.0604728398774259\\
0.497967487010349	0.829385938544536\\
0.408799467632442	0.413280008714775\\
0.342068320081367	0.658406127930017\\
0.612341006184378	0.447494312847866\\
0.00380238067578043	0.847328167122988\\
0.564162356270525	0.302911292528371\\
0.555214669211773	0.951096370599767\\
0.665603141671535	0.951956148689526\\
0.04998734052245	0.280571750288696\\
0.942502064346082	0.13093729986809\\
0.575823517604322	0.000323467199637051\\
0.315616964997199	0.123193350709293\\
0.530285089557389	0.221117606722471\\
0.511179724024634	0.0505937369124845\\
0.470849032490572	0.676027681580494\\
0.998147990158429	0.393395925118293\\
0.625784376231641	0.185844908106217\\
0.0519675884123285	0.885542800993602\\
0.000231221292928074	0.720037064515899\\
0.252991863549669	0.720737244957315\\
0.95069759568225	0.72088038571309\\
0.698332250033386	0.614114985987612\\
0.250702823881684	0.858169312651939\\
0.378476229259268	0.0450629893854663\\
0.492947678198321	0.569016680509133\\
0.181225672158275	0.999564821976116\\
0.894124395022679	0.42551970919704\\
0.58521205194326	0.118768147430553\\
0.375849957258145	0.188546750377525\\
0.743157161380882	0.782412374925448\\
0.77625531104266	0.455042795771636\\
0.105567616167039	0.751642885860268\\
0.148361610344129	0.852252984528798\\
0.266736338168758	0.319252981798956\\
0.480141050593831	0.951639402366525\\
0.759615743582115	0.0557254027924488\\
0.890557931654984	0.641375231507697\\
0.00446270690917561	0.954587085391398\\
0.951069513193771	0.99898107056214\\
0.967996285499506	0.283194250491139\\
0.773401257723087	0.952346727129981\\
0.371195994129655	0.536538816898876\\
0.653483778289615	0.720806228111621\\
0.373781170909889	0.30414880701801\\
0.310155130936978	0.76636151823334\\
0.228602292994712	0.0406915518410532\\
0.281105097117062	0.57718237549397\\
0.329519261076858	0.238043403936297\\
0.95566598754038	0.449130349871878\\
0.597622434588154	0.532761044943593\\
0.133213295458371	0.432493379898707\\
0.817317503807003	0.401166168457535\\
0.965146762185646	0.895241602487024\\
0.75202897377073	0.884680326436958\\
0.534206265054923	0.750667409971841\\
0.767475075021457	0.113372083701068\\
0.659310901697006	0.261182090620988\\
0.513666848120155	0.438911709139465\\
0.102236778232097	0.334535567323712\\
0.683359532654022	0.82409655188402\\
0.958642743905439	0.669061919484508\\
0.464191961761083	0.897824183968117\\
0.302501062187557	0.419435721707697\\
0.454689214071929	0.118809266960584\\
0.896261095446216	0.761091819184558\\
0.108979509169283	0.946678633453206\\
0.198793639427205	0.142293174837834\\
0.401726142700089	0.891182051637083\\
0.0428226575657118	0.455077170461679\\
0.680467044911571	0.537537795846829\\
0.252032194922787	0.66861419394768\\
0.193399366566235	0.425206632558141\\
0.0774947120464097	0.000448489922788342\\
0.388911284160257	0.820966624700742\\
0.653584930584146	0.391203600127461\\
0.0991316264561199	0.150739980622214\\
0.788587088083313	0.225390912600848\\
0.997799571529888	0.115397810693755\\
0.108600029633342	0.637440339841187\\
0.214517333721559	0.49145048952371\\
0.837006825286753	0.0399082016320255\\
0.475699275357323	0.326755943339669\\
0.723951075114522	0.42534279974179\\
0.156140056814081	0.688895443633601\\
0.608658600851492	0.962854791927428\\
0.685366708255053	0.167751547654768\\
0.89340022188519	0.216068688648178\\
0.815788380322759	0.556345656860084\\
0.0960864039996209	0.0637289751195677\\
0.914157824044906	0.045918416529574\\
0.306323092917048	0.955663273459746\\
0.746425992181195	0.301708114587059\\
0.638803743273197	0.63294793009054\\
0.490334705181769	0.782923165711594\\
0.781090778601792	0.694056479402969\\
0.267884261439538	0.0976604357612817\\
0.354485554681638	0.713285907411784\\
0.041171597166764	0.145462751688045\\
0.102901585881212	0.833094290738552\\
0.859608629849779	0.116741597501552\\
0.904031142296992	0.814408911616372\\
0.0397919970721909	0.238094832535759\\
0.885997928426334	0.958085325984749\\
0.65904711093319	0.339379923116347\\
0.819516803156367	0.334635976866127\\
0.894714491103199	0.579405187115819\\
0.249518880099869	0.192617940527958\\
0.410142713563891	0.593155348650916\\
0.0387592368420033	0.0772219007761645\\
0.138848645775396	0.58335216075972\\
0.790050765415438	0.636071319799635\\
0.79973398428648	0.787672500754212\\
0.999385289119881	0.292972594468917\\
0.854427573675919	0.999492478826138\\
0.56201418867334	0.373065622363703\\
0.346068856137716	0.493159032671089\\
0.183229605243131	0.345626320865706\\
0.367158461000831	0.144764112382035\\
0.155008813663204	0.903958567111273\\
0.0353929385969785	0.391607987134763\\
0.441857280271271	0.479515297714534\\
0.0623131165382108	0.779295997031044\\
0.20906755802437	0.95322374862995\\
0.996182984716279	0.201350304869432\\
0.410239274755487	0.0975716981281559\\
0.998072721185625	0.879645923100852\\
0.151961677995797	0.509851078325107\\
0.899586870536693	0.356832432041147\\
0.551829682938923	0.66365223462477\\
0.394116077296721	0.963477680993405\\
0.821298722286147	0.497237456206505\\
0.852635906831626	0.261492513944035\\
0.591985151367839	0.81780459880666\\
0.418526110680621	0.213695769736328\\
0.0356341480801988	0.581005046264728\\
0.224328822706776	0.257265267884055\\
0.958064789980732	0.51191019606529\\
0.32166945653533	0.592027249621198\\
0.845235824241076	0.921224039305354\\
0.212721470766837	0.771935708997495\\
0.719898087714337	0.199332259893725\\
0.742063048341893	0.609102840299074\\
0.343521960046562	0.810816627415827\\
0.590481794930519	0.584631074529847\\
0.908434068500923	0.103600158987346\\
0.679521709866399	0.103697085049709\\
0.204151816029325	0.644769750877668\\
0.428942170733812	0.0318064279569563\\
0.100622954945656	0.379846581556229\\
0.720303015315669	0.0239229917710726\\
0.515161229829718	0.396522154314377\\
0.964608772036144	0.821516180520737\\
0.321427128422534	0.336120988170053\\
0.185099672397776	0.0487473774172725\\
0.268133305520462	0.498299443581632\\
0.951880661605431	0.236926541763824\\
0.964224434274062	0.359051836166085\\
0.737476502049164	0.524616640767519\\
0.567053545576436	0.907033474217429\\
0.979561278525495	0.0291912964309967\\
0.0363450332077453	0.000740039821535032\\
0.525638055616966	0.529524872845257\\
0.348086137501539	0.402086768467859\\
0.546258984465262	0.261775013996143\\
0.2524621884849	0.902237171520118\\
0.548298763321872	0.0818102077228264\\
0.73534701861604	0.973232337716181\\
0.517614423003768	0.86571313595101\\
0.256548672732793	0.40064593684795\\
0.711576614858831	0.68615871216776\\
0.835645800215181	0.871334073874364\\
0.0354755337125208	0.693721914234893\\
0.793095770054638	0.0290860123835395\\
0.574358761609383	0.730614428608679\\
0.275445856531659	0.275427584288459\\
0.0947958121719454	0.220945077537908\\
0.605705280654378	0.405113446185688\\
0.0130529408600961	0.0385073927800946\\
0.872541841515696	0.484754451145222\\
0.434381040788775	0.698365354147062\\
0.965429314380989	0.578372712129101\\
0.551002271180047	0.0301803377680802\\
0.0313812369286238	0.503934623094374\\
0.810991798391869	0.999406027740348\\
0.629298992521858	0.227477128414342\\
0.157892618661144	0.159460085963715\\
0.595402938088771	0.484860856975394\\
0.990249895035542	0.969475888056779\\
0.995215853569103	0.740570584019864\\
0.533604761458924	0.14646091108667\\
0.636888459221293	0.800128056946596\\
0.84557261580422	0.177515335180944\\
0.673845965330673	0.480387427043705\\
0.457098334541413	0.256536639930427\\
0.338192526896853	0.0502626087778089\\
0.667989066835686	0.0300782623456515\\
0.701541344416836	0.263657622447129\\
0.404143120221118	0.452981891402957\\
0.0247544152996244	0.312631402898686\\
0.863385706528431	0.680004911681414\\
0.291472969909145	0.0220868773927776\\
0.139518679866918	0.236596923773103\\
0.299896831864032	0.668507343212427\\
0.244158638085958	0.350202621634315\\
0.764064528776636	0.555811029216973\\
0.627268446420041	0.671845600084137\\
0.340136068988761	0.894959756029468\\
0.926764149815307	0.915047275498465\\
0.726789028089367	0.342290118093369\\
0.534239165156351	0.997657074123476\\
0.838709159537898	0.631123919186398\\
0.136740669339913	0.729946517330912\\
0.714401247329201	0.866332564584279\\
0.816678424084381	0.718061115746059\\
0.619789878281295	0.0966874449468758\\
0.899607868927415	0.529223340766017\\
0.183991180955343	0.597860700576302\\
0.0822471446498815	0.983222446161657\\
0.434317954393693	0.747440378912171\\
0.175972856845896	0.0915213586787215\\
0.706296999574299	0.786075794395381\\
0.610934944612015	0.0285688745407237\\
0.270613434853636	0.998139363779127\\
0.55090691142858	0.604995375869959\\
0.0396343351098334	0.622814049129825\\
0.591537813246293	0.331226692828241\\
0.475573381031699	0.034830072907836\\
0.291889725180073	0.214734957500898\\
0.582216452084114	0.204595731648239\\
0.383505841552936	0.634056912050912\\
0.022274745954288	0.812393486209725\\
0.858821474113976	0.409044381710669\\
0.239957828940725	0.819646841301671\\
0.612086906741428	0.873932681512107\\
0.0828691529968072	0.302702291837713\\
0.23125975689793	0.568332425309306\\
0.022581142728945	0.181481694074295\\
0.898227273412754	0.310461456900185\\
0.128000361010677	0.475715770870098\\
0.449949328023577	0.176939330692966\\
0.0757583308638256	0.917888211371365\\
0.806915227687725	0.185927014927396\\
0.507933848963016	0.476600538214706\\
0.117824030680609	0.0999667276923769\\
0.621497223593885	0.743279247864884\\
0.677072764869851	0.580449994689875\\
0.60318879453967	0.154943662639455\\
0.024305577199099	0.920656535245416\\
0.963784757522866	0.165193806335035\\
0.075673865958756	0.705703865357826\\
0.214363935519911	0.732188519709763\\
0.514597944211801	0.95956652519436\\
0.387256962484709	0.382971441779268\\
0.0188510580231043	0.110847149067529\\
0.437200268324208	0.389932124276798\\
0.734215731364429	0.0909575839999922\\
0.650024005787154	0.914107772716938\\
0.318822942329897	0.861276141497659\\
0.182458644785442	0.289716199346113\\
0.505330537894762	0.660079742801544\\
0.454729784973594	0.58163797142165\\
0.484265402971658	0.0812855680999651\\
0.192916278179553	0.861749887509439\\
0.82817576282148	0.757273104899292\\
0.00113308643370469	0.511416262354048\\
0.775224648900282	0.858868788186904\\
0.49668315048266	0.236657627200636\\
0.26351403837873	0.0557434550914211\\
0.291692953896054	0.168482664258702\\
0.999708479307565	0.663626254684564\\
0.796922163410258	0.368741977790326\\
0.969577476744293	0.415432640025607\\
0.52596020342589	0.309532807613061\\
0.000296674729646118	0.372419427099355\\
0.340654562284378	0.977232595407873\\
0.458032054460713	0.812822423024176\\
0.749306243359841	0.71128534622381\\
0.176285049249117	0.790019542793893\\
0.087366991300767	0.553274854207115\\
0.107540912167279	0.0273962388888399\\
0.354573074410031	0.273668228802976\\
0.740020936293035	0.378033053517225\\
0.276048404015225	0.749641208005673\\
0.408397598629016	0.276597756393924\\
0.486221469100697	0.607970882022777\\
0.808657701047786	0.440607731895836\\
0.0998194880146143	0.790528911316191\\
0.5182364078207	0.186414319221408\\
0.900408915568707	0.87555015417509\\
0.462408567014461	0.981352442985182\\
0.160890716929927	0.0197472448221298\\
};
\end{axis}
\end{tikzpicture}%

%% file: Figures2/monotonicity_ctrs2_gaussian.tex
%
%
\begin{tikzpicture}

\begin{axis}[%
width=0.951\fwidth,
height=0.951\fwidth,
at={(0\fwidth,0\fwidth)},
scale only axis,
xmin=0,
xmax=1,
ymin=0,
ymax=1,
axis background/.style={fill=white},
axis x line*=bottom,
axis y line*=left
]
\addplot [color=blue, draw=none, mark size=0.3pt, mark=*, mark options={solid, blue}, forget plot]
  table[row sep=crcr]{%
0.301660002172934	0.958036509686258\\
0.999999999981787	1.62881874345478e-05\\
0.999068717089444	0.998710070095816\\
0.849857847259978	0.527854582793018\\
1.5261139045844e-05	0.999999999905441\\
0.999960723444986	0.491208843370857\\
0.669565287151331	0.999875050074936\\
0.972577362667315	0.232843398567893\\
0.581000597195964	0.814009367840534\\
0.999787457261159	0.80916791294833\\
0.825120006675552	0.822942223589454\\
0.365495734129728	0.999952616664732\\
0.136594751698536	0.999983058634327\\
0.999992019916987	0.170747819883965\\
0.869810034887475	0.999860534231714\\
0.921338266468856	0.643432575016856\\
0.551849862392558	0.91719373855688\\
0.966040696690268	0.385823369956581\\
0.746369125087001	0.728868217200214\\
0.962332752935254	0.925955062386249\\
0.126449123184255	0.991974227433456\\
0.997525201100167	0.0706900520329418\\
0.734244217187931	0.678946869792068\\
0.741122909732822	0.939950906363187\\
0.999881655819483	0.660392608517091\\
0.429981986627813	0.903066510817748\\
0.999634833664539	0.950147096288402\\
0.52620808531736	0.999835712641843\\
0.0502605958559874	0.999992471246029\\
0.936570625420249	0.350708222206338\\
0.402118367829218	0.961823995601125\\
0.999897356253857	0.312111473342074\\
0.939849556678059	0.999841783756365\\
0.849636996329999	0.618059559109061\\
0.942130574336036	0.774819218729476\\
0.999997200419079	0.0507152349190861\\
};
\end{axis}

\begin{axis}[%
width=1.227\fwidth,
height=0.92\fwidth,
at={(-0.16\fwidth,-0.101\fwidth)},
scale only axis,
xmin=0,
xmax=1,
ymin=0,
ymax=1,
axis line style={draw=none},
ticks=none,
axis x line*=bottom,
axis y line*=left
]
\end{axis}
\end{tikzpicture}%

%% file: Figures2/monotonicity_ctrs1_gaussian.tex
%
%
\definecolor{mycolor1}{rgb}{1.00000,0.00000,1.00000}%
\begin{tikzpicture}

\begin{axis}[%
width=0.951\fwidth,
height=0.951\fwidth,
at={(0\fwidth,0\fwidth)},
scale only axis,
xmin=0,
xmax=1,
ymin=0,
ymax=1,
axis background/.style={fill=white},
axis x line*=bottom,
axis y line*=left
]
\addplot [color=mycolor1, draw=none, mark size=0.3pt, mark=*, mark options={solid, mycolor1}, forget plot]
  table[row sep=crcr]{%
0.958606569284845	0.802864898314709\\
0.00078890937357412	0.00629405149532403\\
0.00266625289758238	0.998601186477411\\
0.998081641788327	4.13755520894954e-05\\
0.443469955065914	0.445635920576325\\
0.547726327910299	0.999487392781013\\
0.51895015771771	0.000107936336919301\\
0.00137557092626472	0.503291394599683\\
0.998919278859817	0.999375300976723\\
0.997443187027926	0.324832205722367\\
0.220794975403636	0.79847547441334\\
0.202454468290183	0.177925858986405\\
0.775161524042295	0.152551462707617\\
0.245855703369031	0.99917516859096\\
0.776631082691697	0.576205674156734\\
0.000626675845690428	0.200683615742686\\
0.825511330196571	0.999213007913655\\
0.815782740225094	5.87660342421525e-06\\
0.00111378402652507	0.818831723058411\\
0.616850735685013	0.811928887118456\\
0.200258578129063	0.000891650462247706\\
0.999628793422505	0.577041985751069\\
0.999877712542627	0.123687630254002\\
0.110478119057229	0.385482999023321\\
0.0916372664409444	0.659042013269988\\
0.0955074911024187	0.934732735502113\\
0.626657884063177	0.324753451768397\\
0.0797947942717375	0.0653399722488788\\
0.999862475302787	0.880354852402363\\
0.890767773643987	0.314757387915598\\
0.761995382439941	0.908514255021298\\
0.434438581859014	0.0984605970987465\\
0.411070530293437	0.925980503599976\\
0.0796305474799995	0.000670815118030132\\
0.923288626234567	0.0533753096050714\\
0.93386047992158	0.946460916824434\\
0.100784059912293	0.999818296875398\\
0.239139559228316	0.545658364359959\\
0.444651194501692	0.651893684375142\\
0.660290819326822	0.0459688388622215\\
0.463666550763128	0.239357166181581\\
0.386321019419994	0.99914416297683\\
0.951784863019497	0.197437260947459\\
0.000553858352565162	0.664273078790187\\
0.927015811168663	0.553741637976951\\
6.493148702158e-05	0.0666838921653704\\
8.4212271623807e-06	0.939320274302818\\
0.832706774283575	0.764466963535734\\
0.927070041022267	0.000517198767625127\\
0.306955024783545	0.329533232339631\\
0.99939071589692	0.716127915608555\\
0.0463310421554083	0.165757219825288\\
0.622312769950298	0.628428053660371\\
0.941800114152583	0.99849305541729\\
0.695332914798672	0.999834277115764\\
0.687548665191918	7.71370382777192e-05\\
0.000682610394365235	0.344846540397056\\
};
\end{axis}

\begin{axis}[%
width=1.227\fwidth,
height=0.92\fwidth,
at={(-0.16\fwidth,-0.101\fwidth)},
scale only axis,
xmin=0,
xmax=1,
ymin=0,
ymax=1,
axis line style={draw=none},
ticks=none,
axis x line*=bottom,
axis y line*=left
]
\end{axis}
\end{tikzpicture}%

%% file: Figures2/monotonicity_conv_rates_gaussian.tex
%
%
\definecolor{mycolor1}{rgb}{1.00000,0.00000,1.00000}%
\begin{tikzpicture}

\begin{axis}[%
width=0.951\fwidth,
height=0.75\fwidth,
at={(0\fwidth,0\fwidth)},
scale only axis,
xmin=0,
xmax=60,
ymode=log,
ymin=1e-06,
ymax=1,
yminorticks=true,
axis background/.style={fill=white},
axis x line*=bottom,
axis y line*=left,
legend style={legend cell align=left, align=left, draw=white!15!black}
]
\addplot [color=blue]
  table[row sep=crcr]{%
1	1\\
2	0.924942101178921\\
3	0.809567963137884\\
4	0.348894612274186\\
5	0.275258859829532\\
6	0.13288668490263\\
7	0.10106167885387\\
8	0.0695555123998578\\
9	0.0254050292811703\\
10	0.0175153649740007\\
11	0.0138953193708314\\
12	0.00974835181468541\\
13	0.00635994508583122\\
14	0.0046423806863557\\
15	0.00322220397238894\\
16	0.0018117017153817\\
17	0.00173458514203093\\
18	0.00120995872751519\\
19	0.000692567164484252\\
20	0.000398211908310666\\
21	0.000283932426041077\\
22	0.000245151716902607\\
23	0.000179482703385516\\
24	0.000142718752824402\\
25	8.48434762581076e-05\\
26	7.36664461704768e-05\\
27	5.09117084171781e-05\\
28	4.20877641448112e-05\\
29	2.66313913199689e-05\\
30	2.22055620476803e-05\\
31	1.48591006784849e-05\\
32	1.16283223863817e-05\\
33	9.01772545356824e-06\\
34	7.03391762699317e-06\\
};
\addlegendentry{$\tilde{\Omega}$}

\addplot [color=mycolor1]
  table[row sep=crcr]{%
1	1\\
2	0.92422245291835\\
3	0.855175009482267\\
4	0.707324503637663\\
5	0.50833044129828\\
6	0.30858593511368\\
7	0.288081837273681\\
8	0.274825890895453\\
9	0.112037658107928\\
10	0.0966776303393007\\
11	0.0747810408904472\\
12	0.0546040318855081\\
13	0.0358013467444573\\
14	0.0233860439850596\\
15	0.0222202000870682\\
16	0.0216423617544871\\
17	0.0165310351153102\\
18	0.0143173253530174\\
19	0.0108604816567428\\
20	0.00634519171668204\\
21	0.00443530865983015\\
22	0.00408667533523133\\
23	0.00295819321030714\\
24	0.00245452612696749\\
25	0.00191481100711991\\
26	0.00122526905277201\\
27	0.00101244799679299\\
28	0.000937303803432686\\
29	0.000687043165060538\\
30	0.000566182425729883\\
31	0.000473499526617585\\
32	0.000414220631948504\\
33	0.000296271482791938\\
34	0.000251300138832125\\
35	0.000222803169510639\\
36	0.000152588976552207\\
37	0.000132672990057775\\
38	0.000124516300799887\\
39	0.000110026574962806\\
40	9.94975470711253e-05\\
41	8.11248088773066e-05\\
42	6.37308239463111e-05\\
43	3.55105891951159e-05\\
44	2.99474707454317e-05\\
45	2.65986832356293e-05\\
46	2.38201753573834e-05\\
47	2.01944499001946e-05\\
48	1.60594764003811e-05\\
49	1.33939047348481e-05\\
50	1.25090487051269e-05\\
51	1.08499299436104e-05\\
52	7.16442192894813e-06\\
53	6.93697412995289e-06\\
54	6.75496143445211e-06\\
55	4.79274851483598e-06\\
56	4.47231689905235e-06\\
};
\addlegendentry{$\Omega$}

\end{axis}

\begin{axis}[%
width=1.227\fwidth,
height=0.92\fwidth,
at={(-0.16\fwidth,-0.101\fwidth)},
scale only axis,
xmin=0,
xmax=1,
ymin=0,
ymax=1,
axis line style={draw=none},
ticks=none,
axis x line*=bottom,
axis y line*=left
]
\end{axis}
\end{tikzpicture}%